\documentclass[12pt]{article}
\usepackage{graphicx} % Required for inserting images

\usepackage{amsfonts}
\usepackage{amsthm}
\usepackage{amsmath}
\usepackage{amssymb}

\usepackage{xcolor}

\usepackage{cite}

\usepackage{booktabs}
\usepackage{tabularx}
\usepackage{tikz}
\usetikzlibrary{arrows.meta,positioning}
\usepackage{enumitem}

\newcommand{\Bb}{\mathbb{B}}
\newcommand{\Zz}{\mathbb{Z}}
\newcommand{\Rr}{\mathbb{R}}
\newcommand{\Aop}{\mathfrak{A}}
\newcommand{\Xop}{\mathfrak{X}}
\newcommand{\DUC}{\mathfrak{D}}

\DeclareMathOperator{\dist}{dist}
\DeclareMathOperator{\Cl}{Cl}
\DeclareMathOperator{\supp}{supp}

\newtheorem{thm}{Theorem}
\newtheorem{lem}[thm]{Lemma}
\newtheorem{conj}[thm]{Conjecture}
\newtheorem{quest}[thm]{Question}
\newtheorem{statement}[thm]{Proposition}

\theoremstyle{definition}
\newtheorem{defin}[thm]{Definition}
\newtheorem{example}[thm]{Example}
\newtheorem{remark}[thm]{Remark}

\usepackage{fullpage}
\begin{document}

%Arxiv
\title{Generalized hyperbolicity for diffeomorphisms of Banach spaces}

\author{ Sergey Tikhomirov 
%Arxiv
\footnote{Pontificia Universidade Catolica do Rio de Janeiro (PUC-Rio); Institute of Mathematics, Statistics and Computer Science, University of São Paulo (IME-USP), Rua do Matão, 1010 – ZIP CODE 05508-090 – São Paulo – SP sergey.tikhomirov@gmail.com}
}

\date{}

%Arxiv
\maketitle

\vspace{-29pt}

\begin{abstract}
%We introduce the notion of generalized hyperbolicity for nonlinear dynamics in Banach spaces. 
%The definition allows the hyperbolic splitting to be discontinuous and assumes only inclusions, rather than equalities, in the invariance condition for both the stable and unstable subspaces. 
%This allows us to cover Morse-Smale systems and generalized hyperbolic linear dynamics. 
%We propose generalized hyperbolicity as the infinite-dimensional counterpart of Axiom~A together with the strong transversality condition. 
%On smooth compact manifolds, generalized hyperbolicity is equivalent to Axiom~A and strong transversality.
%In Banach spaces we establish the dynamical side of this correspondence: every diffeomorphism satisfying generalized hyperbolicity has the Lipschitz shadowing property (with shadowing orbits not necessarily unique), has periodic orbits dense in its chain-recurrent set, and is robust under perturbations small in the uniform $C^1$ distance. 
%The main technical point is that shadowing is obtained without any continuity of the stable/unstable splitting, which rules out the usual graph-transform and contraction arguments. 
%On the structural-stability side we prove two partial results: under a uniformly continuous splitting the diffeomorphism is semi-structurally stable (two-sided semiconjugate to every perturbation small in the uniform $C^1$ distance by uniformly continuous maps close to the identity), and it is structurally stable when in addition $E^s_x\oplus (Df(x))^{-1}E^u_{f(x)}$ is independent of $x$.

We introduce generalized hyperbolicity for nonlinear dynamics in Banach spaces. The definition allows the stable/unstable splitting to be discontinuous and requires only inclusions, rather than equalities, in the invariance conditions for both subspaces.
On smooth compact manifolds, generalized hyperbolicity is equivalent to Axiom~A and the strong transversality condition, providing a finite-dimensional calibration of the proposed Banach-space theory.
For $C^1$-diffeomorphisms of the whole Banach space such that $Df$ and $D(f^{-1})$ are globally bounded and $Df$ is uniformly continuous, we establish the principal dynamical consequences of generalized hyperbolicity: Lipschitz shadowing, density of periodic points in the chain-recurrent set, and robustness under perturbations small in the uniform $C^1$ distance. 
The shadowing result requires no continuity of the splitting, and shadowing trajectories need not be unique.
Under the additional assumption that the splitting is uniformly continuous, we prove semi-structural stability. 
If, in addition, the subspace $E^s_x\oplus (Df(x))^{-1}E^u_{f(x)}$ is independent of $x$, we obtain structural stability.

\end{abstract}


\section{Introduction}

Hyperbolicity is one of the central notions in the qualitative theory of dynamical systems. It was introduced in the pioneering work of D. Anosov \cite{Anosov} on the structural stability of geodesic flows on manifolds with negative curvature. 
Later, hyperbolicity played a crucial role in the characterization of structurally stable diffeomorphisms on smooth compact manifolds: structural stability is equivalent to Axiom A and the strong transversality condition \cite{Mane1982, Mane1987, Rob71, Robinson1976}. 

For infinite-dimensional dynamical systems, the notion of hyperbolicity is more subtle. Let us illustrate the difference at the level of linear maps. An invertible linear map $T: \Rr^n \to \Rr^n$ is called hyperbolic if there exist $k > 0$ and a splitting $E^s \oplus E^u = \Rr^n$ such that
\begin{equation}\label{eq:HypIncl}
    T E^s = E^s, \quad T E^u = E^u,
\end{equation}
\begin{equation}\label{eq:HypContr}
    \mbox{$T|^k_{E^s}$, $T|^{-k}_{E^u}$ are contractions}.
\end{equation}
Such maps are known to be locally structurally stable (Grobman–Hartman theorem), to have the shadowing property, and to be expansive \cite{KatokHas}. 
Recently, the notion of generalized hyperbolicity for Banach spaces was introduced in \cite{Puj18, Puj21}. 

Throughout the paper, $\Bb$ denotes an arbitrary Banach space, with no reflexivity or separability assumption unless explicitly stated
otherwise.
\begin{defin}\label{def:genhyp}
We say that a linear isomorphism $T:\Bb\to\Bb$ is \textit{generalized hyperbolic} if there exists a splitting $\Bb = E^s \oplus E^u$ into two closed subspaces such that for some $k>0$ conditions~\eqref{eq:HypContr} hold and
\begin{equation}\label{eq:GenHypIncl}
    T E^s \subset E^s, \quad T^{-1} E^u \subset E^u.
\end{equation}
\end{defin}
For finite-dimensional spaces $\Bb$, inclusions \eqref{eq:GenHypIncl} imply \eqref{eq:HypIncl}, but this is not so in the infinite-dimensional case. Generalized hyperbolic linear isomorphisms satisfy the shadowing property, and a Grobman-Hartman-type theorem. Their periodic orbits may be dense in the whole space and need not be expansive \cite{Puj18, Puj21, NilsonAli, Bernardes2026}.

In this paper, we introduce generalized hyperbolicity for nonlinear diffeomorphisms of Banach spaces. We propose it as a Banach-space counterpart of Axiom~A together with the strong transversality condition.
\begin{defin}\label{def:CLf}
We say that a diffeomorphism $f:\Bb \to \Bb$  satisfies \textit{generalized hyperbolicity} if there exist $C > 0$, $\lambda \in (0, 1)$ and a family of projections $P_x, Q_x: \Bb \to \Bb$ satisfying the following. Denote 
$E^s_x = P_x(\Bb)$, $E^u_x = Q_x(\Bb)$.
\begin{itemize}
    \item[(GH1)] \textbf{Splitting.} For all $x \in \Bb$, $P_x + Q_x = Id$, $\|P_x\|, \|Q_x\| \leq C$. In particular $E^s_x \oplus E^u_x = \Bb$.
    \item[(GH2)] \textbf{Inclusion.} The subspaces $E^s_x$ and $E^u_x$ are invariant under the forward and backward iterations of $Df$, respectively:
    $$
    Df(x)E^s_x \subseteq E^s_{f(x)}, \; D(f^{-1})(x)E^u_x \subseteq E^u_{f^{-1}(x)}, \quad \mbox{ for any $x \in \Bb$}.
    $$
    \item[(GH3)]  \textbf{Exponential estimates.}  For any $x \in \Bb$, $n > 0$
    $$
    |Df^n(x)v^s_x| \leq C \lambda^n |v^s_x|, \quad v^s_x \in E^s_x;
    $$
    $$
    |D(f^{-n})(x)v^u_x| \leq C \lambda^n |v^u_x|, \quad v^u_x \in E^u_x.
    $$
\end{itemize}
\end{defin}
The definition is inspired by generalized hyperbolicity for linear isomorphisms and by the $(C,\lambda)$-structure introduced by S.~Pilyugin in \cite{PilCL}. 
It differs from ordinary hyperbolicity in two essential respects: condition {\rm(GH2)} requires only inclusions, and no continuity of the splitting $E_x^s\oplus E_x^u$ is assumed. 
As in the linear theory \cite{Puj18,Puj21,NilsonAli}, the splitting need not be uniquely determined by the dynamics. 
%Generalized hyperbolicity is thus a property of $f$ alone: it requires the existence of \textit{some} family $\{P_x, Q_x\}$ satisfying {\rm(GH1)}--{\rm(GH3)}, and imposes no continuity, measurability, or canonical choice of this family.
Fully justifying the term ``hyperbolicity'' would require an adapted metric in which {\rm(GH3)} holds with $C=1$. We expect that such a metric can be constructed as in the finite-dimensional setting \cite{KatokHas,Gourmelon}, but do not pursue this here.

%Definition~\ref{def:CLf} is proposed as the Banach-space counterpart of Axiom~A together with the strong transversality condition. The results of this paper support this proposal from several directions. The first is an exact calibration by the classical finite-dimensional theory; the remaining ones concern a uniform global $C^1$ category, in which the maps act on the entire Banach space, $Df$ and $D(f^{-1})$ are globally bounded, and $Df$ is uniformly continuous.

Our results concern two settings. In finite dimensions, we obtain an exact calibration of Definition~\ref{def:CLf} by the classical theory. In Banach spaces, we work in a uniform global $C^1$ category, in which the maps act on the entire Banach space, $Df$ and $D(f^{-1})$ are globally bounded, and $Df$ is uniformly continuous.

\begin{enumerate}
\item \textbf{Calibration in finite dimensions.} For a $C^1$-diffeomorphism of a smooth compact manifold, generalized hyperbolicity is equivalent to Axiom~A together with the strong transversality condition (Theorem~\ref{stat:finite-dimensional-characterization}). 
%In finite dimensions the definition therefore identifies exactly the classical class, and no larger one.

\item \textbf{Shadowing.} Generalized-hyperbolic diffeomorphisms have the Lipschitz shadowing and Lipschitz periodic shadowing properties (Theorem~\ref{thm:main-results}\,(\ref{item:main-LipSh})--(\ref{item:main-LipPerSh})). The shadowing trajectory is not necessarily unique.

\item \textbf{Density of periodic points.} Periodic points are dense in the chain-recurrent set: $\overline{Per(f)} = CR(f)$ (Theorem~\ref{thm:main-results}\,(\ref{item:main-chain})).

\item \textbf{Robustness.} Generalized hyperbolicity is robust in the uniform $C^1$ distance (Theorem~\ref{thm:main-results}\,(\ref{item:main-robust})).

\item \textbf{Stability.} Under uniform continuity of the splitting we prove semi-structural stability (Definition~\ref{def:semi-SS}), and structural stability if additionally $E^s_x \oplus (Df(x))^{-1}E^u_{f(x)}$ is independent of~$x$; this additional condition is automatic for an invariant hyperbolic splitting (Theorem~\ref{thm:SS}).
\end{enumerate}

Items (2)-(5) are the properties that Axiom~A with strong transversality provides in finite dimensions, and they persist in Banach spaces even when the splitting is discontinuous and both inclusions in {\rm(GH2)} are strict. Two implications that hold in finite dimensions remain open in Banach spaces: whether generalized hyperbolicity implies structural stability, and whether Lipschitz shadowing implies generalized hyperbolicity; see Conjecture~\ref{conj:SS-general} and Section~\ref{sec:discussion}.

Two examples developed in Section~\ref{sec:Examples} isolate the two
distinctive features of Definition~\ref{def:CLf}. The nonlinear shift of
Example~\ref{ex:shift}, a counterpart of the generalized-hyperbolic weighted shifts of \cite{Puj18,Puj21}, has a constant splitting but both inclusions in {\rm(GH2)} are strict; nevertheless it is structurally stable by Theorem~\ref{thm:SS}\,(\ref{item:SS-full}). 
In Example~\ref{ex:product}, an infinite product of a Morse--Smale map of $\Rr$, the splitting is discontinuous and its stable dimension varies along trajectories; for $p=+\infty$, the fixed points realize every finite stable dimension and also infinite-dimensional stable subspaces.
Each example therefore isolates one feature that lies beyond the earlier Banach-space frameworks discussed below; see Remark~\ref{rem:minimal}.

There are three difficulties distinguish the Banach-space setting technically.
First, no continuity of the splitting $E^s_x\oplus E^u_x$ is available: shadowing is obtained by constructing, from a pseudotrajectory, successive corrections whose errors tend to zero.
Second, although the global assumptions provide estimates uniform in the base point, bounded sets in a Banach space need not be compact. The passage from finite to infinite pseudotrajectories therefore cannot rely on selecting a convergent subsequence of shadowing points; instead, it uses the quantitative linear-shadowing result of \cite{NilsonPeris} and an iterative error-reduction argument.
Third, $Df(x)|_{E^u_x}$ need not be invertible, which forces a modification of the classical graph-transform argument through auxiliary spaces $F_x\subset E^u_x$.

The mechanism common to all three is bounded solvability of the linearized inhomogeneous equations.
For a sequence $\mathcal A=\{A_k=Df(x_k)\}$ along an exact trajectory, consider
$$
T_{\mathcal A}\{v_k\}=\{v_{k+1}-A_kv_k\}.
$$
Generalized hyperbolicity provides a bounded linear right inverse for $T_{\mathcal A}$, uniformly over all trajectories (Lemma~\ref{lem:SBS}), and under small perturbations this yields uniformly bounded solutions along pseudotrajectories.
Robustness and stability are obtained by transferring the same construction to sequences of linear operators and to linear cocycles; the cocycle notion, which may also be of independent interest, is the basis of the semi-structural and structural stability arguments (Theorem~\ref{thm:SS}).

In finite-dimensional dynamics, several notions have been introduced that generalize hyperbolicity and are equivalent to Axiom A plus the strong transversality condition, such as locally Anosov diffeomorphisms \cite{Rob71} and piecewise hyperbolic sequences of linear operators \cite{Pli77, Pli80}. 
At the level of sequences of linear operators, generalized hyperbolicity
is equivalent in finite dimensions to exponential dichotomies on
$\Zz^\pm$ together with transversality. In infinite dimensions neither
implication is automatic: generalized hyperbolicity may hold without
half-line dichotomies (Example~\ref{ex:CL-not-ED}); the
converse is Question~\ref{quest:ED}; see Sections~\ref{sec:inh},~\ref{sec:discussion}.

The hyperbolic theory of nonlinear infinite-dimensional systems is still far from fully developed, although the literature is substantial. 
Relevant contributions include studies of dynamics near an equilibrium \cite{Palis1968,FSSW1996,FiedlerTuraev1998,YWPT2022}, shadowing and structural stability near a homoclinic trajectory \cite{SteinleinWalther1989,SteinleinWalther1990,LaniWayda1995,LaniWayda1995-2,LaniWayda1999,Li2003,Zeng1995,Blazquez1986}, Lyapunov exponents in the Hilbert setting \cite{LianYoung2012}, and hyperbolic maps on the infinite-dimensional torus \cite{Glyzin1,Glyzin9}; for general overviews, see \cite{Henry1981,SellYou,Li2004}.
A separate, related literature concerns chaotic linear operators on Banach and Fréchet spaces \cite{GodefroyShapiro1991,HouTianZhu2012,MartínezGimenezOprochaPeris2013,BBMP2013,BBPW2018,BernardesMessaoudi2021,BMGR2025,Nilson25,BBP2026,BV2025,Puj18,NilsonAli,NilsonPeris}; see also the monographs \cite{BayartMatheron2009,GrosseErdmannPerisManguillot2011}.

The closest Banach-space notions are those of \cite{SteinleinWalther1990,LaniWayda1995}, introduced near homoclinic orbits of delay equations. Definition~\ref{def:CLf} relaxes these frameworks in three independent directions. 
First, neither projection is required to depend continuously on $x$, allowing dimension variability and Morse-Smale dynamics such as Example~\ref{ex:product}; see also the related recent Hilbert-space results of \cite{ACK2025}. 
Second, {\rm(GH2)} allows both inclusions to be strict, as realized by Example~\ref{ex:shift}; in particular, because the unstable inclusion is strict, unstable vectors need not grow and the map need not be expansive.
Third, $E_x^u$ may be infinite-dimensional, unlike theories based on finite-dimensional unstable or inertial manifolds \cite{Henry1981,SellYou,DKKP2012}. 
Thus neither principal example belongs to the frameworks of \cite{SteinleinWalther1990,LaniWayda1995,Henry1981,SellYou}. 

Extending the definition to the non-invertible maps also treated in \cite{SteinleinWalther1990,LaniWayda1995} is left for future work. A recent result of Arrieta, Carvalho, and Takaessu \cite{ACK2025} gives Lipschitz shadowing on the global attractor and H\"older shadowing in a neighborhood of the attractor for Morse-Smale semigroups on a Hilbert space. Our setting is different, since we consider invertible dynamics under uniform global assumptions; within this setting, our shadowing conclusion is global on the entire Banach space rather than restricted to an attractor or its neighborhood.

One motivation comes from renormalization operators in dynamical systems, including those associated with the logistic map \cite{RL1,RL2,RL3,RL4}, the Hénon map \cite{RH1,RH2,RH3,RH4}, and the complex quadratic family \cite{RUQ1,RL2,RUQ3}.
Generalized hyperbolicity is likely to be most relevant when the renormalization dynamics has a nontrivial invariant set, as may occur in spontaneous stochasticity \cite{MailRaib1,MailRaib2} and in the limiting dynamics of gravitational fingers for the incompressible porous media equation \cite{otto-menon-2005,Boffetta2,PetrovaTikhomirovEfendiev2025}.
The cocycle version may also be useful for stability questions concerning endomorphisms \cite{End-Berger-Koscard,End-Ikeda,End-Uruguay}.

The paper is organized as follows. 
Section~\ref{sec:thms} states the main results, and Section~\ref{sec:Examples} presents the examples. 
The proof strategy is summarized in Section~\ref{sec:scheme}; the required linear theory is developed in Section~\ref{sec:inh}, and the shadowing, robustness, and stability results are proved in Sections~\ref{sec:proofsSh}--\ref{sec:SS}. 
Section~\ref{sec:discussion} contains open questions and further directions. Appendix~A proves the quantitative infinite-shadowing result used in the paper, while Appendix~B
establishes the finite-dimensional characterization.

\section{Main results}\label{sec:thms}

\subsection{Calibration in finite dimensions}

We begin with the calibration of Definition~\ref{def:CLf} by the classical finite-dimensional theory.

\begin{thm}[Finite-dimensional characterization]
\label{stat:finite-dimensional-characterization}
For a $C^1$-diffeomorphism of a smooth compact manifold, the tangent-bundle analogue of generalized hyperbolicity is equivalent to Axiom~A and the strong transversality condition.
\end{thm}

The precise definition on manifolds and the proof based on Robinson \cite{Robinson1976} and Pilyugin \cite{PilCL} results are given in Appendix~B.

\subsection{The uniform global $C^1$ category}

In the setting of Theorem~\ref{stat:finite-dimensional-characterization}, compactness of the phase space does substantial work: it makes the bounds on the derivative and the modulus of continuity of $Df$ automatically uniform in the base point.
On a Banach space none of this is automatic, and the pseudotrajectories considered below may visit arbitrary regions of a noncompact phase space, so the estimates in the shadowing and robustness arguments must be independent of the base point. We therefore impose the uniformity that compactness would otherwise supply.

We now specify the uniform global $C^1$ category used throughout the paper. Let $f:\Bb\to\Bb$ be a $C^1$-diffeomorphism with Fréchet derivative $Df:\Bb\to\mathcal L(\Bb,\Bb)$. We require uniform forward and backward differential bounds and a base-point-independent modulus of continuity for $Df$. More precisely, assume that there exist $R>0$ and a nondecreasing function $r:\Rr^+\to\Rr^+$ such that
\begin{equation}\label{eq:C1-1}
    |f(x+v) - f(x) - Df(x)v| \leq |v|r(|v|), \quad x, v \in \Bb,
\end{equation}
\begin{equation}\label{eq:C1-2}
      \|Df(x)\| \leq R, \quad \|D(f^{-1})(x)\| \leq R, \quad x \in \Bb,
\end{equation}
\begin{equation}\label{eq:C1-r1}
      \|Df(x+v) - Df(x)\| \leq r(|v|),     \quad x, v \in \Bb, v \ne 0,
\end{equation}
\begin{equation}\label{eq:C1-r2}
       \quad \lim_{z \to 0} r(z) = 0.
\end{equation}
%\begin{remark}\label{rem:standing}
Conditions \eqref{eq:C1-1}--\eqref{eq:C1-r2} define the uniform global $C^1$ framework used in this paper. The bounds on $Df$ and $D(f^{-1})$ provide uniform forward and backward control, while \eqref{eq:C1-r1}--\eqref{eq:C1-r2} provide a uniform estimate for the nonlinear remainder. We emphasize that compactness is replaced only in this respect: the passage from finite to infinite pseudotrajectories is still obtained without extracting convergent subsequences.

Define 
$$
s_f(x, v) := f(x+v) - f(x) - Df(x)v, \quad  x, v \in \Bb.
$$
Conditions \eqref{eq:C1-r1}, \eqref{eq:C1-r2} imply that for any $\varepsilon > 0$ there exists $\delta > 0$ such that 
\begin{equation}\label{eq:s-Lip}
    |s_f(x, v_1) - s_f(x, v_2)| \leq \varepsilon |v_1 - v_2|, \quad x \in \mathbb{B}, |v_1|, |v_2| < \delta.
\end{equation}

We denote by $\DUC(\Bb)$ the set of $C^1$-diffeomorphisms $f:\Bb\to\Bb$ satisfying \eqref{eq:C1-1}--\eqref{eq:C1-r2}. 
We equip this set with the extended uniform $C^1$ distance
$$
\|f-g\|_{C^1}
:=
\sup_{x\in\Bb}|f(x)-g(x)|
+
\sup_{x\in\Bb}\|Df(x)-Dg(x)\|,
\qquad f,g\in\DUC(\Bb).
$$

The value of this distance may be infinite. 
The use of the uniform, rather than compact-open, distance is matched to the global nature of the results: it controls perturbations on the entire Banach space. 
All statements concerning $C^1$-small perturbations, openness, and robustness below refer to this uniform distance and are understood relative to $\DUC(\Bb)$.

This framework is nontrivial and perturbatively stable: all examples in Section~\ref{sec:Examples} satisfy the assumptions, and by Theorem~\ref{thm:main-results}\,(\ref{item:main-robust}) and Proposition~\ref{stat:difconj}, they generate families open in the uniform $C^1$ distance within $\DUC(\Bb)$.

We first record elementary properties of generalized hyperbolicity within this category.

\begin{statement}[Elementary properties]\label{stat:difconj}
Let $f, h \in \DUC(\Bb)$ and assume that $f$ satisfies generalized hyperbolicity with constants $C>0$, $\lambda \in (0,1)$ and splitting $\Bb = E^s_x \oplus E^u_x$. 
Then:
\begin{itemize}
\item[(P1)] $\DUC(\Bb)$ is a group under composition, and $g = h \circ f \circ h^{-1}$
satisfies generalized hyperbolicity with constants $R_h^2 C$ and $\lambda$, where $R_h$ bounds $\|Dh\|$ and $\|Dh^{-1}\|$;
\item[(P2)] $f^{-1}$ satisfies generalized hyperbolicity with the same constants and with the splitting obtained by interchanging $E^s_x$ and $E^u_x$;
\item[(P3)] for every $n \geq 1$, $f^n$ satisfies generalized hyperbolicity with constants $C$, $\lambda^n$ and the same splitting;
\item[(P4)] let $p \in [1,+\infty]$ and let $f_i \in \DUC(\Bb_i)$, $i \in \mathcal I$,
satisfy $f_i(0) = 0$, conditions \eqref{eq:C1-1}--\eqref{eq:C1-r2} with $R$ and $r$ independent of $i$, and generalized hyperbolicity with $C, \lambda$ independent of $i$.
Then $f(\{x_i\}) = \{f_i(x_i)\}$ belongs to $\DUC(\Bb)$, where
$\Bb = \bigl(\bigoplus_{i \in \mathcal I} \Bb_i\bigr)_{\ell^p}$, and satisfies generalized hyperbolicity with constants $C, \lambda$.
\end{itemize}
\end{statement}
The proof is straightforward and is given in Appendix~C.

\begin{remark}
The semiflows of parabolic and delay equations that partly motivate this work belong to a different analytical category: they are typically non-invertible and generated by unbounded, densely defined operators, so \eqref{eq:C1-2}--\eqref{eq:C1-r1} do not apply as stated. Extending generalized hyperbolicity to that setting requires additional ideas and is left for future work; we return to this question in Section~\ref{sec:discussion}.
\end{remark}

\subsection{Shadowing, density of periodic points, and robustness}

One of the important results in the hyperbolic theory of dynamical systems is the shadowing lemma \cite{Anosov, Bowen}, which states that in a neighborhood of a hyperbolic set a dynamical system satisfies the shadowing property. Recall the notion of shadowing property.
For an interval 
\begin{equation}\label{eq:I}
I=(a, b) \subset \Zz \quad \mbox{with $a \in \mathbb{Z} \cup\{-\infty\}, b \in \mathbb{Z} \cup\{+\infty\}$}    
\end{equation}
denote $\tilde{I} = \{k \in \Zz: k, k+1 \in I\}$.

\begin{defin}
For $d>0$, we say that a sequence $\left\{y_{k} \in \Bb\right\}_{k \in I}$ is a $d$-pseudotrajectory if
\begin{equation}\label{eq:pst1}
\operatorname{dist}\left(y_{k+1}, f\left(y_{k}\right)\right) \leq d, \quad k \in \tilde{I}.
\end{equation}
If the interval $I$ is not specified, we assume that $I = \mathbb{Z}$.
\end{defin}
By \eqref{eq:C1-2}, for any $d$-pseudotrajectory $\{y_k\}_{k \in I}$, the following inequality holds 
\begin{equation}\label{eq:pst2}
|y_{k} - f^{-1}(y_{k+1})| \leq Rd, \quad k \in \tilde{I}.
\end{equation}

Initially, pseudotrajectories were introduced in the theory of chain-recurrent sets and structural stability. 
We use the following notions.
\begin{defin}
We say that a point $x \in \Bb$ is chain-recurrent for a diffeomorphism $f$ if, for any $d>0$, there exists a finite $d$-pseudotrajectory $\{y_k\}_{k \in [0, N]}$ with $y_0 = y_N = x$. Denote the set of all chain-recurrent points by $CR(f)$. Note that $CR(f)$ is closed.
\end{defin}
For $f \in \DUC(\Bb)$, denote by $Per(f)$ the set of periodic points. Note that $Per(f) \subset CR(f)$.

\begin{defin}
We say that $f$ has the \textit{finite Lipschitz shadowing property} (FinLipSh) if there exist $M, d_{0} > 0$ such that for any $d < d_{0}$, finite interval $I = (a, b)$, and $d$-pseudotrajectory~$\left\{y_{k}\right\}_{k \in I}$ there exists an exact trajectory~$\{x_{k}\}_{k \in I}$ satisfying $|x_{k} - y_{k}| \leq M d$, $k \in I$.
\end{defin}

\begin{defin}
We say that $f$ has the \textit{Lipschitz shadowing property} (LipSh) if there exist $M, d_{0} > 0$ such that for any $d < d_{0}$ and $d$-pseudotrajectory $\left\{y_{k}\right\}_{k \in \mathbb{Z}}$ there exists an exact trajectory $\left\{x_{k}\right\}_{k \in \mathbb{Z}}$ satisfying $|x_{k} - y_{k}| \leq M d$, $k \in \mathbb{Z}$.
\end{defin}

\begin{defin}
We say that $f$ has the \textit{Lipschitz periodic shadowing property} (LipPerSh) if there exist $M, d_{0} > 0$ such that for any $d < d_{0}$ and periodic $d$-pseudotrajectory $\left\{y_{k}\right\}_{k \in \mathbb{Z}}$ with $y_{k+N} = y_k$ (for some $N > 0$) there exists a periodic exact trajectory $\left\{x_{k}\right\}_{k \in \mathbb{Z}}$ satisfying
\begin{equation}\notag%\label{eq:1-per}
x_{k+N} = x_k, \quad |x_{k} - y_{k}| \leq M d, \quad k \in \mathbb{Z}. 
\end{equation}
\end{defin}

For finite-dimensional spaces, the finite Lipschitz shadowing property is equivalent to the Lipschitz shadowing property. The proof is relatively straightforward \cite{PilBook,Palmer2000}, but uses compactness of closed balls, which is unavailable in infinite-dimensional spaces. For a more comprehensive overview of Lipschitz shadowing, see \cite{PilyuginSakai2017}. In the context of infinite-dimensional Banach spaces, a corresponding finite-to-infinite shadowing result is known for linear systems \cite{NilsonPeris}; for Fréchet spaces, this implication may fail \cite{Nilson25}.

The principal unconditional conclusions of the paper are collected in
the following theorem.

\begin{thm}[Main theorem]\label{thm:main-results}
Let $f\in\DUC(\Bb)$ satisfy generalized hyperbolicity. Then:
\begin{enumerate}[
    label=\textup{(M\arabic*)},
    ref=\textup{M\arabic*}
]
    \item \label{item:main-LipSh}
    $f$ has the Lipschitz shadowing property;
    \item \label{item:main-LipPerSh}
    $f$ has the Lipschitz periodic shadowing property;
    \item \label{item:main-chain}
    periodic points are dense in the chain-recurrent set: $\overline{Per(f)}=CR(f)$;
    \item \label{item:main-robust}
    generalized hyperbolicity is robust in the uniform $C^1$
    distance: if $f$ satisfies generalized hyperbolicity with constants
    $C>0$ and $\lambda\in(0,1)$, then, for every
    $\lambda_1\in(\lambda,1)$, there exist $\delta>0$ and $C_1>0$ such
    that every $g\in\DUC(\Bb)$ satisfying
    $
    \|g-f\|_{C^1}<\delta
    $
    is generalized hyperbolic with constants $C_1$ and $\lambda_1$.
\end{enumerate}
\end{thm}

%\begin{remark}
Theorem~\ref{thm:main-results}\,(\ref{item:main-LipSh}) may be viewed as a generalization of the
shadowing lemma for $C^1$-diffeomorphisms of Banach spaces.
We do not claim uniqueness of the shadowing trajectory; indeed, it fails
for Examples~\ref{ex:shift}, \ref{ex:product}, \ref{ex:coh}.
Theorem~\ref{thm:main-results}\,(\ref{item:main-LipPerSh}) is in the spirit of Shaobo Gan's
shadowing lemma \cite{Gan2002}, although our proof is based on different
ideas.
The proof of Theorem~\ref{thm:main-results}\,(\ref{item:main-robust}) uses the Axiom of Choice; see Remark~\ref{rem:AoC}.
%\end{remark}

\subsection{Structural stability}

The situation with respect to structural stability is more involved. We conjecture that generalized hyperbolicity implies structural stability in full (Conjecture~\ref{conj:SS-general} below); in this paper we prove only a weaker statement (Theorem \ref{thm:SS}); other related open questions are discussed in Section~\ref{sec:discussion}.

\begin{defin}
We say that $f \in \DUC(\Bb)$ is \textit{structurally stable} if for every $\varepsilon > 0$ there exists $d_0 = d_0(\varepsilon) > 0$ such that for any diffeomorphism $g \in \DUC(\Bb)$ with $\|f-g\|_{C^1} < d_0$ there exists a continuous map $h_1: \Bb \to \Bb$ satisfying $\sup_{x \in \Bb}|h_1(x)| < \varepsilon$ such that $(Id+h_1)$ is a homeomorphism and $g \circ (Id+h_1) = (Id+ h_1) \circ f$. If, moreover, the map $h_1$ can be chosen so that $Id+h_1$ and $(Id+h_1)^{-1}$ are uniformly continuous, we say that $f$ is \textit{uniformly structurally stable}.
\end{defin}

\begin{conj}\label{conj:SS-general}
Every diffeomorphism $f \in \DUC(\Bb)$ satisfying generalized hyperbolicity is structurally stable.
\end{conj}

To obtain partial results toward Conjecture~\ref{conj:SS-general}, we impose uniform continuity of the generalized-hyperbolic splitting.
This leads to uniform semi-structural stability and, under an additional compatibility condition, to uniform structural stability.

\begin{defin}
We say that $f \in \DUC(\Bb)$ satisfies \textit{uniformly continuous generalized hyperbolicity} if it satisfies generalized hyperbolicity with a family of projections $\{P_x, Q_x\}$ for which the maps $x \to P_x$, $x \to Q_x$ are in addition uniformly continuous, i.e.
\begin{equation}\label{eq:cont-PQ}
    \mbox{$\forall$ $\varepsilon > 0$ $\exists$ $\delta > 0$ such that $\|P_x - P_y\|, \|Q_x - Q_y\| < \varepsilon$ for $|x-y| < \delta$}. 
\end{equation}
\end{defin}

\begin{defin}\label{def:semi-SS}
We say that $f \in \DUC(\Bb)$ is \textit{uniformly semi-structurally stable} if for every $\varepsilon > 0$ there exists $d_0 = d_0(\varepsilon) > 0$ such that for any diffeomorphism $g \in \DUC(\Bb)$ with $\|f-g\|_{C^1} < d_0$ there exist uniformly continuous maps $h_1, h_2: \Bb \to \Bb$ satisfying
$
\sup_{x \in \Bb}|h_1(x)|, \; \sup_{x \in \Bb}|h_2(x)| < \varepsilon,
$
such that 
$$
g \circ (Id + h_1) = (Id + h_1) \circ f, \quad f \circ (Id + h_2) = (Id + h_2) \circ g.
$$
\end{defin}

\begin{thm}\label{thm:SS}
If $f \in \DUC(\Bb)$ satisfies uniformly continuous generalized hyperbolicity, then there exist $M>0$, $d_0 > 0$ such that for any $d \in (0, d_0)$ and diffeomorphism $g \in \DUC(\Bb)$ with $\|f-g\|_{C^1} < d$ the following holds.
    \begin{enumerate}[
    label=\textup{(S\arabic*)},
    ref=\textup{S\arabic*}
]
        \item \label{item:SS-semi}
            There exist uniformly continuous maps $h_1, h_2:\Bb \to \Bb$ satisfying
        \begin{equation}\label{eq:conj-h1}
        g \circ (Id+h_1) = (Id + h_1) \circ f,
        \end{equation}
        \begin{equation}\label{eq:conj-h2}
        f \circ (Id+h_2) = (Id + h_2) \circ g,
        \end{equation}
        \begin{equation}\label{eq:h1h2leqMdelta}
        |h_1(x)|, |h_2(x)| < Md, \quad x \in \Bb,
        \end{equation}
        hence $f$ is uniformly semi-structurally stable.
        \item \label{item:SS-full}
        If additionally the spaces $E^s_x \oplus (Df(x))^{-1}E^u_{f(x)}$ do not depend on $x$, then $h_1$, $h_2$ can be chosen such that 
        \begin{equation}\label{eq:Idh-homeo}
        (Id+h_1) \circ (Id+h_2) = Id, \quad (Id+h_2) \circ (Id+h_1) = Id.
        \end{equation}
In that case, the maps $(Id+h_1)$, $(Id+h_2)$ are homeomorphisms, and hence $f$ is uniformly structurally stable.
    \end{enumerate}
\end{thm}

\begin{remark}
The additional condition in Theorem~\ref{thm:SS}\,(\ref{item:SS-full}) is automatic for an invariant hyperbolic splitting. Indeed, if
$
Df(x)E_x^u=E_{f(x)}^u,
$
then
$
(Df(x))^{-1}E_{f(x)}^u=E_x^u,
$
and consequently
$
E_x^s\oplus(Df(x))^{-1}E_{f(x)}^u
=
E_x^s\oplus E_x^u
=
\Bb.
$
\end{remark}

To prove this theorem, we introduce a notion similar to generalized hyperbolicity for linear cocycles and establish a robustness property for it (Section \ref{sec:SS}). 

\begin{remark}\label{rem:LaniWaydaResults}
    In \cite{LaniWayda1995} strong results on shadowing and stability (persistence in their terminology) were obtained. At the same time, their proof strongly relies on an expansivity argument, which is not available in our setting. For other cases in which uniqueness or expansivity simplifies the analysis of the interplay between shadowing, structural stability, and hyperbolicity or exponential dichotomy, see, for example, \cite{Henry94}.
\end{remark}

\section{Examples}\label{sec:Examples}

The examples below serve three complementary purposes. 
The nonlinear shift shows that the inclusions in {\rm(GH2)} may be strict and that shadowing need not be unique. 
The coordinatewise Morse--Smale example shows that the splitting may be discontinuous and that the stable index may vary without bound or become infinite. 
Finally, the pushforward constructions show how generalized hyperbolicity can be transferred from nonlinear dynamics to natural linear operators on spaces of functions and vector fields.

\subsection{Strict invariant inclusions: a nonlinear shift}
In this subsection, we construct a nonlinear shift satisfying generalized hyperbolicity, generalizing hyperbolic weighted translation operators \cite{Puj18,Puj21,LatushkinRandolph1995}.
\begin{example}\label{ex:shift}
    Let $\lambda \in (0, 1)$, $R > 1/\lambda$, $M > 1$  and diffeomorphisms $a_k: \Rr \to \Rr$ satisfy 
\begin{equation}\label{eq:shift-ak}
    a_k(0) = 0; \quad 
    a'_k(x) \in \left(\frac{1}{\lambda}, R\right), \; k < 0; \quad
    a'_k(x) \in \left(\frac{1}{R}, \lambda\right), \; k \geq 0,    
\end{equation}
\begin{equation}\label{eq:shift-akc2}
a_k \in C^2, \quad |a''_k|, |(a^{-1}_k)''| \leq M.
\end{equation}
Let $\Bb = \ell^p(\Zz)$ with $p \in [1, +\infty]$ and consider the map $f: \Bb \to \Bb$ defined as
$$
f(\{x_k\}) = \{y_k\}, \quad \mbox{where $y_{k+1} = a_k(x_k)$}.
$$
Then $f$ is a diffeomorphism satisfying generalized hyperbolicity.
\end{example}
\begin{proof}
By the definition of $f$, the inequalities $|y_{k+1}| \leq R |x_k|$ hold, hence $\|f(\{x_k\})\| \leq R \|\{x_k\}\|$, and $f$ is well-defined. Moreover, $f^{-1}$ is defined by $f^{-1}(\{y_k\}) = \{x_k\}$, where $x_k = a_k^{-1}(y_{k+1})$, with $\|f^{-1}(\{y_k\})\| \leq R\|\{y_k\}\|$; hence $\{x_k\}$ is well-defined.

For $\{x_k\} \in \Bb$ define a linear operator $A(\{x_k\}):\Bb \to \Bb$ by
$$
A(\{x_k\})\{v_k\} = \{w_k\}, \qquad \mbox{where $w_{k+1} = a'_k(x_k)v_k$}.
$$
It is invertible, with $A^{-1}(\{x_k\})\{w_k\} = \{v_k\}$, where $v_k = w_{k+1}/a'_k(x_k)$, and by \eqref{eq:shift-ak} we have $\|A(\{x_k\})\|, \|A^{-1}(\{x_k\})\| \leq R$. 
Let us show that $Df(\{x_k\}) = A(\{x_k\})$, where $Df$ is a Fréchet derivative. Indeed,
    \begin{multline*}
    f(\{x_k\} + \{v_k\}) - f(\{x_k\}) - A(\{x_k\})\{v_k\} = \{z_k\}, \\ \mbox{where $z_{k+1} = a_k(x_k + v_k) - a_k(x_k) - a'_k(x_k)v_k$}.    
    \end{multline*}
According to \eqref{eq:shift-akc2}, we have $|z_k| \leq M|v_k|^2$ and hence inequality \eqref{eq:C1-1} holds with $r(|v|) = M|v|$. Similarly, equation \eqref{eq:C1-r1} holds and $r$ satisfies \eqref{eq:C1-r2}.

Let us show that $f$ satisfies generalized hyperbolicity. Denote 
    $$
    E^s = \{ v_k = 0, k < 0 \}, \quad E^u = \{ v_k = 0, k \geq 0 \}.
    $$
Note that $E^s \oplus E^u = \Bb$. Let $P_x, Q_x$ be the projections onto $E^s$ and $E^u$ respectively (independent of $x$). Then
    $
    E^s_x = P_x(\Bb) = E^s$, $E^u_x = Q_x(\Bb) = E^u$, and $\|P_x\| = \|Q_x\| = 1
    $,
so condition (GH1) holds with $C \geq 1$. By construction,
    $A(x)E^s \subset E^s$ and $A^{-1}(x)E^u \subset E^u$,
hence condition (GH2) holds. By \eqref{eq:shift-ak} we have $\|A(x)|_{E^s}\|, \|A^{-1}(x)|_{E^u}\| \leq \lambda$, and hence condition (GH3) holds with $C = 1$ and the same $\lambda$ as above.
\end{proof}

\noindent\textbf{Consequences of the construction.}
For every $p\in[1,+\infty]$, all conclusions of
Theorem~\ref{thm:main-results} hold. Since the splitting is constant,
it satisfies the additional hypothesis of
Theorem~\ref{thm:SS}\,(\ref{item:SS-full}); hence $f$ is uniformly
structurally stable.

Both inclusions in (GH2) are strict. 
For instance $A^{-1}(x)E^u \subsetneq E^u$, the missing direction being $e_{-1}$.
Indeed, $A(x)e_{-1} = a_{-1}'(x_{-1})\,e_0 \in E^s$, so the unstable direction $e_{-1}$ is carried by $Df$ into the stable indices and $|A^n(x)e_{-1}| \to 0$ instead of growing. Similarly, $A(x)E^s \subsetneq E^s$, the missing direction being $e_0$.
In particular, $f$ is not expansive, and the shadowing trajectory in Theorem~\ref{thm:main-results}\,(\ref{item:main-LipSh}) need not be unique. 
Thus strict invariant inclusions, nonexpansivity, and nonunique shadowing are compatible with uniform structural stability.

The generalized-hyperbolic splitting is itself not uniquely determined
by the dynamics. For every $N\in\Zz$, one may take
$C=R^{|N|}$ and
\[
E_x^s=\{v_k=0,\ k<N\},
\qquad
E_x^u=\{v_k=0,\ k\geq N\},
\]
with the corresponding coordinate projections.

\subsection{Discontinuous splittings and varying index:
products of Morse-Smale maps}

Fix $0<\lambda_1<\lambda_2<1$.
Consider an increasing $C^\infty$-diffeomorphism
$a:\Rr\to\Rr$ whose only fixed points are $-1,0,1$, satisfying
\[
a(-1)=-1,\quad a'(-1)=\lambda_1,\qquad
a(0)=0,\quad a'(0)=1/\lambda_1,\qquad
a(1)=1,\quad a'(1)=\lambda_1.
\]
Assume that, for some $\varepsilon\in(0,1/2)$ and $M>0$,
\begin{equation}\label{eq:aLip}
\lambda_1\leq a'(x)\leq\frac1{\lambda_1},
\qquad |a''(x)|\leq M,
\end{equation}
\[
a'(x)\leq\lambda_2
\quad\text{for }x\in(-\infty,-1+\varepsilon)\cup(1-\varepsilon,+\infty),
\qquad
a'(x)\geq\frac1{\lambda_2}
\quad\text{for }x\in(-\varepsilon,\varepsilon).
\]

Since $a(0)=0$, condition \eqref{eq:aLip} implies
$
|a(x)|\leq\frac1{\lambda_1}|x|$ and
$|a^{-1}(x)|\leq\frac1{\lambda_1}|x|.
$

The exact form of $a$ is not important. Such a map can be obtained by smoothly interpolating between a map of slope $\lambda_1$ in the two attracting regions and a map of slope $1/\lambda_1$ near the repelling fixed point, while preserving the inequalities above.

Thus $a$ is a Morse-Smale diffeomorphism of the line. The points $\pm 1$ are hyperbolic attracting fixed points; the point $x = 0$ is a hyperbolic repelling fixed point. For large enough~$n_0(\varepsilon)$, each trajectory $\{x_n\}$ spends at most $n_0$ steps outside $(-\infty, -1+\varepsilon) \cup (-\varepsilon, \varepsilon) \cup (1-\varepsilon, +\infty)$. Then $a$ satisfies generalized hyperbolicity with constants $C = (\lambda_1\lambda_2)^{-n_0}$, $\lambda = \lambda_2$ and splitting
$$
G^s_x = 
\begin{cases}
    \Rr, \quad |x| > 1/2,\\
    0, \quad |x| \leq 1/2,
\end{cases}
\quad
G^u_x = 
\begin{cases}
    0, \quad |x| > 1/2,\\
    \Rr, \quad |x| \leq 1/2.
\end{cases}
$$

As in Example \ref{ex:shift}, consider $\Bb = \ell^p(\Zz)$ with $p \in [1, +\infty]$.
\begin{example}\label{ex:product}
    Consider $f:\Bb \to \Bb$ defined by
    $
    f(\{x_k\}) = \{a(x_k)\}.
    $
    Then $f \in \DUC(\Bb)$ and satisfies generalized hyperbolicity. 
\end{example}
\begin{proof}
By \eqref{eq:aLip}, the map $a$ satisfies \eqref{eq:C1-2} with $R = 1/\lambda_1$ and \eqref{eq:C1-r1} with $r(t) = Mt$, and this $r$ satisfies \eqref{eq:C1-r2}; hence $a \in \DUC(\Rr)$. As explained above, $a$ satisfies generalized hyperbolicity with constants $C = (\lambda_1\lambda_2)^{-n_0}$, $\lambda = \lambda_2$ and the splitting $\Rr = G^s_x \oplus G^u_x$. Since $a(0) = 0$ and
$$
\Bb = \ell^p(\Zz) = \Bigl(\bigoplus_{k \in \Zz} \Rr\Bigr)_{\ell^p},
$$
Proposition~\ref{stat:difconj}\,{\rm(P4)}, applied with $\Bb_k = \Rr$ and $f_k = a$ for every $k \in \Zz$, shows that $f \in \DUC(\Bb)$, that $Df(\{x_k\})\{v_k\} = \{a'(x_k)v_k\}$, and that $f$ satisfies generalized hyperbolicity with the same constants $C$ and $\lambda$, with respect to the splitting
$$
E^s_x = \bigoplus_{k \in \Zz} G_{x_k}^s = \{v_k = 0, \mbox{if $|x_k| \leq 1/2$}\},
\qquad
E^u_x = \bigoplus_{k \in \Zz} G_{x_k}^u = \{v_k = 0, \mbox{if $|x_k| > 1/2$}\}
$$
and the corresponding projections $P_x$, $Q_x$.
\end{proof}

The inclusion in {\rm(GH2)} has a simple dynamical interpretation in
this example. Define
$
I^s(x):=\{k\in\Zz:\ |x_k|>1/2\}.
$
Then
$
I^s(x)\subseteq I^s(f(x)).
$
Thus the stable subspace can acquire new coordinate directions along a
forward trajectory. A direction is added precisely when the
corresponding coordinate leaves the central repelling region and enters
one of the attracting regions. In this sense, the stable index is
nondecreasing along forward trajectories.

The discontinuity of the splitting is visible already in one coordinate.
Let
$
x=\frac12e_0,
$
and $x^{(m)}=\left(\frac12+\frac1m\right)e_0$. Then $x^{(m)}\to x$ in $\Bb$, whereas
$
P_xe_0=0,$
$P_{x^{(m)}}e_0=e_0$.
Consequently,
$
\|P_{x^{(m)}}-P_x\|=1.
$
In particular, the splitting is discontinuous in the operator-norm
topology on the projections.

The diffeomorphism $f$ satisfies generalized hyperbolicity for any $p \in [1, \infty]$;
and one can choose a trajectory for which the inclusion
$
Df(x_n)E^s_{x_n}\subset E^s_{x_{n+1}}
$
is strict for infinitely many values of $n$. 

The structure of the set of fixed points, however, depends on the value of $p$. The point $\{0\} \in \Bb$ is an unstable fixed point with $E^u_{\{0\}} = \Bb$.

\textbf{Case $p \in [1, \infty)$.} A point $q = \{q_k\}$ is a fixed point if $q_k \ne 0$ only for finitely many indices $k \in \Zz$ and for those indices $q_k = \pm 1$. For each of those points, the stable subspace $E^s_q$ is finite-dimensional with $\dim E^s_q = \#\{k: q_k \ne 0\}$. Note that the set of fixed points is countable,  and each of them has a finite-dimensional stable direction; however, the dimension of the stable subspace can be arbitrarily large. %Similarly, for any trajectory $\{x_n \in \Bb\}$ of $f$, the number of indices at which the inclusion $Df(x_n) E^s_{x_n} \subset E^s_{x_{n+1}}$ is strict is finite, but it can be arbitrarily large. 

\textbf{Case $p = \infty$.} A point $q = \{q_k\} \in \Bb$ is a fixed point if $q_k \in \{-1, 0, 1\}$ for all $k \in \Zz$. The stable subspace $E^s_q$ can be finite- or infinite-dimensional, with
$\dim E^s_q=\#\{k:q_k\ne0\}$, when this set is finite, and is
infinite-dimensional otherwise. The number of fixed points $q$ in this case is uncountable. 
%Note that it is possible to construct a trajectory $\{x_n \in \Bb\}$ of $f$ such that the  number of indices for which inclusion $Df(x_n) E^s_{x_n} \subset E^s_{x_{n+1}}$ is strict is infinite. 
%We find these properties quite interesting. 

For every $p\in[1,+\infty]$, all conclusions of Theorem~\ref{thm:main-results} hold for $f$.
At the same time, spaces $E^s_x$, $E^u_x$ do not depend continuously on $x$, so the assumptions of Theorem~\ref{thm:SS} are not satisfied. 
%We believe that the continuity of $E_x^{s, u}$ is not necessary for Theorem \ref{thm:SS}. 
%Example~\ref{ex:product}, especially for $p = +\infty$, which admits uncountably many fixed points, emphasizes the importance of relaxing the assumptions in Theorem~\ref{thm:SS}  (see Section~\ref{sec:discussion}).  

\begin{remark}\label{rem:minimal}
Examples~\ref{ex:shift} and~\ref{ex:product} isolate, respectively, strict invariant inclusions and a discontinuous, dimension-varying splitting.
The examples are not covered by the Banach-space frameworks of
\cite{SteinleinWalther1990,LaniWayda1995,Henry1981,SellYou}.
That the examples are rigid is therefore the point rather than a limitation: they are minimal witnesses.
Moreover, by Theorem~\ref{thm:main-results}\,(\ref{item:main-robust}) and Proposition~\ref{stat:difconj}, each generates a conjugacy-closed family that is open in the uniform $C^1$ distance within $\DUC(\Bb)$, so the phenomena they exhibit are stable, not exceptional.
\end{remark}

\subsection{Pushforward operators on function spaces}

Unlike the preceding separation examples, this subsection transfers generalized hyperbolicity from nonlinear maps to operators on function spaces and identifies the regularity obstruction for continuous and measurable vector fields. Shifted hyperbolic composition operators associated with north-south and Morse-Smale dynamics on $L^p$-spaces were previously constructed in \cite[Section~3]{Puj21}. The weighted construction below allows a nonconstant weight and also treats the space of continuous functions.

Let $g:\mathbb{S}^1 \to \mathbb{S}^1$ be  a north–south diffeomorphism  of the circle. Denote by $s, u \in \mathbb{S}^1$ its hyperbolic fixed points, where $s$ attracts and $u$ repels.
Let $a: \mathbb{S}^1 \to \Rr^+$ be a continuous function such that 
\begin{equation}\label{eq:3-3-1}
    a(s) < 1, \quad a(u) > 1.
\end{equation}
Since $\mathbb{S}^1$ is compact, there exists $R> 0$ such that $a(x) \in [1/R, R]$ for $x \in \mathbb{S}^1$.
\begin{example}[Weighted composition operator]\label{ex:coh}
Let $\Bb = L^p(\mathbb{S}^1, \Rr)$ with $p \in [1, +\infty]$, or $\Bb = C^0(\mathbb{S}^1, \Rr)$. Define a linear isomorphism $T:\Bb \to \Bb$, as 
\begin{equation}\label{eq:3-3-T}
T(v)(x) = w(x), \quad \mbox{with $w(g(x)) = a(x)v(x)$}.
\end{equation}
Then $T$ is generalized hyperbolic.   
\end{example}
\begin{proof}
    By \eqref{eq:3-3-1}, there exist $\lambda \in (0, 1)$ and interval neighborhoods $U_s$ of $s$ and $U_u$ of $u$ such that
\begin{equation}\label{eq:3-3-s}
    g(U_s) \subset U_s, \quad a(x) < \lambda \quad x \in U_s,
\end{equation}    
\begin{equation}\label{eq:3-3-u}
    g^{-1}(U_u) \subset U_u, \quad a(x) > 1/\lambda \quad x \in U_u.
\end{equation}   
%Let $\xi_1, \xi_2$ be two boundary points of $U_u$, and let $N > 1$ be such that $g^N(\xi_{1}) \in U_s$ and $g^N(\xi_{2}) \in U_s$. Set $V = \mathbb{S}^1 \setminus U_u$. 

\textbf{Case 1.} $\Bb = L^p(\mathbb{S}^1, \Rr)$ with $p \in [1, +\infty]$.
Set $b_p(x):=a(x)|g'(x)|^{1/p}$, and set $b_\infty:=a$. Since $s$ and $u$ are hyperbolic attracting and repelling fixed points, respectively,
$b_p(s)<1$, $b_p(u)>1$.
For $p<+\infty$, a change of variables gives
$$
|Tv|_p^p=\int_{\mathbb S^1} b_p(x)^p|v(x)|^p\,dx,
$$
so $b_p$ is the effective weight for the $L^p$ estimates.
Decreasing the neighborhoods $U_s,U_u$, while preserving the inclusions in \eqref{eq:3-3-s}, \eqref{eq:3-3-u}, we can assume that there exists $\lambda_p\in(0,1)$ such that
$$
b_p(x)<\lambda_p\quad(x\in U_s),
\qquad
b_p(x)>\lambda_p^{-1}\quad(x\in U_u).
$$
Let $\xi_1,\xi_2$ be the boundary points of $U_u$, set $V=\mathbb S^1\setminus U_u$, and choose $N>1$ such that $g^N(V)\subset U_s$.
Define the projections $P, Q: \Bb \to \Bb$ by
$$
Pv(x) = 
\begin{cases}
    v(x), & \quad x \in V, \\
    0, & \quad x \in U_u;
\end{cases}
\quad 
Qv(x) = 
\begin{cases}
    0, & \quad x \in V, \\
    v(x), & \quad x \in U_u,
\end{cases}
$$
and set $E^s := P(\Bb)$, $E^u := Q(\Bb)$.
Note that 
\begin{equation}\label{eq:3-3-incl}
g(V) \subset V, \quad g^{-1}(U_u) \subset U_u, \quad V \cup U_u = \mathbb{S}^1.
\end{equation}
Hence 
$TE^s \subset E^s$, $T^{-1}E^u \subset E^u$, $P+Q = Id$.
For a suitable global bound $R_p\geq1$ on $b_p$ and $b_p^{-1}$, set
$C_p=(R_p/\lambda_p)^N$.
The choice of $U_s,U_u,V,N$ and \eqref{eq:3-3-incl} imply that
\begin{equation}\label{eq:Pushs}
|T^n v^s| \leq C_p\lambda_p^{n}|v^s|, \quad v^s \in E^s, n > 0;    
\end{equation}
\begin{equation}\label{eq:Pushu}
|T^{-n} v^u| \leq C_p\lambda_p^{n}|v^u|, \quad v^u \in E^u, n > 0,
\end{equation}
and $T$ is generalized hyperbolic with the splitting $\Bb = E^s \oplus E^u$.

\textbf{Case 2.} $\Bb = C^0(\mathbb{S}^1, \Rr)$.
For this case, take $p=+\infty$ in the preceding choice, so that
$b_\infty=a$. Consider the continuous functions 
$
\Xi_1:[\xi_1, g(\xi_1)] \to \Rr$, $\Xi_2:[\xi_2, g(\xi_2)] \to \Rr, 
$  
such that
\begin{equation}\label{eq:Xi01}
0 \leq \Xi_i \leq 1, \quad \Xi_1(\xi_1) = 0, \; \Xi_1(g(\xi_1)) = 1, \quad  \Xi_2(\xi_2) = 0, \; \Xi_2(g(\xi_2)) = 1.  
\end{equation}
Define the projection $P:\Bb \to \Bb$ as follows
$$
Pv(x) = 
\begin{cases}
v(x), & x \in g(V), \\
v(g(\xi_1))\Xi_1(x), & x \in [\xi_1, g(\xi_1)], \\
v(g(\xi_2))\Xi_2(x), & x \in [\xi_2, g(\xi_2)], \\
0, & x \in U_u.
\end{cases}
$$
By \eqref{eq:Xi01} if $v \in \Bb$, then $Pv \in \Bb$. Define the projection $Q:\Bb \to \Bb$ by $Qv = v - Pv$. Set $E^s = P(\Bb)$, $E^u = Q(\Bb)$. 
Note that 
\begin{itemize}
    \item if $v^s \in E^s$ then $\supp v^s \subset V$, if $v^u \in E^u$ then $\supp v^u \subset \Cl(\mathbb{S}^1 \setminus g(V))$;
    \item if $\supp v^s \subset g(V)$ then $v^s \in E^s$, if $\supp v^u \subset U_u = \mathbb{S}^1 \setminus V$ then $v^u \in E^u$.
\end{itemize}
Hence spaces $E^s$, $E^u$ satisfy inclusions \eqref{eq:GenHypIncl}.
The same argument as in Case 1, with $p=+\infty$, proves \eqref{eq:Pushs}, \eqref{eq:Pushu}. 
Therefore, $T$ is generalized hyperbolic. 
\end{proof}
\begin{statement}[Pushforward to bounded functions]
\label{stat:pushforward-Linfty}
    Let $f \in \DUC(\Bb)$ satisfy generalized hyperbolicity. Let $X = B(\Bb, \Bb)$ be the Banach space of bounded maps $\Bb \to \Bb$ with the supremum norm. Consider the linear isomorphism $T_f: X \to X$ defined as follows
\begin{equation}\label{eq:3-3-Tf}
    T_fv(x) = w(x), \quad \mbox{with $w(f(x)) = Df(x)v(x)$}.    
\end{equation}
    Then $T_f$ is generalized hyperbolic.
\end{statement}
\begin{proof}
Inequalities \eqref{eq:C1-2} imply that $T_f$ is a bounded linear isomorphism. Let $C$, $\lambda$ be the corresponding constants, and let $P_x, Q_x: \Bb \to \Bb$ be the family of projections associated with the generalized hyperbolicity of $f$. Define the projections $P, Q: X \to X$ by the following
$$
Pv(x) = P_xv(x), \quad Qv(x) = Q_xv(x).
$$
Condition (GH1) implies that for any $v \in X$ we have $Pv, Qv \in X$ and $|Pv|, |Qv| \leq C|v|$. Set $E^s = P(X)$, $E^u = Q(X)$. Condition (GH2) implies inclusions \eqref{eq:GenHypIncl}. Condition (GH3) implies
$$
|T_f^nv^s|\leq C \lambda^n |v^s|, \quad n > 0, \; v^s \in E^s,
$$
$$
|T_f^{-n}v^u|\leq C \lambda^n |v^u|, \quad n > 0, \; v^u \in E^u.
$$
Therefore $T_f$ is generalized hyperbolic.
\end{proof}

%\begin{remark}[Regularity obstruction]
The pointwise construction in Proposition~\ref{stat:pushforward-Linfty} formally extends to other function spaces. 
On $X=C_b^0(\Bb,\Bb)$, however, the projections $Pv(x)=P_xv(x)$ and $Qv(x)=Q_xv(x)$ need not preserve continuity, while on $X=L^p(\Bb,\Bb)$ the same construction requires measurability of $x\mapsto P_x,Q_x$. Neither property is assumed. 
Thus the obstruction is regularity of the induced projections rather than the exponential estimates.
%\end{remark}

The same obstruction arises for pushforwards of vector fields. Let
$h:M\to M$ be a diffeomorphism of a smooth compact manifold, and let
$\Bb$ be either $L^p(M,TM)$, $p\in[1,+\infty]$, or $C_b^0(M,TM)$.
Define $S_h:\Bb\to\Bb$ by
\begin{equation}\label{eq:3-3-S}
S_h(v)(x) = w(x), \quad \mbox{with $w(h(x)) = Dh(x)v(x)$}.
\end{equation}
When $\Bb=C_b^0(M,TM)=C^0(M,TM)$, the operator $S_h$ is the classical Mather operator associated with $h$. Mather proved that $h$ is Anosov if and only if $S_h$ is a hyperbolic linear operator on $C^0(M,TM)$ \cite{Mather1968}. For connections between Axiom~A, the transversality condition, and surjectivity of $Id-S_h$, see \cite{Rob71,Mane1975,Mane1977,An2009}.
\begin{conj}\label{conj:AA}
   If $h$ satisfies Axiom A and the strong transversality condition, then $S_h$ is generalized hyperbolic as a linear operator on $\Bb = L^p(M, TM)$ for $p \in [1, + \infty]$, and also on $\Bb = C^0_b(M, TM)$.
\end{conj}
After completing this work, we became aware of an unpublished result announced by B.~Gollobit concerning the shadowing property of the Mather operator on continuous vector fields. The announced result is closely related to Conjecture~\ref{conj:AA}, but concerns shadowing rather than generalized hyperbolicity.

If $h$ is a north-south diffeomorphism, an argument similar to that in Example~\ref{ex:coh} shows that $S_h$ is generalized hyperbolic. More generally, Theorem~\ref{stat:finite-dimensional-characterization} provides the pointwise splittings and exponential estimates required in Conjecture~\ref{conj:AA}, so what remains is precisely the regularity obstruction described above, now for the projections induced on a space of vector fields.

%Although this paper focuses on the nonlinear setting, we expect the linear constructions \eqref{eq:3-3-T}, \eqref{eq:3-3-Tf}, and \eqref{eq:3-3-S} to be useful in future work; in particular, $T_f$ may help in addressing Conjecture~\ref{conj:inv-cocycle} in Section~\ref{sec:discussion}.

\section{Scheme of the proofs of main results}\label{sec:scheme}

Although the main object of the paper is a diffeomorphism satisfying generalized hyperbolicity, the proofs proceed through its linearizations along trajectories, described by sequences of operators, and through the associated cocycle.
Table~\ref{tab:notions} summarizes these companion notions and their roles in the proofs.

% ----------------------------------------------
% TABLE 1 : the family of generalized-hyperbolicity notions
% ----------------------------------------------
\begin{table}[!ht]
\small
\renewcommand{\arraystretch}{1.25}
\begin{tabularx}{\textwidth}{@{}l p{2.6cm} >{\raggedright\arraybackslash}X >{\raggedright\arraybackslash}X@{}}
\toprule
\textbf{Acts on} & \textbf{Defined by} & \textbf{Invariance (inclusions)} & \textbf{Role in the paper} \\
\midrule
Linear isomorphism $T$ &
Definition~\ref{def:genhyp} &
$TE^s\subseteq E^s$\newline $T^{-1}E^u\subseteq E^u$ &
Motivating linear case; new linear examples (Section~\ref{sec:Examples}). \\

Diffeomorphism $f$ &
(GH1)--(GH3), Def.~\ref{def:CLf} &
$Df(x)E^s_x\subseteq E^s_{f(x)}$\newline $D(f^{-1})(x)E^u_x\subseteq E^u_{f^{-1}(x)}$ &
Central object; standing hypothesis of Theorems~\ref{thm:main-results} and~\ref{thm:SS}. \\

Sequence $\{A_k\}$ &
(GH1*)--(GH3*), Def.~\ref{def:CLAk}  &
$A_kE^s_k\subseteq E^s_{k+1}$\newline $A_k^{-1}E^u_{k+1}\subseteq E^u_k$ &
Linearization $\{Df(x_k)\}$ along orbits; robustness engine (Lemmas~\ref{lem:SBS-periodic},~\ref{lem:CL-robust-Ak}). \\

Cocycle $(\alpha,A)$ &
(C-GH1)--(C-GH3), Def.~\ref{def:CL-cocycle} &
$A(x)E^s_x\subseteq E^s_{\alpha(x)}$\newline $A^{-1}(x)E^u_{\alpha(x)}\subseteq E^u_x$ &
Structural stability via the cocycle $(f,Df)$ (Lemmas~\ref{lem:SS1},~\ref{lem:SS2}). \\
\bottomrule
\end{tabularx}
\caption{The four notions of generalized hyperbolicity used in the paper and their roles in the proofs.}
\label{tab:notions}
\end{table}

The common linear mechanism in the proofs is uniform bounded solvability of the inhomogeneous equations
\[
v_{k+1}=A_kv_k+w_{k+1}.
\]
The bounded solution property requires that every uniformly bounded sequence $\{w_k\}$ admits a uniformly bounded  solution $\{v_k\}$. 
The strong bounded solution property requires, in addition, that such a solution can be chosen linearly and continuously on $\{w_k\}$; equivalently, the associated difference operator has a bounded linear right inverse.

Generalized hyperbolicity yields the strong bounded solution property for linearizations along exact trajectories. A perturbation result for this property supplies the bounded solution property for nearby sequences, which is the form used for linearizations along pseudotrajectories. Precise definitions are given in Section~\ref{sec:inh}.

To prove Theorem~\ref{thm:main-results}\,(\ref{item:main-LipSh})--(\ref{item:main-LipPerSh}), we first establish the finite Lipschitz shadowing property (Lemma~\ref{lem:FinLipSh}).

The proof is by induction on the length of a pseudotrajectory.
Assuming that shadowing has been constructed on $[-N,N]$, we use the strong bounded solution property for the differentials $B_k=Df(y_k)$ to extend it to $[-(N+1),N+1]$. 
This construction is based on the ideas developed by the author in \cite{TikhZapSem, MonTikh, TikhHolSh}.

The proofs of Theorem~\ref{thm:main-results}\,(\ref{item:main-LipSh}) and Theorem~\ref{thm:SS} form the technical core of the paper. 
The proof of Lipschitz shadowing has three main steps.

\begin{enumerate}
    \item Lemmas \ref{lem:SBS-robust} and \ref{lem:FinLipSh} imply that the differentials $B_k=Df(y_k)$ along every sufficiently accurate finite pseudotrajectory satisfy the bounded solution property with constant~$2L$.

    \item Lemma~\ref{lem:Bk-Inf} extends this property to pseudotrajectories indexed by $\Zz$, with constant $4L$.

    \item Lemma~\ref{lem:BS-SH} uses this bounded solvability to construct successively more accurate pseudotrajectories at summable distances from one another. 
    Their limit is an exact trajectory shadowing the original pseudotrajectory.
\end{enumerate}

For a periodic pseudotrajectory, Lipschitz shadowing and Lemma~\ref{lem:CL-robust-Ak} imply generalized hyperbolicity of the periodic sequence of differentials, although its splitting need not be periodic.
Lemma~\ref{lem:SBS-periodic} then gives the periodic bounded solution property, and the same correction procedure yields periodic shadowing. Density of periodic points in the chain-recurrent set follows from periodic shadowing, while robustness follows from the shadowing results and Lemma~\ref{lem:CL-robust-Ak} on robustness of generalized hyperbolicity for sequences of operators.

%In the proof of Theorem \ref{thm:SS}, we introduce generalized hyperbolicity for cocycles (Definition~\ref{def:CL-cocycle}) and construct an inverse operator by analogy with the strong bounded solution property (Lemma \ref{lem:SS1}). The construction of a semi-conjugacy \eqref{eq:conj-h1} follows the ideas of the Grobman-Hartman theorem. To establish semi-conjugacy \eqref{eq:conj-h2} we prove the robustness of generalized hyperbolicity for cocycles (Lemma \ref{lem:SS2}). Under the additional assumption that the subspace $E^s_x \oplus F_x$ is independent of $x$, the resulting semi-conjugacies are, in fact, inverse to each other. To prove this last assertion, we additionally show that, under small perturbations of the cocycle, the subspaces $E^s_x \oplus F_x$ can be preserved (Lemma \ref{lem:SS2}).

For Theorem~\ref{thm:SS}, we introduce generalized hyperbolicity for cocycles and construct an inverse operator analogous to that arising from the strong bounded solution property (Lemma~\ref{lem:SS1}).
A Grobman--Hartman-type argument produces one semi-conjugacy, while robustness of generalized hyperbolicity for cocycles produces the reverse one (Lemma~\ref{lem:SS2}).
The semi-conjugacies are constructed as near-identity maps whose corrections take values in $E_x^s\oplus F_x$. If these subspaces are independent of $x$, the admissible corrections remain in the same space under composition, and uniqueness for the corresponding conjugacy equations implies that the two semi-conjugacies are inverse to each other.

Figure~\ref{fig:roadmap} summarizes the logical dependencies between these arguments.

% ----------------------------------------------
% FIGURE : dependency diagram (TikZ)
% ----------------------------------------------
\begin{figure}[t]
\centering
\resizebox{\textwidth}{!}{%
\begin{tikzpicture}[
  box/.style={draw, rounded corners=1pt, align=center, font=\scriptsize,
              text width=2.55cm, inner sep=3pt, minimum height=0.95cm},
  hyp/.style={box, very thick, fill=black!12},
  lem/.style={box, fill=white},
  thm/.style={box, thick, fill=black!7},
  ext/.style={box, dashed, fill=white},
  ar/.style={-{Latex[length=1.8mm]}, draw=black!70, line width=0.6pt},
]
% -- hypothesis + external input (column 1) --
\node[hyp] (fgh)  at (0,3.0)   {Diffeo $f$ with gen. hyp. (Def.~\ref{def:CLf})};
\node[ext] (bern) at (0,-0.0)  {Linear shadowing\\ Thm~\ref{thm:Bernandes-shadowing}~\cite{NilsonPeris}};

% -- operator / cocycle lemmas (column 2) --
\node[lem] (sbs)    at (4.3,3.0)  {Lem~\ref{lem:SBS}\\ $\{Df(x_k)\}$ has the SBSP};
\node[lem] (sbsr)   at (4.3,1.6)  {Lem~\ref{lem:SBS-robust}\\ SBSP $\Rightarrow$ BSP under perturbation};
\node[lem] (bkinf)  at (4.3,0.0)  {Lem~\ref{lem:Bk-Inf}\\ finite BSP $\Rightarrow$ BSP on $\mathbb{Z}$};
%\node[lem] (sbsper) at (4.3,-1.6) {Lem~\ref{lem:SBS-periodic}\\ periodic BSP};
\node[lem] (clrob)  at (8.7,-0.8) {Lem~\ref{lem:CL-robust-Ak}\\ robustness of seq.\ gen.\ hyp.};
\node[lem] (ss1)    at (8.7,-3.0) {Lem~\ref{lem:SS1}\\ inverse of $T_{\alpha,A}$};
\node[lem] (ss2)    at (8.7,-4.6) {Lem~\ref{lem:SS2}\\ robustness of cocycle gen.\ hyp.};

% -- intermediate shadowing lemmas (column 3) --
\node[lem] (fin)  at (8.7,3.0) {Lem~\ref{lem:FinLipSh}\\ finite shadowing};
\node[lem] (bssh) at (8.7,1.4) {Lem~\ref{lem:BS-SH}\\ BSP $\Rightarrow$ LipSh};

% -- main theorems (column 4) --
\node[thm] (tlip)   at (13.0,3.0)  {Thm~\ref{thm:main-results}(\ref{item:main-LipSh})\\ Lipschitz shadowing};
\node[thm] (tper)   at (13.0,1.2)  {Thm~\ref{thm:main-results}(\ref{item:main-LipPerSh})\\ periodic shadowing};
\node[thm] (tchain) at (13.0,-0.4) {Thm~\ref{thm:main-results}(\ref{item:main-chain})\\ periodic orbits dense in $CR(f)$};
\node[thm] (trob)   at (13.0,-2.2) {Thm~\ref{thm:main-results}(\ref{item:main-robust})\\ robustness of gen.\ hyp.};
\node[thm] (tss)    at (13.0,-4.0) {Thm~\ref{thm:SS}\\ (semi-)structural stability};

% -- edges --
\draw[ar] (fgh)  -- (sbs);
\draw[ar] (bern) -- (bkinf);
\draw[ar] (sbs)  -- (fin);
%\draw[ar] (sbs)  -- (bssh);
\draw[ar] (sbsr) -- (bssh);
\draw[ar] (sbsr) -- (fin);
\draw[ar] (bkinf)-- (bssh);
\draw[ar] (fin)  -- (bssh);
\draw[ar] (bssh) -- (tlip);
\draw[ar] (tlip) -- (tper);
\draw[ar] (clrob)-- (tper);
%\draw[ar] (sbsper)-- (tper);
\draw[ar] (tper) -- (tchain);
\draw[ar] (clrob)-- (trob);
\draw[ar] (ss1)  -- (tss);
\draw[ar] (ss2)  -- (tss);
% skip-edges routed on the right of the theorem column
\draw[ar] (tlip.east) -- ++(0.75,0) |- (trob.east);
\draw[ar] (tper.east) -- ++(0.40,0) |- (trob.east);
\end{tikzpicture}}
\caption{Logical structure of the proofs. Thick, plain, shaded, and dashed boxes denote, respectively, the standing hypothesis, auxiliary lemmas, main theorems, and external input; arrows indicate dependence.}
\label{fig:roadmap}
\end{figure}

\section{Inhomogeneous linear equations}\label{sec:inh}

We now formalize the bounded-solvability mechanism described in Section~\ref{sec:scheme}.

\begin{defin}\label{def:bs}
    We say that a sequence of operators $\mathcal{A} = \{A_{k\in I}:\Bb \to \Bb\}$ ($I$ as in \eqref{eq:I}) with uniformly bounded norms $\|A_k\|, \|A_k^{-1}\| \leq R$ satisfies the \textit{bounded solution property on interval $I$ with constant $L>0$} if  for any sequence $\{w_{k+1} \in \Bb\}_{k \in \tilde{I}}$ with $|w_{k+1}| \leq 1$ there exists a sequence $\{v_k \in \Bb\}_{k \in I}$ with $|v_k| \leq L$ such that
    \begin{equation}\label{eq:Inh}
    v_{k+1} = A_k v_k + w_{k+1}, \quad k \in \tilde{I}.        
    \end{equation}
    If $I = \Zz$ we say that $\mathcal{A}$ satisfies the \textit{bounded solution property with constant $L>0$}.
\end{defin}

Equations \eqref{eq:Inh} are called \textit{inhomogeneous linear equations} \cite{Pli77, Coppel, Maizel1954, Palmer1984, Palmer1988, Slyusarchuk1983, ChiconeLatushkin1999, LRS-JDDE1998} and have been widely used in studies of the shadowing property \cite{Palmer2000, TikhLipSh, TikhHolSh, TikhPerSh, TikhZapSem, MonTikh, PriezzhevsTikh, Pilyugin2010, Todorov2013, Palmer1984, Palmer1988DR, BackesDragicevic2021, BD2019, BD2022}.
The bounded solution property implies the surjectivity of the operator 
$T_{\mathcal{A}}: \ell^{\infty}(I, \Bb) \to \ell^{\infty}(\tilde{I}, \Bb)$ defined by:
\begin{equation}\label{eq:TA}
T_{\mathcal{A}}\{v_k\}_{k \in I} = \{v_{k+1}-A_kv_k\}_{k \in \tilde{I}}.   
\end{equation}

\begin{defin}\label{def:sbs}
If there exists  a bounded linear right inverse $S_{\mathcal{A}}$ of the operator $T_{\mathcal{A}}$ defined by \eqref{eq:TA} with $I$ defined in \eqref{eq:I} and $\|S_{\mathcal{A}}\| \leq L$, we say that $\mathcal{A}$ has the \textit{strong bounded solution property on interval $I$ with constant $L > 0$}. 
\end{defin}

In the study of the bounded solution property, the notion of exponential dichotomy is widely used \cite{BDV2018,ChiconeLatushkin1999,Coppel,Henry94,Maizel1954,
Palmer1988,Slyusarchuk1983}; we recall it here for completeness of exposition.

Consider a two-sided sequence of bounded linear isomorphisms.
$$
\mathcal{A}=\{A_k: \Bb \to \Bb\}_{k\in\mathbb Z}, \qquad 
\|A_k\|,\;\|A_k^{-1}\|\le R.
$$
Denote the associated cocycle by
$$
\Phi_A(k,\ell)=\begin{cases}
 A_{k-1}\dots A_{\ell} & (\ell<k),\\[2pt]
 I & (\ell=k),\\[2pt]
 A_k^{-1}\dots A_{\ell-1}^{-1} & (\ell>k).
 \end{cases}
$$

\begin{defin}[Exponential dichotomy]\label{def:exp-dichotomy}
We say that the sequence $\mathcal{A}$ admits an \emph{exponential dichotomy} on $I$, where $I \in \{\Zz^+, \Zz^-, \Zz\}$, if there exist
constants $C>0$, $\lambda\in(0,1)$ and families of bounded projections
$P_k,Q_k:\Bb \to \Bb$ ($k\in\mathbb Z$), satisfying the following properties. Denote $E^s_k:=P_k(\Bb)$, $E^u_k:=Q_k(\Bb)$.
\begin{enumerate}
    \item[(ED1)] \textbf{Splitting.}  $P_k+Q_k=Id$, 
          $\|P_k\|,\|Q_k\|\le C$,  hence 
          $
          E^s_k\oplus E^u_k = \Bb
          $.
    \item[(ED2)] \textbf{Invariance.}  The splitting is carried exactly
          onto itself: $A_kE^s_k=E^s_{k+1}$, $A_k^{-1}E^u_{k+1}=E^u_{k}$ for $k\in I$.
          Equivalently $A_kP_k=P_{k+1}A_k$ and $A_kQ_k=Q_{k+1}A_k$.
    \item[(ED3)] \textbf{Exponential estimates.}  For $n \geq 0$ the following estimates hold   
          \begin{align*}
             &\|\Phi_A(k+n, k)  v^s\|\le C\lambda^n\|v^s\|,
               && v^s\in E^s_k, \; [k, k+n] \subset I, \\
             &\|\Phi_A(k-n, k) v^u\|\le C\lambda^n\|v^u\|,
               && v^u\in E^u_k, \; [k-n, k] \subset I .
          \end{align*}
\end{enumerate}
\end{defin}

\begin{defin}[Transversality condition]\label{def:transversality}
Assume that the restrictions of the sequence $\mathcal{A}$ to the future and to the past,
$
{\mathcal{A}}^{+}:=\{A_k\}_{k\ge 0}, {\mathcal{A}}^{-}:=\{A_k\}_{k\le 0},
$
each admit an exponential dichotomy with splittings
$E^{s,+}_k\oplus E^{u,+}_k=\Bb$ and
$E^{s,-}_k\oplus E^{u,-}_k=\Bb$, respectively.
We say that $\mathcal{A}$ satisfies the \emph{transversality condition} if
$
E^{s,+}_0\ + E^{u,-}_0 = \Bb.
$
\end{defin}

In finite dimensions, an exponential dichotomy on
$\mathbb Z_{\ge 0}$ and $\mathbb Z_{\le 0}$ together with the transversality
condition is equivalent to the existence of a bounded (hence continuous)
right inverse of the inhomogeneous difference operator
$T_{\cal A}\colon\ell^\infty(\mathbb Z,\Bb)\to\ell^\infty(\mathbb Z,\Bb)$,
$T_{\cal A}\{v_k\}=\{v_{k+1}-A_k v_k\}$, as proved in~\cite{Pli77, Coppel}.  This is precisely the
strong bounded solution property introduced in Definition~\ref{def:sbs}.  
For infinite-dimensional Banach spaces, a similar statement is not known in general. However, in many settings, an analog of this statement has been established under additional assumptions of uniqueness or of the possibility of making a canonical choice of the sequence \eqref{eq:Inh}; see, for instance, \cite{LRS-JDDE1998, SellYou, HuyMinhCMAA2001}.

In the following, we introduce the sequence-of-operators analog of generalized hyperbolicity (Definition~\ref{def:CLf}), which we believe is an appropriate analog of exponential dichotomy together with the transversality condition for the purposes of this paper.
\begin{defin}\label{def:CLAk}
    Let $\mathcal{A} = \{A_k: \Bb \to \Bb\}_{k \in \Zz}$ be a sequence of isomorphisms satisfying $\|A_k\|, \|A^{-1}_k\| \leq R$ for some $R>0$. We say that $\mathcal{A}$ satisfies \textit{generalized hyperbolicity} if there exist $C > 0$, $\lambda \in (0, 1)$ and  families of projections $P_k, Q_k: \Bb \to \Bb$ satisfying the following properties. Denote $E^s_k = P_k(\Bb)$, $E^u_k = Q_k(\Bb)$. For every $k \in \Zz$ the following holds
\begin{itemize}
    \item[(GH1*)] \textbf{Splitting.} $P_k + Q_k = Id$, $\|P_k\|, \|Q_k\| \leq C$. Note that $E^s_k \oplus E^u_k = \Bb$.
    \item[(GH2*)] \textbf{Inclusion.} The spaces $E^s_k$ and $E^u_k$ are forward- and backward-invariant, respectively\footnote{Compare with condition (ED2), the only difference is inclusion instead of equality.}
    $$
    A_kE^s_k \subseteq E^s_{k+1}, \quad A_k^{-1}E^u_{k+1} \subseteq E^u_{k}.
    $$
    \item[(GH3*)] 

    \textbf{Exponential estimates.}  For every $n\ge 0$    
          \begin{align*}
             &\|\Phi_A(k+n, k)  v^s\|\le C\lambda^n\|v^s\|,
               && v^s\in E^s_k, \\
             &\|\Phi_A(k-n, k) v^u\|\le C\lambda^n\|v^u\|,
               && v^u\in E^u_k.
          \end{align*}
\end{itemize}
    For a periodic sequence of operators with period $m$ (i.e. $A_{k+m} = A_k$), we say that $\mathcal{A}$ satisfies \textit{periodic generalized hyperbolicity} if projections $P_k$, $Q_k$ are also $m$-periodic, that is, $P_{k+m} = P_k$, $Q_{k+m} = Q_k$.
\end{defin}

Definitions~\ref{def:CLAk} and \ref{def:CLf} are related as follows.

%\begin{remark}
%    Definitions \ref{def:CLAk} and \ref{def:CLf} are similar. If a diffeomorphism $f:\Bb \to \Bb$ satisfies generalized hyperbolicity, then for any trajectory $\{x_k \in \Bb\}$ of $f$, the sequence $\{A_k = Df(x_k)\}$ satisfies generalized hyperbolicity. 
%\end{remark}

\begin{lem}[Orbitwise characterization]\label{lem:orbitwise}
Let $f \in \DUC(\Bb)$, $C > 0$, $\lambda \in (0,1)$. Then $f$ satisfies generalized
hyperbolicity with constants $C, \lambda$ if and only if
\begin{itemize}
\item for every nonperiodic trajectory $\{x_k\}_{k \in \Zz}$ of $f$, the sequence
$\{Df(x_k)\}_{k \in \Zz}$ satisfies generalized hyperbolicity with constants $C, \lambda$;
\item for every $m$-periodic trajectory $\{x_k\}_{k \in \Zz}$ of $f$, the sequence
$\{Df(x_k)\}_{k \in \Zz}$ satisfies $m$-periodic generalized hyperbolicity with constants
$C, \lambda$.
\end{itemize}
\end{lem}
\begin{proof}
If $f$ satisfies generalized hyperbolicity, set $E^{s,u}_k := E^{s,u}_{x_k}$. Then
{\rm(GH1*)}--{\rm(GH3*)} are {\rm(GH1)}--{\rm(GH3)} at the points $x_k$, and if
$x_{k+m} = x_k$ the splitting is automatically $m$-periodic.

Conversely, the trajectories of $f$ partition $\Bb$. Choose a point $y_\omega$ in each
trajectory $\omega$ and let $\{E^s_k, E^u_k\}$ be a splitting as in the hypothesis for the
sequence $\{Df(f^k(y_\omega))\}$. The definition $E^{s,u}_{f^k(y_\omega)} := E^{s,u}_k$ is
unambiguous: for a nonperiodic $\omega$ the points $f^k(y_\omega)$ are pairwise distinct,
and for an $m$-periodic $\omega$ this is exactly the $m$-periodicity of the splitting.
Conditions {\rm(GH1)}--{\rm(GH3)} then hold with constants $C, \lambda$.
\end{proof}

Let $\mathcal{A}$ be as in Definition~\ref{def:CLAk}. 
Suppose $\mathcal{A}$ admits exponential dichotomies on $\Zz^+$ and on $\Zz^-$ and satisfies the transversality condition. In finite dimensions this implies that $\mathcal{A}$ satisfies generalized hyperbolicity; in fact, the two notions coincide, this being the $(C,\lambda)$-structure of \cite{PilCL}. In infinite dimensions, the forward implication is considerably more delicate and, unlike the finite-dimensional case, is not a routine adaptation of the classical argument; we return to it in Section~\ref{sec:discussion}. We do not use this direction in the proofs. In infinite dimensions the two notions no longer coincide: generalized hyperbolicity is strictly weaker, since the sequence of Example~\ref{ex:CL-not-ED} satisfies generalized hyperbolicity but admits no exponential dichotomy on either $\Zz^+$ or $\Zz^-$.

To prove Lemma \ref{lem:FinLipSh} and Theorem~\ref{thm:main-results}\,(\ref{item:main-LipSh}), we use the following statements that describe the properties of bounded solutions.

\begin{lem}\label{lem:SBS}
If $f \in \DUC(\Bb)$ satisfies generalized hyperbolicity, then there exists $L>0$ such that, for any exact trajectory $\{x_k\}$ and interval $I$ (finite or infinite), the associated sequence of linear operators $\mathcal{A} = \{A_k = Df(x_k)\}_{k \in I}$ satisfies the strong bounded solution property on $I$ with constant $L$. 
\end{lem}

\begin{lem}\label{lem:SBS-robust}
For any $L > 0$ there exists $\varepsilon > 0$ such that the following holds. Let $I$ be an interval (finite or infinite), and let $\mathcal{A} = \{A_k\}_{k \in I}$ be a sequence satisfying the strong bounded solution property on $I$ with constant $L$. If $\mathcal{B} = \{B_k\}$ is another sequence of uniformly bounded isomorphisms satisfying 
$\|B_k - A_k\| \leq \varepsilon$,
then it satisfies the bounded solution property on $I$ with constant~$2L$.
\end{lem}

\begin{lem}\label{lem:Bk-Inf}
    Let $\mathcal A=\{A_k\}_{k\in\Zz}$ be a sequence of linear isomorphisms such that $\|A_k\|,\|A_k^{-1}\|\leq R$.
 Assume that there exists $L>0$ such that for any finite interval $I = (a, b)$ the sequence $\{A_k\}_{k \in I}$ satisfies the bounded solution property on $I$ with constant $L$. Then the complete sequence $\{A_k\}_{k \in \Zz}$ satisfies the bounded solution property on $\Zz$ with the constant $2L$.
\end{lem}

\begin{lem}\label{lem:SBS-periodic}
Assume that $\mathcal{A} = \{A_k\}_{k \in \Zz}$ is $m$-periodic ($A_{k+m}=A_k$, $k \in \Zz$) and satisfies generalized hyperbolicity with constants $C, \lambda$, with a not necessarily periodic splitting. Then there exists $L = L(C, \lambda)>0$ such that for every bounded $m$-periodic sequence 
$w=\{w_k\}_{k\in\mathbb{Z}}\subset\mathbb{B}$ with $|w_k| \leq 1$ there exists an $m$-periodic sequence $v=\{v_k\}_{k\in\mathbb{Z}}\subset\mathbb{B}$, satisfying 
\begin{equation}\label{eq:inhom-diff}
v_{k+1}=A_k v_k + w_{k+1}, \quad |v_k| \leq L \qquad k\in\mathbb{Z}.
\end{equation}
\end{lem}

We prove the robustness of generalized hyperbolicity for sequences of operators.
\begin{lem}\label{lem:CL-robust-Ak}
Assume that $\mathcal{A}=\{A_k\}$ satisfies generalized hyperbolicity
with constants $C>0$, $\lambda\in(0,1)$ and $\|A_k\|,\|A_k^{-1}\|\leq R$, for $k\in\Zz$. 
Then for any $\lambda_1 \in (\lambda, 1)$, there exist $C_1 = C_1(C, \lambda, \lambda_1, R)$ and $\varepsilon > 0$ such that 
    \begin{itemize}
        \item[(Robust1)]  any sequence $\mathcal{B} = \{B_k\}_{k \in \Zz}$, where 
\begin{equation}\label{eq:AkBk}
        \|B_k - A_k\| < \varepsilon, \quad k \in \Zz  
\end{equation}
        also satisfies generalized hyperbolicity with constants $C_1$, $\lambda_1$;
        \item [(Robust2)] if, in addition, the sequence $\mathcal{A}$ is $m$-periodic and satisfies periodic generalized hyperbolicity, then any $m$-periodic sequence $\mathcal{B} = \{B_k\}_{k \in \Zz}$ satisfying \eqref{eq:AkBk} also satisfies periodic generalized hyperbolicity with constants $C_1$, $\lambda_1$.
    \end{itemize}
\end{lem}
\begin{remark}
We expect that for sufficiently small $\varepsilon$, the sequence $\mathcal{B}$ satisfies generalized hyperbolicity for any prescribed $C_1 > C$, $\lambda_1 \in (\lambda, 1)$. However, the proof of Lemma \ref{lem:CL-robust-Ak} presented in this paper does not yield an optimal control of the constant $C_1$.
\end{remark}

\subsection{Example of a sequence of linear operators}\label{sec:cl-not-ed}

We provide an example of a sequence $\mathcal{A}$ that satisfies generalized hyperbolicity, but does not admit an exponential dichotomy. This example generalizes \cite[Example 1.13]{PilBook}, which was introduced in a slightly different context.
\begin{example}\label{ex:CL-not-ED}
    Let $\Bb = \ell^p(\Zz)$, $p \in [1, \infty]$. Consider a sequence of operators $\mathcal{A} = \{A_k: \Bb \to \Bb\}_{k \in \Zz}$ defined as
    $$
(A_k x)_m =
\begin{cases}
\dfrac{1}{2} x_m, & m \leq k, \\
2 x_m, & m > k.
\end{cases}    
    $$
    Then $\mathcal{A}$ satisfies generalized hyperbolicity, but does not admit an exponential dichotomy on $\Zz^+$ or $\Zz^-$.
\end{example}
\begin{proof}
    Consider the projections
    $$
    (P_k x)_m =
\begin{cases}
x_m, & m \leq k, \\
0,   & m > k,
\end{cases}
\quad
    (Q_k x)_m =
\begin{cases}
0, & m \leq k, \\
x_m,   & m > k.
\end{cases}
    $$
A straightforward check shows that $\mathcal{A}$ satisfies generalized hyperbolicity with projections $P_k$, $Q_k$, and $C=1$, $\lambda = 1/2$.

%Assume that $\mathcal{A}$ admits an exponential dichotomy on $\Zz^+$ with projections $\tilde P_k$, $\tilde Q_k$ and constants $C > 0$, $\lambda \in (0, 1)$. Denote $\tilde E^s_0 = \tilde P_0\Bb$. Then $\tilde E^s_0$ is closed and
%$$
%\tilde E^s_0 = \left\{v: \lim_{n \to +\infty} \left|\Phi_A(n, 0)v\right| = 0 \right\}.
%$$
%For $m \in \Zz$ denote  $v_m = \{x_k\}$, with $x_m = 1$, $x_k = 0$ for $k \ne m$.
%Then $\lim_{n \to +\infty} \left|\Phi_A(n, 0)v_m\right| = 0$ and therefore $v_m \in \tilde E^s_0$. 
%On the other hand, for $m < 0$ we have $|\Phi_A(m, 0)v_m| = 2^{|m|}$. However, the exponential dichotomy estimate gives $|\Phi_A(m, 0)v_m| \leq C\lambda^{|m|}$, which leads to a contradiction for $|m|$ sufficiently large. 
%Hence $\mathcal{A}$ does not admit an exponential dichotomy on $\Zz^+$. Similarly, $\mathcal{A}$ does not admit an exponential dichotomy on~$\Zz^-$.

Assume that $\mathcal{A}$ admits an exponential dichotomy on $\Zz^+$ and denote its stable subspace at zero by $\tilde E^s_0$. Then $\tilde E^s_0$ is closed and
$$
\tilde E^s_0=
\left\{v\in\Bb:
\lim_{n\to+\infty}|\Phi_A(n,0)v|=0
\right\}.
$$
Denote by $e_m$ the standard basis vectors in $\ell^p(\Zz)$ and set
$
v^{(N)}=\sum_{m=1}^N2^{-m}e_m.
$
Every $v^{(N)}$ has finite support and
$
|\Phi_A(n,0)v^{(N)}|\to0
$
as $n\to+\infty$. Hence $v^{(N)}\in\tilde E^s_0$. Moreover,
$
v^{(N)}\to v:=\sum_{m=1}^{+\infty}2^{-m}e_m
$
in $\ell^p(\Zz)$ for every $p\in[1,+\infty]$. However, the $n$-th coordinate of $\Phi_A(n,0)v$ is equal to $1$, and therefore
$
|\Phi_A(n,0)v|\geq1
$
for every $n>0$. Thus $v\notin\tilde E^s_0$, contradicting the closedness of $\tilde E^s_0$. Hence $\mathcal{A}$ does not admit an exponential dichotomy on $\Zz^+$. Similarly, $\mathcal{A}$ does not admit an exponential dichotomy on $\Zz^-$.
\end{proof}

\subsection{Bounded solution property}\label{sec:BS}

\begin{proof}[Proof of Lemma \ref{lem:SBS}]
For an interval $I$ of the form \eqref{eq:I} the operator $S_{\mathcal{A}}$ can be defined explicitly: $S_{\mathcal{A}}\{w_{k+1}\}_{\tilde{I}} = \{v_k\}_{k \in I}$,
where $v_k$ is given by Perron sums
\begin{equation}\label{eq:Perron}
v_k = \sum_{i \leq k} D(f^{k-i})(x_i) P_{x_i}w_i -\sum_{i>k}D(f^{k-i})(x_i)Q_{x_i}w_i,   
\end{equation}
with the convention that $w_i = 0$ for $i \notin I$.
By construction the operator $S_{\mathcal{A}}$ is linear.
Due to~(GH1),~(GH3)
$$
|v_k| \leq \sum_{i \leq k} C \lambda^{k-i}C|w_i| + \sum_{i>k} C \lambda^{i-k}C|w_i| \leq C^2 \frac{1+\lambda}{1-\lambda}\|\{w_k\}\|,
$$
and hence $\|S_{\mathcal{A}}\| \leq C^2 \frac{1+\lambda}{1-\lambda}$. Conditions (GH1), (GH2) imply that $\{v_k\}$ satisfies \eqref{eq:Inh} and hence 
$
T_{\mathcal{A}} \circ S_{\mathcal{A}} \{w_{k+1}\}_{k \in \tilde{I}} = \{w_{k+1}\}_{k \in \tilde{I}}.
$
\end{proof}

\begin{proof}[Proof of Lemma \ref{lem:SBS-robust}]
Consider $\varepsilon < \frac{1}{2L}$, operator $T_{\mathcal{A}}$ and its right inverse $S_{\mathcal{A}}$:
$$
T_{\mathcal{A}} \circ S_{\mathcal{A}} = Id, \quad \|S_{\mathcal{A}}\| \leq L.
$$
Let $B_k = A_k + \Delta_k$, then equation
$
v_{k+1} - B_kv_k = w_{k+1}$, for $k \in \tilde{I}
$
is equivalent to
\begin{equation}\label{eq:TB1}
T_{\mathcal{A}}\{v_k\} - \{\Delta_k v_k\} = \{w_{k+1}\}, \quad k \in \tilde{I}.    
\end{equation}
Consider the operator $\Delta: \ell^{\infty}(I, \Bb) \to \ell^{\infty}(\tilde{I}, \Bb)$ defined by
$\Delta \{v_k\}_{k \in I} = \{w'_{k+1} = \Delta_k v_k\}_{k \in \tilde{I}}$. Note that $\| \Delta \| \leq \varepsilon$. 
Then any solution of the following equation is also a solution of \eqref{eq:TB1}
\begin{equation}\label{eq:TB2}
\{v_k\} - S_{\mathcal{A}}\Delta \{v_k\} = S_{\mathcal{A}}\{w_{k+1}\}.
\end{equation}
Note that 
\begin{equation}\label{eq:TB3}
\|S_{\mathcal{A}}\Delta\|\leq L\varepsilon < \frac{1}{2}.
\end{equation}
Then the sequence
$$
\{v_k\} = (Id - S_{\mathcal{A}}\Delta)^{-1}S_{\mathcal{A}}\{w_{k+1}\}, \quad \mbox{where} \quad (Id - S_{\mathcal{A}}\Delta)^{-1} = Id +\sum_{l \geq 1} (S_{\mathcal{A}}\Delta)^l
$$
is a solution of \eqref{eq:TB2}. Due to \eqref{eq:TB3} the series converges and
$$
\|\{v_k\}\| \leq \left(1+\frac{\varepsilon L}{1-\varepsilon L}\right)L \|\{w_k\}\| \leq 2L.
$$
\end{proof}

\begin{proof}[Proof of Lemma \ref{lem:Bk-Inf}]
We use the following statement for linear isomorphisms of Banach spaces.
\begin{thm}[{\cite[Theorem 1]{NilsonPeris}}]\label{thm:Bernandes-shadowing}
    For any invertible continuous linear operator $T$ on any Banach space $X$, if $T$ has the finite Lipschitz shadowing property with constant $L$, then $T$ has the Lipschitz shadowing property with constant $2L$.
\end{thm}
\begin{remark}
    The statement of Theorem 1 in \cite{NilsonPeris} gives only the shadowing property without specifying the constant. However, the proof can be repeated to establish the constant $2L$. For completeness of exposition we provide the proof tracing the constant in Appendix A.
\end{remark}

Consider the Banach space$
\Xop =\ell^{\infty}(\Zz, \Bb)=\bigl\{\{v_n\}_{n\in\Zz}: v_n\in \Bb \bigr\}
$
and a linear operator $\Aop \colon \Xop \to \Xop$ defined by
$$
\Aop\{v_n\}=\{v_n'\}, \qquad \text{where }\; v_{n+1}'=A_n v_n.
$$

Let us prove that under the assumptions of Lemma \ref{lem:Bk-Inf}, the operator $\Aop$ has the finite shadowing property. 

Consider a finite interval $[k_1, k_2]$ and sequences $\xi^k\in \Xop$, $k \in [k_1, k_2]$, with $|\xi^k|\leq1$; denote by $\xi_n^k \in \Bb$ ($k\in[k_1,k_2]$, $n\in\Zz$) their elements.
Below we show that there exists a sequence $\nu^k \in \Xop$, $k\in[k_1,k_2]$ with $|\nu^k| \leq L$ such that
\begin{equation}\label{eq:nu}
  \nu^{k+1}=\Aop\nu^k+\xi^k,\quad k \in [k_1, k_2-1].
\end{equation}
Denote the elements of $\nu^k$ by $\nu^k_n \in \Bb$, i.e. $\nu^k = \{\nu^k_n\}$.
Equations \eqref{eq:nu} are equivalent to
\begin{equation}\label{eq:1-inf}
  \nu_{n+1}^{k+1}=A_n \nu_n^k+\xi_{n+1}^k, \qquad n \in \Zz,\; k \in [k_1,k_2-1].
\end{equation}

The idea is to read the doubly-indexed system \eqref{eq:1-inf} along the diagonals $k-n=j$: on each of them it becomes a finite inhomogeneous difference equation in the single index $n$ for the operators $\{A_n\}$, to which the finite bounded solution property applies.
Denote $j := k-n$ and consider the sequences
\begin{equation}\label{eq:1.5}
  \mu_n^j := \nu_n^{j+n},\qquad \theta_n^j := \xi_n^{j+n-1}.
\end{equation}

Then \eqref{eq:1-inf} is equivalent to
\begin{equation}\label{eq:2}
  \mu_{n+1}^j = A_n\mu_n^j+\theta_{n+1}^j,\qquad n\in\Zz,\; n+j\in[k_1,k_2-1].
\end{equation}
Rewriting indexes we have
\begin{equation}\label{eq:2'}
  \mu_{n+1}^j = A_n\mu_n^j+\theta_{n+1}^j,\qquad j\in\Zz,\; n\in[k_1-j,\,k_2-j-1].
\end{equation}

For any $j\in\Zz$, according to \eqref{eq:Inh} there exists a sequence $\mu^j_n$ satisfying \eqref{eq:2'}, and $|\mu_n^j| \le L$ for $n\in[k_1-j,k_2-j]$.
Hence, for all $j\in\Zz$ and $n+j\in[k_1,k_2-1]$, \eqref{eq:2} holds, which means that $\nu_n^k$ defined by \eqref{eq:1.5} satisfies \eqref{eq:1-inf}. Hence, the operator $\Aop$ has the finite Lipschitz  shadowing property with constant $L$. According to Theorem \ref{thm:Bernandes-shadowing}, this implies that $\Aop$ has the Lipschitz shadowing property with constant $2L$.

Let us now show that the sequence $\mathcal{A}$ satisfies the bounded solution property for $I = \Zz$. Indeed, fix $\{w_k \in \Bb\}$ with $|w_k| \leq 1$. Consider the sequence $\xi^k \in \Xop$ defined as 
$$
\xi^k = \{ \dots, w_{k+1}, w_{k+1}, w_{k+1}, \dots\}, \; \xi^k_n = w_{k+1}, \quad n\in \Zz.
$$
Note that $|\xi^k| = |w_{k+1}| \leq 1$. Since $\Aop$ has the shadowing property, there exists a sequence $\nu^k \in \Xop$ with $|\nu^k| \leq 2L$ such that
\begin{equation}\label{eq:nu-inf}
  \nu^{k+1}=\Aop\nu^k+\xi^k,\quad k \in \Zz.
\end{equation}
Denote $\nu^k = \{\nu^k_n\}$, then $|\nu^k_n| \leq 2L$. According to \eqref{eq:nu-inf}, the sequence  $v_k = \nu^k_k$ satisfies \eqref{eq:Inh} for $I = \Zz$ and hence $\mathcal{A}$ satisfies the bounded solution property with constant $2L$ for $I = \Zz$.
\end{proof}

\begin{proof}[Proof of Lemma \ref{lem:SBS-periodic}]

For $i\in\Zz$, let $r(i)\in\{0,\ldots,m-1\}$ be its residue modulo $m$.
Periodicity gives
$
\Phi_A(k+m,\ell+m)=\Phi_A(k,\ell)
$
and, writing $i=r(i)+qm$, we have
\begin{equation}\label{eq:periodic-shift}
\Phi_A(k,i)=\Phi_A(k-i+r(i),r(i)).
\end{equation}
Consider the sequence
\begin{equation}\label{eq:green-series}
v_k:=
\sum_{i \leq k}\Phi_A(k,i)P_{r(i)}w_{i}
-\sum_{ i > k}\Phi_A(k,i)Q_{r(i)}w_{i}, 
\quad k \in \Zz.
\end{equation}
Although the splitting need not be periodic, \eqref{eq:periodic-shift} allows us to use the projections at the representative indices $r(i)$.
By \eqref{eq:periodic-shift} and {\rm(GH1*)}, {\rm(GH3*)},
\[
\|\Phi_A(k,i)P_{r(i)}\|
\leq C^2\lambda^{k-i},
\qquad i\leq k,
\]
\[
\|\Phi_A(k,i)Q_{r(i)}\|
\leq C^2\lambda^{i-k},
\qquad i>k.
\]
Hence the series in \eqref{eq:green-series} converge.  
Summing the geometric tail yields $|v_k| \leq L = C^{2}\frac{1+\lambda}{1-\lambda}$.
Changing the summation index by $m$ and using
\[
\Phi_A(k+m,i+m)=\Phi_A(k,i),\qquad
r(i+m)=r(i),\qquad w_{i+m}=w_i,
\]
gives $v_{k+m}=v_k$. Moreover, using
$
A_k\Phi_A(k,i)=\Phi_A(k+1,i)
$
and $P_{r(k+1)}+Q_{r(k+1)}=Id$, direct cancellation in
\eqref{eq:green-series} gives
$
v_{k+1}=A_kv_k+w_{k+1}.
$
\end{proof}

\subsection{Robustness for sequences of operators}\label{sec:robust-Ak}

\begin{proof}[Proof of Lemma \ref{lem:CL-robust-Ak}]
    Let $B_k = A_k + \Delta_k$. Denote 
    $$
    A_k = \begin{bmatrix}
A_k^{ss} & A_k^{su} \\
A_k^{us} & A_k^{uu}
\end{bmatrix}
, \quad 
    \Delta_k = \begin{bmatrix}
\Delta_k^{ss} & \Delta_k^{su} \\
\Delta_k^{us} & \Delta_k^{uu}
\end{bmatrix},
    $$
where 
\begin{align*}
A_k^{ss}: E_k^s \to E_{k+1}^s, & \quad  A_k^{ss} = P_{k+1}A_k;& \qquad A_k^{su}: E_k^u \to E_{k+1}^s, & \quad  A_k^{su} = P_{k+1}A_k;\\
A_k^{us}: E_k^s \to E_{k+1}^u, & \quad  A_k^{us} = Q_{k+1}A_k;& \qquad
A_k^{uu}: E_k^u \to E_{k+1}^u, & \quad  A_k^{uu} = Q_{k+1}A_k,
\end{align*}
The operators $\Delta_k^{ss}$, $\Delta_k^{su}$, $\Delta_k^{us}$, and $\Delta_k^{uu}$ are defined in the same way, replacing $A_k$ by $\Delta_k$.

By (GH2*) we have $A_k^{us}=0$, and therefore
    $$
    A_k = \begin{bmatrix}
A_k^{ss} & A_k^{su} \\
0 & A_k^{uu}
\end{bmatrix}
$$

Denote 
$
F_k^u = A_k^{-1} E_{k+1}^u.
$
By (GH2*) we have $F_k^u \subseteq E_k^u$, and the map $A_k:F_k^u \to E_{k+1}^u$ is a bijection. Therefore, it has an inverse; denote it by $Z_k: E_{k+1}^u \to F_k^u$,
$$
Z_kA_k^{uu} = Id|_{F_k^u}, \quad A_k^{uu}Z_k = Id|_{E_{k+1}^u}, \quad Z_k = A_k^{-1}|_{E^u_{k+1}}.
$$
Let us first prove item \textit{(Robust1)}.
Denote $L' = 1+ C^2 \frac{\lambda^2}{1-\lambda^2}$. Fix  $\varepsilon<1/(2R)$ sufficiently small so that, for
\begin{equation}\label{eq:defep1ep2}
\varepsilon_1 = C\varepsilon, \quad \varepsilon_2 = 2L'R\varepsilon_1
\end{equation}
the following inequality holds
\begin{equation}\label{eq:condep1ep2}
    R(2\varepsilon_1 + 2(R+\varepsilon_1)\varepsilon_2 + (R+\varepsilon_1)2\varepsilon_2) \leq \frac{1}{2L'}.
\end{equation}

Since
$
\|A_k^{-1}(B_k-A_k)\|<\frac{1}{2},
$
each $B_k$ is an isomorphism and
$
\|B_k^{-1}\|\leq 2R.
$

In the following we show that there exists a sequence of linear maps
    $H_k: E_k^s \to F_k^u$, $\|H_k\| \leq \varepsilon_2$,
such that for linear spaces
\begin{equation}\label{eq:tildeEDef}
\tilde{E}_k^s = \{v_k^s+H_k v_k^s: \; v_k^s \in E_k^s\}    
\end{equation}
holds the inclusion
\begin{equation} \label{eq:ETildaIncl}
B_k \tilde{E}_k^s \subseteq \tilde{E}_{k+1}^s.    
\end{equation}
\begin{remark}
    In the standard robustness proof for exponential dichotomy, the latter inclusion is replaced by an equality, and the operators $H_k$ are assumed to act on $E_k^u$. These differences lead to the corresponding changes in the arguments.
\end{remark}
Consider the vectors $v_k^s \in E_k^s$ and $\tilde{v}_k^s = v_k^s + H_k v_k^s \in \tilde{E}_k^s$. The following holds
\begin{multline*}
    B_k \tilde{v}_k^s = \left(\begin{bmatrix}
A_k^{ss} & A_k^{su} \\
0 & A_k^{uu}
\end{bmatrix} 
+ \begin{bmatrix}
\Delta_k^{ss} & \Delta_k^{su} \\
\Delta_k^{us} & \Delta_k^{uu}
\end{bmatrix} 
\right)
\begin{bmatrix}
v_k^s  \\
H_k v_k^s  
\end{bmatrix} = 
\begin{bmatrix}
A_k^{ss}v_k^s + A_k^{su}H_kv_k^s + \Delta_{k}^{ss}v_k^s + \Delta_k^{su} H_k v_k^s  \\
A_k^{uu}H_kv_k^s + \Delta_k^{us}v_k^s + \Delta_k^{uu}H_k v_k^s
\end{bmatrix}.
\end{multline*}

Inclusion $B_k \tilde{v}_k^s \in \tilde{E}_{k+1}^s$ is equivalent to equality 
$$
(H_{k+1}A_k^{ss} + H_{k+1}A_k^{su}H_k + H_{k+1}\Delta_k^{ss} + H_{k+1}\Delta_k^{su}H_k)v_k^s = 
(A_k^{uu}H_k + \Delta_k^{us} + \Delta_k^{uu}H_k)v_k^s.
$$
\begin{statement}
    If a sequence of operators $H_k: E_k^s \to F_k^u$ satisfies
\begin{equation}\label{eq:StHk}
        A_k^{uu}H_k = H_{k+1}A_k^{ss} - \Delta_k^{us} - \Delta_k^{uu}H_k + H_{k+1}A_k^{su}H_k + H_{k+1}\Delta_k^{ss} + H_{k+1}\Delta_k^{su}H_k
\end{equation}
    then inclusions \eqref{eq:ETildaIncl} hold.
\end{statement}
Consider the following auxiliary problem. Given a sequence of bounded operators $S_k: E_k^s \to F_k^u$, we seek a bounded sequence of operators $H_k: E_k^s \to F_k^u$ that satisfies
\begin{equation}\label{eq:HkSk}
H_k = Z_kH_{k+1}A_k^{ss} + S_k.
\end{equation}
Let $X$ be the Banach space of such operator sequences $\{S_k\}_{k \in \Zz}$ endowed with the norm
$\|\{S_k\}\| = \sup_k \|S_k\|$. Consider the operator $F:X\to X$ defined by
\begin{equation}\label{eq:defFHk}
F(\{S_k\}) = \{H_k\}, \quad \mbox{where} \quad H_k = S_k + \sum_{l = 1}^{+\infty}Z_k\cdot \ldots \cdot Z_{k+(l-1)}S_{k+l}A_{k+(l-1)}^{ss} \cdot \ldots \cdot A_k^{ss}.    
\end{equation}
Due to condition (GH3*) for $l \geq 0$, the following inequalities hold
$$
    \|Z_k\cdot \ldots \cdot Z_{k+(l-1)}\|   \leq C \lambda^l, \quad
    \|A_{k+(l-1)}^{ss}\cdot \ldots \cdot A_k^{ss}\|   \leq C \lambda^l.
$$
Hence, the right-hand side of \eqref{eq:defFHk} converges and
\begin{equation}\label{eq:FL'}
\|F(\{S_k\})\| \leq \|\{S_k\}\| \left(1+ \sum_{l=1}^{\infty} C^2\lambda^{2l} \right) = \|\{S_k\}\|\left(1+ C^2 \frac{\lambda^2}{1-\lambda^2}\right) = L' \|\{S_k\}\|.    
\end{equation}
A straightforward substitution shows that $H_k$, $S_k$ satisfy \eqref{eq:HkSk}.

Consider the nonlinear operator
$Q:X \to X$ defined as
$Q(\{H_k\}) = \{S_k\}$, where
$$
S_k = Z_k(- \Delta_k^{us} - \Delta_k^{uu}H_k + H_{k+1}A_k^{su}H_k + H_{k+1}\Delta_k^{ss} + H_{k+1}\Delta_k^{su}H_k).
$$
Note that if $\{H_k\} \in X$ is a fixed point of $F \circ Q$, i.e. $\{H_k\} = F \circ Q (\{H_k\})$, then $\{H_k\}$ satisfies~\eqref{eq:StHk}.

Let $\varepsilon_2$ satisfy the conditions \eqref{eq:condep1ep2}. Let $U$ be the closed ball in $X$ of radius $\varepsilon_2$ centered at $0$. In the following, we show that $F \circ Q : U \to U$ and $F \circ Q|_{U}$ is a contraction.
Consider $H = \{H_k\}, H' = \{H'_k\} \in U$, and assume that
$
\|H_k\|, \|H'_k\| \leq \varepsilon_2$, $\|H-H'\| = \delta
$.
Let us estimate $Q\{H_k\} - Q\{H'_k\} = \{S_k - S'_k\}$. The following holds
\begin{multline*}
    \|S_k - S'_k\| \leq 
    \|Z_k\| \cdot \|-\Delta_k^{uu}(H_k-H_k') + H_{k+1}A_k^{su}H_k -H'_{k+1}A_k^{su}H'_k + \\ (H_{k+1}-H'_{k+1})\Delta_k^{ss} + H_{k+1}\Delta_k^{su}H_k - H'_{k+1}\Delta_k^{su}H'_k\| \leq \\
    R(2\varepsilon_1 \delta + (H_{k+1}-H'_{k+1})(A_k^{su}+\Delta_k^{su})H_k +  H_{k+1}(A_k^{su}+\Delta_k^{su})(H_k-H'_k) \\ - (H_{k+1}-H'_{k+1})(A_k^{su}+\Delta_k^{su})(H_k-H'_k)) \leq \\
    R(2 \varepsilon_1 \delta + 2\delta(R+\varepsilon_1)\varepsilon_2 + (R+\varepsilon_1)\delta^2 ) = R(2\varepsilon_1 + 2(R+\varepsilon_1)\varepsilon_2 + (R+\varepsilon_1)\delta)\delta.
\end{multline*}
Taking into account \eqref{eq:condep1ep2} and $\delta \leq 2\varepsilon_2$ we conclude that
$
\|QH - QH'\| \leq \frac{1}{2L'}\|H-H'\|.
$
Since $F$ is linear, inequality \eqref{eq:FL'} implies that 
\begin{equation}\label{eq:FQcontr}
\|F \circ QH - F \circ QH'\| \leq L'\frac{1}{2L'}\|H-H'\| = \frac{1}{2}\|H-H'\|,    
\end{equation}
and hence $F \circ Q$ is contracting on $U$. 
Note that 
$\|Q\{H_k = 0\}\| = \|\{-Z_k\Delta_k^{us}\}\| \leq R\varepsilon_1$, hence
$
\|F \circ Q\{H_k = 0\}\| \leq L'R\varepsilon_1
$
and according to \eqref{eq:FQcontr} and \eqref{eq:defep1ep2} for any $H \in U$ we have
$$
\|F \circ Q H\| \leq
\|F\circ QH-F\circ Q\{0\}\|+\|F\circ Q\{0\}\| \leq \varepsilon_2/2 + L'R\varepsilon_1 \leq \varepsilon_2
$$
and hence $F \circ Q(U) \subset U$. By the contraction mapping principle, there exists a unique $H = \{H_k\}$, with $|H_k| \leq \varepsilon_2$ such that $F \circ Q H = H$. Consequently, the subspaces $\tilde{E}_k^s$ defined by \eqref{eq:tildeEDef} satisfy \eqref{eq:ETildaIncl}.

Let us show that for small enough $\varepsilon, \varepsilon_2$ the spaces $\tilde{E}_k^s$ satisfy 
\begin{equation}\label{eq:GH3Bs}
    |B_{k+n-1} \cdot \ldots \cdot B_k\tilde{v}^s_k| \leq C_1 \lambda_1^n |\tilde{v}^s_k|, \quad \tilde{v}^s_k \in \tilde{E}^s_k, n  \geq 0    
\end{equation}
for some $C_1 = C_1(C, \lambda, \lambda_1, R)$. Indeed, let us choose $N$ such that $2C\lambda^N \leq \lambda_1^N$ and choose $C_1 = \left(\frac{R+1}{\lambda_1}\right)^N$. For small enough $\varepsilon, \varepsilon_2$ we have
$$
\|B_{k+N-1} \cdot \ldots \cdot B_k|_{\tilde{E}_k^s}\| \leq 2 \|A_{k+N-1} \cdot \ldots \cdot A_k|_{{E}_k^s}\| \leq 2C\lambda^N \leq \lambda_1^N, \quad k \in \Zz.
$$
Represent $n > 0$ in the form $n = lN + r$, with $l \geq 0$, $r \in [0, N-1]$.
Note that
\begin{multline}
\|B_{k+n-1} \cdot \ldots \cdot B_k|_{\tilde{E}_k^s}\| \leq \\
\left(\prod_{i = 0}^{l-1}\|B_{k+(i+1)N-1} \cdot \ldots \cdot B_{k+iN}|_{\tilde{E}_{k+iN}^s}\|\right)\|B_{k+lN+r-1} \cdot \ldots \cdot B_{k+lN}|_{\tilde{E}_{k+lN}^s}\| \leq \\ (\lambda_1^N)^l(R+1)^r\leq C_1\lambda_1^n, \label{eq:B-proof}
\end{multline}
which implies \eqref{eq:GH3Bs}.

Similarly, for small enough $\varepsilon$ we construct a sequence of operators $H_k^u:E_k^u \to A_{k-1}E_{k-1}^s$ and spaces $\tilde{E}_k^u = \{ v_k^u + H_k^uv_k^u : v_k^u \in E_k^u\}$ such that 
$
B_k^{-1} \tilde{E}_{k+1}^u \subseteq \tilde{E}_{k}^u
$
and
\begin{equation}\label{eq:GH3Bs-inv}
    |B_{k-n}^{-1} \cdot \ldots \cdot B_{k-1}^{-1}\tilde{v}^u_k| \leq C_1 \lambda_1^n |\tilde{v}^u_k|, \quad \tilde{v}^u_k \in \tilde{E}^u_k, n  \geq 0.
\end{equation}
For small enough $\varepsilon_2$ the inequalities $\|H_k\|, \|H_k^u\| \leq \varepsilon_2$ imply that the operator
\begin{equation}\notag%\label{eq:HktoPk}
T_k = 
\begin{bmatrix}
Id & H_k^u \\
H_k & Id 
\end{bmatrix}   
\end{equation}
is invertible, and projections
$
P_{\tilde{E}_{k}^{s, u}} = T_k\circ P_{{E}_k^{s, u}} \circ T^{-1}_{k} 
$
satisfy
$\|P_{\tilde{E}_{k}^{s, u}}\| \leq 2C$ and $\tilde{E}_k^s \oplus \tilde{E}_k^u = \Bb$.

Hence, the sequence $\mathcal{B}$ satisfies generalized hyperbolicity with spaces $\tilde{E}_k^s, \tilde{E}_k^u$ and constants $C_1, \lambda_1$. Item \textit{(Robust1)} is proved.

To prove item \textit{(Robust2)}, we consider an $m$-periodic sequence of operators $\mathcal{A} = \{A_k\}$ that satisfies periodic generalized hyperbolicity, with the corresponding periodic splitting $\Bb = E_k^s \oplus E_k^u$. Let $\mathcal{B} = \{B_k\}$ be an $m$-periodic perturbation of $\mathcal{A}$ satisfying \eqref{eq:AkBk}. We show that the subspaces $\tilde{E}_k^s$, $\tilde{E}_k^u$ constructed in the proof of  \textit{(Robust1)}  are $m$-periodic. 

The sequence of operators $H = \{H_k\} \in X$ satisfying $(F \circ Q)H = H$ can be obtained as a limit of iterates $H^l = (F \circ Q)^l\{H_k = 0\} \in X$:
$
H = \lim_{l \to \infty}H^l.
$
If sequences  $\mathcal{A}$, $\mathcal{B}$ and projections $P_k$, $Q_k$ are $m$-periodic, then the operator $(F \circ Q)$ maps $m$-periodic sequences to $m$-periodic sequences. Hence, for any $l>0$ the iterate $H^l$ is $m$-periodic, and so is its limit $H$. Therefore, the subspaces $\tilde{E}_k^s$ constructed in the proof of item \textit{(Robust1)} are $m$-periodic.  Similar considerations show that $\tilde{E}_k^u$ is $m$-periodic as well. Consequently, $\mathcal{B}$ satisfies periodic generalized hyperbolicity. This proves item \textit{(Robust2)}.
\end{proof}

\section{Shadowing, density of periodic points, and robustness}\label{sec:proofsSh}

\subsection{Finite shadowing} \label{sec:proofFinLipSh}

\begin{lem}\label{lem:FinLipSh}
If $f\in\DUC(\Bb)$ satisfies generalized hyperbolicity, then $f$ satisfies the finite Lipschitz shadowing property (FinLipSh).
\end{lem}
\begin{proof}
Let $M = 2L$. Fix $d_0 > 0$, satisfying the following inequalities 
\begin{equation}\label{eq:M2'}
    L + 2(RM+R)(RM+R+L)\cdot r\left((RM+R+L)d_0\right) \leq M, 
\end{equation}
\begin{equation}\label{eq:M2''}
 2 (RM+R+L) \cdot r\left((RM+R+L)d_0\right) < 1.
\end{equation}
We first prove the statement for intervals of the form $[-N,N]$.
\begin{lem}\label{lem:NLipSh}
    For any $d < d_0$ and any $d$-pseudotrajectory $\{y_k\}_{k \in [-N, N]}$ there exists an exact trajectory $\{x_k\}_{k \in [-N, N]}$ such that
    $
    |y_k - x_k| \leq Md
    $.
\end{lem}
\begin{proof}
We prove the claim by induction on $N$. For $N=0$, the statement is immediate.

Consider a $d$-pseudotrajectory $\{y_{k \in [-(N+1), N+1]}\}$ with $d<d_0$. 
By the induction hypothesis, the $d$-pseudotrajectory $\{y_{k \in [-N, N]}\}$ is $Md$-shadowed by an exact trajectory $\{p_{k \in [-N, N]}\}$
\begin{equation}\label{Add4.1}
|y_k - p_k| \leq M d, \quad k \in [-N, N].
\end{equation}
Note that $p_{N+1} = f(p_N)$, $p_{-(N+1)} = f^{-1}(p_{-N})$. Due to \eqref{eq:C1-2}, \eqref{eq:pst1}, \eqref{eq:pst2} the following estimates hold
\begin{equation}\label{Add4.2}
|y_{N+1} - p_{N+1}| \leq (RM+1)d, \quad |y_{-(N+1)} - p_{-(N+1)}| \leq (RM+R)d.
\end{equation}
Hence, we conclude the following.
\begin{statement}\label{stat:dn+1}
For  $d< d_0$ a $d$-pseudotrajectory $\{y_{k \in [-(N+1), N+1]}\}$ can be $(RM+R)d$-shadowed.   
\end{statement}

In the following, we show how to construct a better shadowing trajectory.
Denote $A_k = Df(p_k)$. Consider the sequence $\{w_k \in \Bb\}$ defined by
$$
y_{k+1} - f(y_k) = - d w_{k+1}, \quad k \in [-(N+1), N], 
$$
$$
w_{k+1} = 0, \quad k \notin [-(N+1), N].
$$
Note that $|w_k| \leq 1$.
By Lemma \ref{lem:SBS} there exists a sequence $\{v_k \in \Bb\}_{k \in \Zz}$, $|v_k| \leq L$ such that equalities \eqref{eq:Inh} hold.
Consider the sequence $z_k = y_k + dv_k$. Note that 
$$
f(z_k) = f(y_k + dv_k) = f(p_k + (y_k-p_k) +dv_k), \quad k \in [-(N+1), N+1].
$$
Inequalities \eqref{eq:C1-1}, \eqref{eq:C1-2}, \eqref{Add4.1}, \eqref{Add4.2} imply that
\begin{multline*}
|f(z_k) - (f(p_k) + A_k (y_k - p_k + dv_k) )| \leq |y_k - p_k + dv_k|r(|y_k - p_k + dv_k|) \leq \\ (RM+R+L)d \cdot r\left((RM+R+L)d\right).    
\end{multline*}
Then for $k \in [-(N+1), N]$ the following inequalities hold
\begin{multline*}
|z_{k+1} - f(z_k)| = |y_{k+1} + dv_{k+1} - p_{k+1} - \\ A_k (y_k - p_k + dv_k) - (f(z_k) - (f(p_k) + A_k (y_k - p_k + dv_k) ))|  \leq \\  
|y_{k+1}  - p_{k+1} + dw_{k+1} - A_k (y_k - p_k)| +|f(z_k) - (f(p_k) + A_k (y_k - p_k + dv_k) )| \leq \\  
|f(y_k) - f(p_k) - A_k (y_k - p_k)| + (RM+R+L)d \cdot r\left((RM+R+L)d\right) \leq \\
|y_k - p_k|r(|y_k - p_k|) + (RM+R+L)d \cdot r\left((RM+R+L)d\right) \leq \\ 
(RM+R)d\cdot r((RM+R)d) + (RM+R+L)d \cdot r\left((RM+R+L)d\right) \leq \\
2 (RM+R+L)d \cdot r\left((RM+R+L)d\right).
\end{multline*}
Hence $\{z_{k \in [-(N+1), N+1]}\}$ is a $2 (RM+R+L)d \cdot r\left((RM+R+L)d\right)$-pseudotrajectory. Inequality~\eqref{eq:M2''} implies 
$
2 (RM+R+L)d \cdot r\left((RM+R+L)d\right) < d_0
$.
According to Proposition \ref{stat:dn+1} the pseudotrajectory $\{z_{k \in [-(N+1), N+1]}\}$ can be $2(RM+R)(RM+R+L)d \cdot r\left((RM+R+L)d\right)$-shadowed by an exact trajectory $q_{k \in [-(N+1), N+1]}$. Then by \eqref{eq:M2'}
\begin{multline*}
|y_k - q_k| \leq |y_k - z_k| + |z_k-q_k| \leq \\ Ld + 2(RM+R)(RM+R+L)d \cdot r\left((RM+R+L)d\right) \leq Md    
\end{multline*}
and the pseudotrajectory $\{y_{k \in [-(N+1), N+1]}\}$ is $Md$-shadowed by $\{q_{k \in [-(N+1), N+1]}\}$, which completes the induction step and proves Lemma \ref{lem:NLipSh}.
\end{proof}

Every finite interval $I\subset\Zz$ is contained in $[-N,N]$ for some $N>0$.
Restricting the trajectory supplied by Lemma~\ref{lem:NLipSh} to $I$ proves Lemma~\ref{lem:FinLipSh}.
\end{proof}

\subsection{Lipschitz shadowing. Proof of Theorem~\ref{thm:main-results}\,(\ref{item:main-LipSh})}\label{sec:proofLipSh}

The main step of the proof of Theorem~\ref{thm:main-results}\,(\ref{item:main-LipSh}) is the following lemma.

\begin{lem}\label{lem:BS-SH}
Assume that $f:\Bb\to\Bb$ is a $C^1$-diffeomorphism satisfying \eqref{eq:C1-1}-\eqref{eq:C1-r2}, and that there exists $d_1, M > 0$ such that, for any $d < d_1$ and any $d$-pseudotrajectory $\{y_k\}$, the sequence ${\mathcal{B}} = \{B_k = Df(y_k)\}_{k \in \Zz}$ satisfies the bounded solution property with constant $M$. Then $f$ satisfies the Lipschitz shadowing property. 
\end{lem}
\begin{proof}
    By \eqref{eq:C1-r2} there exists $d_0 \in (0, d_1)$ such that
\begin{equation}\label{eq:1.1'}
    M\cdot r(Md_0) < \frac{1}{2}.
\end{equation}
\begin{statement}\label{stat:d2}
    For $d_0$ satisfying \eqref{eq:1.1'} and any $d$-pseudotrajectory $\{y_k\}_{k \in \Zz}$ with $d<d_0$ there exists $d/2$-pseudotrajectory $\{z_k\}_{k \in \Zz}$ such that
    \begin{equation}\label{eq:d2}
    |y_k-z_k| < Md.
    \end{equation}
\end{statement}
\begin{proof}
    Consider a $d$-pseudotrajectory $\{y_k\}$ with $d < d_0$. Similarly to the proof of Lemma~\ref{lem:FinLipSh} consider the sequence $\{w_{k \in \Zz}\}$ defined by
\begin{equation}\label{eq:defw}
    y_{k+1}-f(y_k) = -dw_{k+1}, \quad k \in \Zz.
\end{equation}
    Consider $B_k = Df(y_k)$. By the assumptions of Lemma \ref{lem:BS-SH} there exists a sequence $v_k \in \Bb$ such that
    $$
    v_{k+1} - B_kv_k = w_{k+1}, \quad |v_k| \leq M, \quad k \in \Zz.
    $$
 Consider the sequence $\{z_k = y_k + dv_k\}_{k \in \Zz}$. By construction, $\{z_k\}$ satisfies \eqref{eq:d2}. We now show that $\{z_k\}$ is a $d/2$-pseudotrajectory. By the construction of the sequence $\{v_k\}$ we have
    \begin{equation}\label{eq:zkBk}
        y_{k+1} + dv_{k+1} - (f(y_k) + B_kdv_k) = 0.
    \end{equation}
     Due to \eqref{eq:C1-1}, \eqref{eq:1.1'} and \eqref{eq:zkBk} the following inequalities hold
\begin{multline*}
|z_{k+1} - f(z_k)| = |y_{k+1} + dv_{k+1} - (f(y_k)+B_kdv_k) - (f(y_k+dv_k) - f(y_k) -B_kdv_k)|  = \\
|f(y_k+dv_k) - f(y_k) -B_kdv_k| \leq d|v_k|r(d|v_k|) \leq Md \cdot r(Md) \leq \frac{d}{2}.
\end{multline*}
Proposition \ref{stat:d2} is proved.
\end{proof}

Continue the proof of Lemma \ref{lem:BS-SH}.
For a given $d$-pseudotrajectory $\{y_k\}_{k \in \Zz}$ with $d< d_0$ consider the sequence of $\delta_l$-pseudotrajectories $\{y_k^l\}_{k \in \Zz}$ defined inductively.

Define $\delta_0 = d$, $\{y_k^0 = y_k\}_{k \in \Zz}$.

By a given $\delta_l$-pseudotrajectory $\{y_k^l\}_{k \in \Zz}$ with $\delta_l < d_0$ according to Proposition  \ref{stat:d2} choose $\delta_{l+1} = \delta_l/2$-pseudotrajectory $\{y_k^{l+1}\}$ such that 
$
|y_k^{l+1} - y_k^l| \leq M \delta_l.
$
By construction we have $\delta_l = \frac{d}{2^l}$ and
\begin{equation}\label{eq:3.1-3.2}
    |y_{k+1}^l - f(y_k^l)| \leq \frac{d}{2^l}, \quad |y_k^{l+1} - y_k^l| \leq M\frac{d}{2^l}.
\end{equation}   
Hence, for any fixed $k \in \Zz$, the sequence $\{y_k^l\}$ is a Cauchy sequence with respect to $l$ and there exists a limit
$
x_k = \lim_{l \to \infty} y_k^l.
$
Note that
\begin{equation}\label{eq:42.5}
|y_k - x_k| = |y_k^0 - x_k| \leq \sum_{l\geq 0} |y_k^l-y_k^{l+1}| \leq \sum_{l\geq 0}M\frac{d}{2^l} = 2Md.    
\end{equation}
The continuity of $f$ and condition \eqref{eq:3.1-3.2} imply that $x_{k+1} = f(x_k)$. Therefore, $\{x_k\}$ is an exact trajectory of $f$ that $2Md$-shadows $\{y_k\}$. Lemma \ref{lem:BS-SH} is proved.
\end{proof}

Now we show that the hypotheses of Theorem~\ref{thm:main-results}\,(\ref{item:main-LipSh}) imply the hypotheses of Lemma~\ref{lem:BS-SH}. By Lemma \ref{lem:FinLipSh}, the map $f$ satisfies the finite Lipschitz shadowing property. Let $L > 0$ be the constant from Lemma \ref{lem:SBS}, let $\varepsilon > 0$ be the constant from Lemma \ref{lem:SBS-robust}, and let $M, d_0$ be the constants from the finite Lipschitz shadowing property. 

By condition \eqref{eq:C1-r1}, choose $\delta > 0$ such that if $|x - y|< \delta$ we have $\|Df(x)-Df(y)\| < \varepsilon$. Set $d_1 := \min(d_0, \delta/M)$. Fix $d<d_1$ and a $d$-pseudotrajectory $\{y_k\}_{k \in \Zz}$. Define the sequence of operators $\mathcal{B} = \{B_k = Df(y_k)\}_{k \in \Zz}$. We claim that the sequence $\mathcal{B}$ satisfies the bounded solution property with constant $4L$.

By Lemma \ref{lem:FinLipSh}, for any finite interval $I \subset \Zz$ the sequence $\{y_k\}_{k \in I}$ can be $Md$-shadowed by an exact trajectory $\{x_k\}_{k \in I}$. Set $\{A_k = Df(x_k)\}_{k \in I}$. By Lemma~\ref{lem:SBS} it satisfies the strong bounded solution property with constant $L$. Moreover, since $Md < \delta$ we have $\|B_k - A_k\| \leq \varepsilon$, $k \in I$.
By Lemma \ref{lem:SBS-robust}, the sequence $\{B_k\}_{k \in I}$ satisfies the bounded solution property with constant~$2L$. Lemma~\ref{lem:Bk-Inf} implies that $\{B_k\}_{k \in \Zz}$ satisfies the bounded solution property with constant~$4L$. Therefore, the hypotheses of Lemma \ref{lem:BS-SH} hold, which completes the proof of Theorem~\ref{thm:main-results}\,(\ref{item:main-LipSh}).

\subsection{Periodic shadowing. Proof of Theorem~\ref{thm:main-results}\,(\ref{item:main-LipPerSh})}\label{sec:LipPerSh}

Let $\{y_k\}$ be an $m$-periodic $d$-pseudotrajectory. By Theorem~\ref{thm:main-results}\,(\ref{item:main-LipSh}), there exists  $M > 0$ such that $\{y_k\}$ is $Md$-shadowed by an exact trajectory (not necessarily periodic) $\{x_k\}$: 
$$
\dist(y_k, x_k) \leq Md, \quad k \in \mathbb{Z}.
$$
Define the sequences of the linear maps $\mathcal{B} = \{B_k := Df(y_k)\}$ and $\mathcal{A} = \{A_k := Df(x_k)\}$. Since $f$ satisfies generalized hyperbolicity, the sequence $\mathcal{A}$ satisfies generalized hyperbolicity as well. By item (Robust1) of Lemma~\ref{lem:CL-robust-Ak}, for any $\lambda_1 \in (\lambda, 1)$ there exists $C_1 = C_1(C, \lambda, \lambda_1, R) > 0$ such that, for sufficiently small $d$, the sequence $\mathcal{B}$ also satisfies generalized hyperbolicity with constants $C_1$, $\lambda_1$.

Consider the $m$-periodic sequence $\{w_k \in \mathbb{B}\}$ defined by~\eqref{eq:defw}. By Lemma~\ref{lem:SBS-periodic}, there exist $L = L(C_1, \lambda_1) > 0$ and an $m$-periodic sequence $\{v_k \in \mathbb{B}\}$ such that
$$
v_{k+1} = B_k v_k + w_{k+1}, \quad \text{and} \quad \|v_k\| \leq L.
$$
Using an argument as in Proposition~\ref{stat:d2}, and reducing $d$ if necessary, there exists a constant $M_1 > 0$ such that the sequence $y_k^1 := y_k + d v_k$, $k \in \mathbb{Z}$, is an $m$ -periodic $d/2$ -pseudotrajectory and, furthermore, satisfies $\|y_k - y_k^1\| \leq M_1 d$.
By iterating this construction, as in the proof of Theorem~\ref{thm:main-results}\,(\ref{item:main-LipSh}), we obtain a sequence of $m$-periodic pseudotrajectories $\{y_k^l\}_{l \geq 1}$ satisfying the following estimates (compare with~\eqref{eq:3.1-3.2}):
$$
\|y_{k+1}^l - f(y_k^l)\| \leq \frac{d}{2^l}, \quad \|y_k^{l+1} - y_k^l\| \leq M_1 \frac{d}{2^l}.
$$
Taking the pointwise limit $z_k := \lim_{l \to \infty} y_k^l$, we obtain an $m$-periodic exact trajectory, similarly to~\eqref{eq:42.5}, satisfying
$
\|y_k - z_k\| \leq 2M_1 d.
$
This completes the proof of Theorem~\ref{thm:main-results}\,(\ref{item:main-LipPerSh}).

\begin{proof}[Proof of Theorem~\ref{thm:main-results}\,(\ref{item:main-chain})]
Let $M, d_0 > 0$ be the constants from the Lipschitz periodic shadowing property in Theorem~\ref{thm:main-results}\,(\ref{item:main-LipPerSh}). Fix $x \in CR(f)$. For any $d \in (0, d_0)$ consider a finite pseudotrajectory $y_{k \in [0, N]}$ such that $y_0 = y_N = x$. This segment can be continued to an $N$-periodic pseudotrajectory 
$$
y_{k+iN} = y_k, \quad k \in [0, N-1], i \in \Zz.
$$
By Theorem~\ref{thm:main-results}\,(\ref{item:main-LipPerSh}) there exists an $N$-periodic exact trajectory $x_k^{(d)}$ that $Md$-shadows $\{y_k\}$; in particular
$
|x - x_0^{(d)}| = |y_0 - x_0^{(d)}| \leq Md
$.
Since $x_0^{(d)} \in Per(f)$ it follows that $\overline{Per(f)} = CR(f)$.
\end{proof}

\subsection{Uniform $C^1$-robustness. Proof of Theorem~\ref{thm:main-results}\,(\ref{item:main-robust})}\label{sec:robust}

\begin{proof}[Proof of Theorem~\ref{thm:main-results}\,(\ref{item:main-robust})]
The proof of Theorem~\ref{thm:main-results}\,(\ref{item:main-robust}) combines Theorem~\ref{thm:main-results}\,(\ref{item:main-LipSh})--(\ref{item:main-LipPerSh}) with the robustness of generalized hyperbolicity for sequences of operators (Lemma~\ref{lem:CL-robust-Ak}).

Indeed, consider a diffeomorphism $f:\Bb \to \Bb$ satisfying generalized hyperbolicity with constants $C, \lambda$. Fix $\lambda_1 \in (\lambda, 1)$. 
For any trajectory $\{x_k\}$, the sequence of linear operators $\{Df(x_k)\}$ satisfies generalized hyperbolicity. Take $M, d_0 > 0$ as in Theorem~\ref{thm:main-results}\,
(\ref{item:main-LipSh})--(\ref{item:main-LipPerSh}) and $\varepsilon > 0, C_1 > C$ provided by Lemma~\ref{lem:CL-robust-Ak}. Choose $\delta \in (0, \min(d_0, \varepsilon/2))$ such that 
\begin{equation}\label{eq:delta1}
    \|Df(x)-Df(y)\| \leq \varepsilon/2, \quad |x-y| < M\delta.
\end{equation}
Let $g\in\DUC(\Bb)$ satisfy
\begin{equation}\label{eq:delta2}
\|f-g\|_{C^1}<\delta.
\end{equation}

Consider the set $\mathcal{O}$ of all trajectories of the diffeomorphism $g$. For any trajectory $\omega \in \mathcal{O}$ choose a point $y_{\omega} \in \omega$. 

We now construct spaces $E_y^{s, u}$ for all points $y \in \omega$. Consider an exact trajectory $y_{k\in \Zz}$ of diffeomorphism $g$ with $y_0 = y_{\omega}$.
Note that $|y_{k+1} - f(y_k)| = |g(y_k) - f(y_k)| \leq \delta$, $k \in \Zz$, hence $y_k$ is a $\delta$-pseudotrajectory of the diffeomorphism $f$. Denote $B_k = Dg(y_k)$. Consider two cases.

\textbf{Case 1.} If the sequence $\{y_{k}\}$ is not periodic, then the points $\{y_k\}$ are pairwise distinct. By Theorem~\ref{thm:main-results}\,(\ref{item:main-LipSh}) there exists an exact trajectory $\{x_k\}$ of diffeomorphism $f$ that $M\delta$-shadows~$\{y_k\}$:
\begin{equation}\label{eq:robustykxk}
    |y_k-x_k| \leq M\delta, \quad x_{k+1} = f(x_k), \quad k \in \Zz.    
\end{equation}
Denote $A_k = Df(x_k)$. According to \eqref{eq:delta1}, \eqref{eq:delta2} the following holds
\begin{equation}\label{eq:robustAkBk}
\|A_k - B_k\| \leq \|A_k - Df(y_k)\| + \|Df(y_k) - B_k\| \leq \varepsilon/2 + \delta < \varepsilon.
\end{equation}
By item (Robust1) in Lemma \ref{lem:CL-robust-Ak}, the sequence $\{B_k\}$ satisfies generalized hyperbolicity with constants $C_1, \lambda_1$. 

\textbf{Case 2.} Suppose that the sequence $\{y_{k}\}$ is $m$-periodic. Then the sequence $\{B_k\}$ is $m$-periodic as well. By Theorem~\ref{thm:main-results}\,(\ref{item:main-LipPerSh}), there exists an $m$-periodic exact trajectory $\{x_k\}$ satisfying \eqref{eq:robustykxk}. The sequence $\{A_k = Df(x_k)\}$ is $m$-periodic and satisfies \eqref{eq:robustAkBk}. By item (Robust2) in Lemma \ref{lem:CL-robust-Ak}, the sequence $\{B_k\}$ satisfies $m$-periodic generalized hyperbolicity with constants $C_1, \lambda_1$. 

In both cases denote the spaces corresponding to generalized hyperbolicity for $\{B_k\}$ by $\tilde{E}_k^s$, $\tilde{E}_k^u$. Define (slightly abusing notation) for point~$y_k$ the spaces 
\begin{equation}\label{eq:Eg}
    \tilde{E}_{y_k}^{s, u} := \tilde{E}_k^{s, u}.    
\end{equation}
If $\omega$ is an $m$-periodic trajectory, then $\tilde{E}_k^{s, u} = \tilde{E}_{k+m}^{s, u}$, and hence $\tilde{E}_{y_k}^{s, u}$ is well-defined. Moreover, Definition \ref{def:CLf} does not impose any condition on the dependence of $E_y^{s, u}$ and the corresponding projections on the point $y$. Hence the splitting defined by \eqref{eq:Eg} satisfies the conditions of generalized hyperbolicity in Definition \ref{def:CLf}. 
\end{proof}
\begin{remark}\label{rem:AoC}
Note that this construction uses the Axiom of Choice, and we do not know whether this dependence can be removed. For each trajectory $\omega$ of $g$, we choose a base point $y_\omega$ and a trajectory of $f$ shadowing the trajectory of $g$ through $y_\omega$. Once these choices are fixed, the construction in Lemma~\ref{lem:CL-robust-Ak} and its equivariance under shifts define the splitting consistently along the whole trajectory $\omega$. The Axiom of Choice is used to make these selections simultaneously over all trajectories of $g$; the nonuniqueness of shadowing prevents a canonical choice.
\end{remark}

\section{Structural stability}\label{sec:SS}

In the proof of Theorem \ref{thm:SS} we use a scale-of-spaces analog of the strong bounded solution property. A straightforward analog would be to consider the equation
$$
w(x) = v(f(x)) - Df(x)v(x), \quad x \in \Bb,
$$
where $v, w: \Bb \to \Bb$ are bounded uniformly continuous functions (compare with the notion of infinitesimal stability \cite{Mane1977}). However, we need a slightly more general notion.

Consider a homeomorphism $\alpha: \Bb \to \Bb$ and a linear cocycle $A:\Bb \to GL(\Bb)$. Assume that there exists $R > 0$ such that for any $x \in \Bb$, the operator $A(x):\Bb \to \Bb$ is a bounded linear isomorphism satisfying 
$
\|A(x)\|, \|A^{-1}(x)\| \leq R
$.
Moreover, assume that $A(x)$ is uniformly continuous in $x$, i.e., for any $\varepsilon > 0$ there exists $\delta > 0$ such that
$$
\|A(x) - A(y)\|, \|A^{-1}(x) - A^{-1}(y)\|\leq \varepsilon, \quad \mbox{for $|x-y| < \delta$}.
$$
Denote
$$
A^n_{\alpha}(x) = \begin{cases}
A(\alpha^{n-1}(x))\dots A(\alpha(x))A(x), & \quad n >0,\\
Id, &\quad n = 0,\\
A^{-1}(\alpha^{n}(x)) \dots A^{-1}(\alpha^{-2}(x))A^{-1}(\alpha^{-1}(x)), & \quad n < 0.
\end{cases}
$$

\begin{defin}\label{def:CL-cocycle}
    We say that the pair $(\alpha, A)$ satisfies \textit{cocycle generalized hyperbolicity} if there exist $C > 0$, $\lambda \in (0, 1)$ and a family of projections $P_x, Q_x: \Bb \to \Bb$ satisfying the following. 
\begin{itemize}
    \item[(C-GH1)] \textbf{Splitting.} $P_x + Q_x = Id$, $\|P_x\|, \|Q_x\| \leq C$. \\Denote $E^s_x = P_x(\Bb), E^u_x = Q_x(\Bb)$; then $E^s_x \oplus E^u_x = \Bb$.
    \item[(C-GH2)] \textbf{Inclusion.} $A(x)E^s_x \subseteq E^s_{\alpha(x)}$, $A^{-1}(x)E^u_{\alpha(x)} \subseteq E^u_{x}$,
    \item[(C-GH3)] \textbf{Exponential estimates.}
    $$
    |A^n_{\alpha}(x)v^s_x| \leq C \lambda^n |v^s_x|, \quad v^s_x \in E^s_x, n  \geq 0;
    $$
    $$
    |A^{-n}_{\alpha}(x)v^u_x| \leq C \lambda^n |v^u_x|, \quad v^u_x \in E^u_x, n  \geq 0.
    $$
\end{itemize}
If, additionally, $\alpha$ and $\alpha^{-1}$ are uniformly continuous and the projections $P_x$, $Q_x$ depend uniformly continuously on $x$, i.e. satisfy \eqref{eq:cont-PQ}, then we say that $(\alpha,A)$ satisfies \textit{uniformly continuous cocycle generalized hyperbolicity}.
\end{defin}

Consider a pair $(\alpha, A)$ satisfying uniformly continuous cocycle generalized hyperbolicity, with associated spaces $E^s_x$, $E^u_x$. As in Section \ref{sec:robust-Ak}, define
$$
F_x := A^{-1}(x)(E_{\alpha(x)}^u) \subset E_x^u.
$$
Let $X = BUC(\Bb, \Bb)$ be the Banach space of all bounded uniformly continuous functions from $\Bb$ to itself with the supremum norm. 
Define the subspace $Y \subset X$ by
$$
Y = \{v \in X: v(x) \in E_x^s \oplus F_x\}.
$$
The space $Y$ is a closed linear subspace of $X$ and hence is a Banach space with the supremum norm.
Define the operator $T_{\alpha, A}: X \to X$ by
\begin{equation}\label{eq:Twv}
T_{\alpha, A} v = w, \quad \mbox{where $w(\alpha(x)) = v(\alpha(x)) - A(x)v(x)$}.    
\end{equation}

\begin{lem}\label{lem:SS1}
Assume that $(\alpha, A)$ satisfies the uniformly continuous cocycle generalized hyperbolicity with constants~$C, \lambda$. Then there exists $L = L(\lambda, C)$, such that the restriction $T_{\alpha, A}: Y \to X$ is a bijection and admits a bounded inverse  
\begin{equation}\label{eq:Sf}
S_{\alpha, A} = T^{-1}_{\alpha, A}: X \to Y, \quad \|S_{\alpha, A}\| \leq L.
\end{equation}
\end{lem}
\begin{remark}
This statement generalizes \cite[Theorem 1, Step 1]{NilsonAli}, where the case in which $A(x)$ is a constant generalized hyperbolic isomorphism was considered.
\end{remark}

\begin{proof}
    Consider the operator $S_{\alpha, A}:X \to Y$ as in Lemma \ref{lem:SBS}. 
    Given $w \in X$, define $v = S_{\alpha, A}w$ by
\begin{equation}\label{eq:Svdef}
v(x) = \sum_{i \leq 0} A^{-i}_{\alpha}(\alpha^i(x)) P_{\alpha^i(x)}w(\alpha^i(x)) -\sum_{i>0}A^{-i}_{\alpha}(\alpha^i(x))Q_{\alpha^i(x)}w(\alpha^i(x)).    
\end{equation}        
As in Lemma \ref{lem:SBS}, the series \eqref{eq:Svdef} converges for every $x$ and defines a function $v$ that satisfies \eqref{eq:Twv}. Moreover, $|v(x)| \leq L\|w\|$, where $L = C^2\frac{1+\lambda}{1-\lambda}$. Define $v_N(x)$ analogously by truncating the sum in \eqref{eq:Svdef} to $|i| \leq N$. Then $v_N \to v$ uniformly and since each $v_N$ is uniformly continuous, $v$ is also uniformly continuous.
The first sum in \eqref{eq:Svdef} takes values in $E_x^s$, and the second sum takes values in $F_x$. Hence, $v(x) \in E_x^s \oplus F_x$ and $v \in Y$. Therefore, $S_{\alpha, A}$ is well-defined and satisfies
$
T_{\alpha, A}S_{\alpha, A} = Id$, $\|S_{\alpha, A}\| \leq L,
$
so $T_{\alpha, A}: Y \to X$ is surjective. 

We now show that its kernel is trivial. Assume that for some $v \in Y$, $v \ne 0$, 
\begin{equation}\label{eq:KerT}
T_{\alpha, A}v = 0.    
\end{equation}
Consider $x \in \Bb$ such that $v(x) \ne 0$. Relation \eqref{eq:KerT} implies that
$$
v(\alpha^n(x)) = A^n_{\alpha}(x)v(x), \quad n \in \Zz.
$$
Denote $v_n = v(\alpha^n(x)) \in E^s_{\alpha^n(x)} \oplus F_{\alpha^n(x)}$. Define $v_n = v_n^s + v_n^u$, where $v_n^s \in E^s_{\alpha^n(x)}$, $v_n^u \in F_{\alpha^n(x)}$. Note that 
$$
A(\alpha^n(x)) v_n^s \in E_{\alpha^{n+1}(x)}^s, \quad A(\alpha^n(x)) v_n^u \in E_{\alpha^{n+1}(x)}^u.
$$
Hence,
$$
v_{n+1}^s = A(\alpha^n(x)) v_n^s, \quad v_{n+1}^u = A(\alpha^n(x)) v_n^u, \quad n \in \Zz.
$$
By induction, we conclude
$$
v_{n+k}^s = A^k_{\alpha}(\alpha^n(x)) v_n^s, \quad v_{n+k}^u = A^k_{\alpha}(\alpha^n(x)) v_n^u, \quad n, k \in \Zz.
$$
Conditions (C-GH1), (C-GH3) imply that
$$
|v_{-n}| \geq \frac{1}{C^2}\lambda^{-n} |v_{0}^s|, \quad |v_{n}| \geq \frac{1}{C^2}\lambda^{-n} |v_{0}^u|, \quad n > 0.
$$
Note that 
\begin{itemize}
    \item if $v_{0}^s \ne 0$ then the values $|v_{-n}|$, $n \geq 0$ are unbounded,
    \item if $v_{0}^u \ne 0$ then the values $|v_{n}|$, $n \geq 0$ are unbounded.
\end{itemize}
Since $v(x) \ne 0$, at least one of $v_0^s$, $v_0^u$ is nonzero, so the sequence $\{v(\alpha^n(x))\}_{n \in \Zz}$ is unbounded; this is impossible, since $v \in Y \subset X$ is bounded. This shows that the kernel of $T_{\alpha, A}: Y \to X$ is trivial, so $T_{\alpha, A}$ is injective. Hence, $T_{\alpha, A}$ is invertible and $S_{\alpha, A}$ is its inverse.   
\end{proof}
We will need the following form of robustness of uniformly continuous cocycle generalized hyperbolicity. We construct the splitting for the perturbed cocycle by a graph transform of the splitting of the original cocycle. Note that since $E^s_x$, $E^u_x$ vary with $x$ we need the following notion.
\begin{defin}
    Consider a cocycle $(\alpha, A)$ satisfying uniformly continuous cocycle generalized hyperbolicity. We say that the linear mappings $H^s_x:E^s_x \to F_x$ depend uniformly continuously on $x$ if the mappings
    $$
        \hat{H}^s_x:\mathbb{B} \to \mathbb{B}, \quad \mbox{ with $\hat{H}^s_x(v) := H^s_x(P_x v)$}
    $$
    are uniformly continuous in $x$ in the operator norm.
 Similarly, we define the uniform continuity for $H^u_{x}: E_{\alpha(x)}^u \to A(x)E_x^s \subset E_{\alpha(x)}^s$.
\end{defin}

\begin{lem}\label{lem:SS2}
    Assume that the pair $(\alpha, A)$ satisfies uniformly continuous cocycle generalized hyperbolicity with constants $C > 0$, $\lambda \in (0, 1)$. Then for any $\lambda_1 \in (\lambda, 1)$ and $\varepsilon >0$ there exist $C_1 > 1$ and $\delta >0$ such that for any homeomorphism $\beta: \Bb \to \Bb$, with $\beta$ and $\beta^{-1}$ uniformly continuous, satisfying 
\begin{equation}\label{eq:al-be-ep}
    |\alpha(x) - \beta(x)|, |\alpha^{-1}(x) - \beta^{-1}(x)| < \delta, \quad x \in \Bb    
\end{equation}
 there exist mappings $H^s_x: E_x^s \to F_x \subset E_x^u$, $H^u_{x}: E_{\alpha(x)}^u \to A(x)E_x^s \subset E_{\alpha(x)}^s$ that depend uniformly continuously on $x$ such that
    $
    \|H^s_x\|, \|H^u_x\| < \varepsilon
    $
    and the pair $(\beta, A)$ satisfies uniformly continuous cocycle generalized hyperbolicity with constants $(C_1, \lambda_1)$ and spaces 
\begin{align*}
    \tilde{E}_x^s &= (Id + H^s_x)E^s_x = \{v+H^s_x v: v \in E^s_x\},\\
        \tilde{E}_{\beta(x)}^u &= (Id + H^u_x)E^u_{\alpha(x)} = \{v+H^u_x v: v \in E^u_{\alpha(x)}\}.    
\end{align*}
There exists $\varepsilon_0$ such that if $\varepsilon < \varepsilon_0$ then 
\begin{equation}\label{eq:tildeE-EF}
    \tilde{E}^s_x \oplus A^{-1}(x)\tilde{E}^u_{\beta(x)} = E^s_x \oplus F_x.
\end{equation}
\end{lem}

\begin{proof}
The proof uses ideas similar to those in Lemma \ref{lem:CL-robust-Ak}. We highlight the main differences. 
    
    First, we make a construction of $\tilde{E}^s_x$. Fix $\varepsilon >0$, $\lambda_1 \in (\lambda, 1)$. Denote $F_x = A(x)^{-1}E^u_{\alpha(x)}$

\begin{minipage}{0.49\textwidth}
\begin{align*}
A_x^{ss}: E_x^s \to \Bb, & \quad  A_x^{ss} = P_{\alpha(x)}A(x),\\
A_x^{su}: E_x^u \to \Bb, & \quad  A_x^{su} = P_{\alpha(x)}A(x),\\
A_x^{us}: E_x^s \to \Bb, & \quad  A_x^{us} = Q_{\alpha(x)}A(x),\\
A_x^{uu}: E_x^u \to \Bb, & \quad  A_x^{uu} = Q_{\alpha(x)}A(x),
\end{align*}   
\end{minipage}
\hfill
\begin{minipage}{0.49\textwidth}
\begin{align*}
B_x^{ss}: E_x^s \to \Bb, & \quad  B_x^{ss} = P_{\beta(x)}A(x),\\
B_x^{su}: E_x^u \to \Bb, & \quad  B_x^{su} = P_{\beta(x)}A(x),\\
B_x^{us}: E_x^s \to \Bb, & \quad  B_x^{us} = Q_{\beta(x)}A(x),\\
B_x^{uu}: E_x^u \to \Bb, & \quad  B_x^{uu} = Q_{\beta(x)}A(x),
\end{align*}
\end{minipage}

\smallskip 
According to (C-GH2) 
\begin{equation}\label{eq:Asux}
A^{us}_x = 0.
\end{equation}

Since $P_x, Q_x$ depend uniformly continuously on $x$, for any $\varepsilon_1 \in (0, 1)$ there exists $\delta > 0$ such that, if $\beta$ satisfies \eqref{eq:al-be-ep}, then
\begin{equation}\label{eq:AB-ep1}
\|A_x^{ss} - B_x^{ss}\|, \|A_x^{us} - B_x^{us}\|,\|A_x^{su} - B_x^{su}\|,\|A_x^{uu} - B_x^{uu}\| < \varepsilon_1, \quad x \in \Bb.   
\end{equation}
Define $J_x: E^u_{\alpha(x)} \to E^u_{\beta(x)}$ as $J_x := Q_{\beta(x)}|_{E^u_{\alpha(x)}}$, then $J_x$ is an isomorphism, and 
\begin{equation}\notag%\label{eq:R-ep1}
|J_x v_1 - v_1| \leq \varepsilon_1 |v_1|, \; |J^{-1}_x v_2 - v_2| \leq \varepsilon_1 |v_2|, \quad x \in \Bb, v_1 \in E^u_{\alpha(x)}, v_2 \in E^u_{\beta(x)}.
\end{equation}

In the following we show that for any $\lambda_2 \in (\lambda, \lambda_1)$ there exists $C_2 > 0$ and small enough~$\varepsilon_1$ such that $B^{ss}$ and $B^{uu}$ satisfy 
\begin{itemize}
    \item[\textbf{(B1)}] 
        $
    \|B^{ss}_{\beta^{n-1}(x)} \dots B^{ss}_{\beta(x)}B^{ss}_x\| \leq C_2\lambda_2^n, \quad n \geq 0, x \in \Bb,
    $
\item[\textbf{(B2)}] $B^{uu}_{x}$ has the right inverse $Z_x: E_{\beta(x)}^u \to F_x$ such that 
\begin{equation}\label{eq:Zx-cond}
B^{uu}_{x} Z_x  = Id|_{E_{\beta(x)}^u}, \|Z_x\| \leq (1+\varepsilon_1)R,
\end{equation}
\begin{equation}\label{eq:Zx-exp}
    \|Z_{\beta^{-n}(x)} \dots Z_{\beta^{-2}(x)}Z_{\beta^{-1}(x)}\| \leq C_2\lambda_2^n, \quad n \geq 0, x \in \Bb.
\end{equation}

\end{itemize}

The proof of \textbf{(B1)} is similar to the argument in \eqref{eq:GH3Bs}. Fix $\lambda_2$ and choose $N$ so that $2C\lambda^N \leq \lambda_2^N$. Set $C_2=(R(C+2)/\lambda_2)^N$.
Note that for any $\delta_1 > 0$ there exists $\delta > 0$ such that, whenever $\beta$ satisfies \eqref{eq:al-be-ep}, it holds
$
\dist(\beta^n(x), \alpha^n(x)) \leq \delta_1
$
, $(n \in [1, N])$.
For small enough~$\delta_1$ we have
$$
\|B^{ss}_{\beta^{N-1}(x)} \dots B^{ss}_{\beta(x)}B^{ss}_x\| \leq 2\|A^{ss}_{\alpha^{N-1}(x)} \dots A^{ss}_{\alpha(x)}A^{ss}_x\|.
$$
Similarly to \eqref{eq:B-proof} we conclude \textbf{(B1)}.

The proof of \textbf{(B2)} is similar to the argument in \textbf{(B1)} and \eqref{eq:GH3Bs-inv}. Choosing $N$ as in \textbf{(B1)}, reducing additionally $\delta_1$ we may assume that 
$
\dist(\beta^{-n}(x), \alpha^{-n}(x)) \leq \delta_1
$, $(n \in [1, N])$.
Define $Z_x = A(x)^{-1}J_x^{-1}$, which satisfies \eqref{eq:Zx-cond}. As in \textbf{(B1)}, we conclude that 
$$
\|Z_{\beta^{-N}(x)} \dots Z_{\beta^{-2}(x)}Z_{\beta^{-1}(x)}\| \leq 2 \|A^{-1}_{\alpha^{-N}(x)} \dots A^{-1}_{\alpha^{-2}(x)}A^{-1}_{\alpha^{-1}(x)}|_{E_x^u}\| \leq 2C\lambda^N \leq \lambda_2^N.
$$
For arbitrary $n\geq0$, write $n=lN+r$, where $l\geq0$ and $0\leq r<N$. 
Splitting the product into blocks of length $N$ and using $\|Z_x\|\leq(1+\varepsilon_1)R$, we obtain \eqref{eq:Zx-exp}.
Item \textbf{(B2)} is proved.

In the following, we construct a map $H_x^s:E_x^s \to F_x$ with the following property: for any $v^s \in E^s_x$ there exists $w^s \in E_{\beta(x)}^s$ such that 
\begin{equation}\label{eq:Hsx-cond1}
A(x)(v^s+ H^s_xv^s) = w^s+ H^s_{\beta(x)}(w^s)
\end{equation}
For an arbitrary map $H_x^s$ we can represent $A(x)(v^s+ H^s_x(v^s)) = w^s + w^u$, where 
\begin{equation*}
    w^s = B^{ss}_x v^s + B^{su}_x H^s_xv^s \in E^s_{\beta(x)}, \quad
    w^u = B^{us}_x v^s + B^{uu}_x H^s_xv^s \in E^u_{\beta(x)}.
\end{equation*}
Condition \eqref{eq:Hsx-cond1} is equivalent to $w^u = H^s_{\beta(x)}w^s$, i.e.,
\begin{equation}\label{eq:H-via-B}
B^{uu}_x H^s_x = H_{\beta(x)}^s(B^{ss}_x + B^{su}_xH^s_x) - B^{us}_x.
\end{equation}

To satisfy \eqref{eq:H-via-B} it is sufficient to verify the following (compare with \eqref{eq:StHk})
\begin{equation}\label{eq:H-via-Z}
H_x^s = Z_x H^s_{\beta(x)} B_x^{ss} + (Z_x H^s_{\beta(x)}B^{su}_xH^s_x  - Z_xB^{us}_x).    
\end{equation}
As in Section \ref{sec:robust-Ak}, we consider the following auxiliary problem. Given a family of bounded linear operators $S_x:E^s_x \to F_x$ that depends uniformly continuously on $x$, we seek a family of bounded linear operators $H_x: E^s_x \to F_x$ that depends uniformly continuously on $x$, which satisfies 
$
H_x = Z_xH_{\beta(x)}B_x^{ss} + S_x
$ 
(compare to \eqref{eq:HkSk}).

Denote by $U$ the Banach space of uniformly bounded linear operators $S_x : E^s_x \to F_x$ that depend uniformly continuously on $x$.

Define an operator $F: U \to U$ by $F(S_x) = H_x$, where
\begin{equation}\notag
H_x = S_x + \sum_{l = 1}^{+\infty}Z_x\cdot \ldots \cdot Z_{\beta^{l-1}(x)}S_{\beta^{l}(x)}B_{\beta^{l-1}(x)}^{ss} \cdot \ldots \cdot B_x^{ss}.    
\end{equation}
Similarly to \eqref{eq:FL'} the following holds
\begin{equation}\label{eq:FL'cont}
\|F(S_x)\| \leq \|S_x\| \left(1+ \sum_{l=1}^{\infty} C_2^2\lambda_2^{2l} \right) = \|S_x\|\left(1+ C_2^2 \frac{\lambda_2^2}{1-\lambda_2^2}\right) = L'_2 \|S_x\|.        
\end{equation}

Since $\beta$ and $\beta^{-1}$ are uniformly continuous, the right-hand side depends uniformly continuously on $x$, and hence $H_x \in U$.

Define the nonlinear operator $G:U \to U$ by $G(H_x) = S_x$, where
$$
S_x = Z_x H^s_{\beta(x)}B^{su}_xH^s_x  - Z_x  B^{us}_x.
$$
Similarly to Section \ref{sec:robust-Ak}, for small enough $\delta > 0$, $\varepsilon_2 > 0$ and $\|H_x\|, \|H'_x\| \leq \varepsilon_2$ we have
$
\|GH_x-GH'_x\| \leq \frac{1}{2L_2'}\|H_x-H'_x\|.
$
Since $F$ is linear, inequality \eqref{eq:FL'cont} implies that
\begin{equation}\notag%\label{eq:FQcontrcont}
\|F \circ GH - F \circ GH'\| \leq L_2'\frac{1}{2L_2'}\|H-H'\| = \frac{1}{2}\|H-H'\|,    
\end{equation}
According to \eqref{eq:Asux}, \eqref{eq:AB-ep1} and \textbf{(B2)} we have inequalities $\|G(H_x = 0)\| = \|Z_xB^{us}_x\| \leq CR\varepsilon_1$ and 
$
\|F \circ G (H_x = 0)\|\leq L_2'CR\varepsilon_1
$.
Hence, there exists $H_x \in U$ with $\|H_x\| \leq 2L_2'CR\varepsilon_1$ such that 
$F\circ G (H_x) = H_x$, which implies \eqref{eq:H-via-Z}. Consequently, \eqref{eq:H-via-B} holds.

Proceeding as in \eqref{eq:B-proof}, by choosing $\varepsilon_1$ sufficiently small and, if necessary, decreasing $\delta$, we obtain generalized hyperbolicity estimates with exponent $\lambda_1$ for spaces $\tilde{E}_x^s$. 

\medskip

Now we carry out the corresponding construction for $\tilde{E}^u_x$. Here, the stable and unstable cases are genuinely asymmetric, more so than in Lemma~\ref{lem:CL-robust-Ak}: there the operators are perturbed while the index shift $k \to k+1$ is common to $\mathcal{A}$ and its perturbation, whereas here the cocycle $A$ is fixed and the base map is perturbed, $\alpha \to \beta$. Time-reversal $A \to A^{-1}$ therefore does not turn the perturbed system into a base-only perturbation of the reversed original (its fiber map is sampled at $\beta^{-1}(x)$, not $\alpha^{-1}(x)$), so $H^s_x$ and $H^u_x$ act between genuinely different spaces and cannot be related by symmetry; this is what forces the auxiliary splitting at $\alpha(\beta^{-1}(x))$ introduced below.

In order to define $H^u_{x}: E_{\alpha(x)}^u \to A(x)E_x^s$ we need to establish an analog of \eqref{eq:Hsx-cond1}, i.e. to show that for any $\tilde{v}^u \in \tilde{E}^u_{\beta(x)}$ there exists $\tilde{w}^u \in \tilde{E}^u_{x}$ such that
$
A(x)\tilde{w}^u = \tilde{v}^u
$,
equivalently, for any $v^u \in E^u_{\alpha(x)}$ there exists $w^u \in E^u_{\alpha(\beta^{-1}(x))}$ such that
\begin{equation}\label{eq:cond-Eu}
w^u + H_{\beta^{-1}(x)}^uw^u =  A^{-1}(x)(v^u + H_x^uv^u).
\end{equation}
Note that 
$$
H_{\beta^{-1}(x)}^uw^u \in A(\beta^{-1}(x))E^s_{\beta^{-1}(x)} \subset E^s_{\alpha(\beta^{-1}(x))},
\quad
H_{x}^uv^u \in A(x)E^s_{x} \subset E^s_{\alpha(x)}.
$$

For small enough $\delta$ we have the splitting $\mathbb{B} = E^s_{\alpha(\beta^{-1}(x))} \oplus E^u_{\alpha(\beta^{-1}(x))}$. Denote by $\tilde{P}_x$, $\tilde{Q}_x$ the projections corresponding to this splitting.
Similarly to the construction of $\tilde{E}^s_x$ define
\begin{align*}
{C}_x^{ss}: E_{\alpha(x)}^s \to \Bb, & \quad  C_x^{ss} = \tilde{P}_{x}A^{-1}(x); &\qquad
{C}_x^{su}: E_{\alpha(x)}^u \to \Bb, & \quad  C_x^{su} = \tilde{P}_{x}A^{-1}(x);\\
{C}_x^{us}: E_{\alpha(x)}^s \to \Bb, & \quad  C_x^{us} = \tilde{Q}_{x}A^{-1}(x); & \qquad
{C}_x^{uu}: E_{\alpha(x)}^u \to \Bb, & \quad  C_x^{uu} = \tilde{Q}_{x}A^{-1}(x).
\end{align*}
Note that since we considered the decomposition of $A^{-1}$, the operator $C_x^{ss}$ is ``expanding'' and the operator $C_x^{uu}$ is ``contracting''.
Equation \eqref{eq:cond-Eu} is equivalent to
$$
w^u = {C}^{uu}_x v^u + C^{us}_x H_{x}^uv^u, 
\quad
H_{\beta^{-1}(x)}^uw^u = {C}^{su}_x v^u + {C}^{ss}_x H_{x}^uv^u.
$$
Hence, the condition on $H^u$ is
\begin{equation}\label{eq:H-Eu-cond-1}
H_{\beta^{-1}(x)}^u \circ \left( {C}^{uu}_x  + C^{us}_x H_{x}^u \right) = {C}^{su}_x + {C}^{ss}_x H_{x}^u.    
\end{equation}
Similarly to \eqref{eq:Zx-cond} define the operator $\tilde{Z}_x:E^s_{\alpha(\beta^{-1}(x))} \to A(x)E^s_x$ such that $\tilde{Z}_x \circ C^{ss}_x = Id_{A(x)E^s_x}$. Condition \eqref{eq:H-Eu-cond-1} is equivalent to
\begin{equation}\label{eq:H-Eu-cond-2}
H_{x}^u = \tilde{Z}_x  H_{\beta^{-1}(x)}^u {C}^{uu}_x + \tilde{Z}_x  H_{\beta^{-1}(x)}^u  C^{us}_x H_{x}^u -\tilde{Z}_x C^{su}_x.
\end{equation}
Exactly as for \textbf{(B1)} and \eqref{eq:Zx-exp}, but using (C-GH3) for $A^{-1}$ in place of the estimates for $A$, for small enough $\varepsilon_1$ and some $C_2 > 0$ the contracting block $C^{uu}$ and the right inverse $\tilde{Z}_x$ satisfy
\begin{equation}\notag%\label{eq:Cuu-exp}
\|C^{uu}_{\beta^{-(n-1)}(x)} \cdots C^{uu}_{\beta^{-1}(x)}C^{uu}_x\| \leq C_2\lambda_2^n, \quad \|\tilde{Z}_x \tilde{Z}_{\beta^{-1}(x)} \cdots \tilde{Z}_{\beta^{-(n-1)}(x)}\| \leq C_2\lambda_2^n, \quad n \geq 0,\ x \in \Bb.
\end{equation}
Hence \eqref{eq:H-Eu-cond-2} is the analog of \eqref{eq:H-via-Z} with $\beta, B^{ss}_x, B^{su}_x, B^{us}_x, Z_x$ replaced by $\beta^{-1}, C^{uu}_x, C^{us}_x, C^{su}_x, \tilde{Z}_x$, and the fixed-point argument for $H^s_x$ applies verbatim: for any $\varepsilon > 0$ and small enough $\delta$ there is an operator $H^u_x$ satisfying \eqref{eq:H-Eu-cond-2} with $\|H^u_x\| < \varepsilon$.

The spaces $\tilde{E}^u_x$ satisfy the exponential estimates with exponent $\lambda_1$ as for $\tilde{E}^s_x$, via \eqref{eq:B-proof}.
As in the construction of $H^s_x$ we prove that $H^u_x$ depends uniformly continuously on~$x$.

Let us prove equality \eqref{eq:tildeE-EF}.
According to construction $\tilde{E}^s_x \subset E^s_x \oplus F_x$ and
\begin{align*}
    A^{-1}(x)\tilde{E}^u_{\beta(x)} & \subset A^{-1}(x)E^u_{\alpha(x)} \oplus A^{-1}(x)A(x)E^s_x = E^s_x \oplus F_x.
\end{align*}
Hence,
\begin{equation}\label{eq:tildeE-EF-1}
    \tilde{E}^s_x \oplus A^{-1}(x)\tilde{E}^u_{\beta(x)} \subset E^s_x \oplus F_x.
\end{equation}
For any $x \in \mathbb{B}$ and $v^s_x \in E^s_x$ define the vectors 
$$
v^u_x := H^s_xv^s_x \in F_x = A^{-1}(x)E^u_{\alpha(x)}, \quad v^u_{\alpha(x)} := A(x)v^u_x \in E^u_{\alpha(x)}, \quad v^s_{\alpha(x)} := H^u_x v^u_{\alpha(x)} \in A(x)E_x^s.
$$ 
Then $v^s_x + v^u_x \in \tilde{E}^s_x$ and  $v^u_{\alpha(x)} + v^s_{\alpha(x)} \in \tilde{E}^u_{\beta(x)}$. Hence,
$$
v^s_x - A^{-1}(x) H^u_x A(x) H^s_x v^s_x = v^s_x + v^u_x - A^{-1}(x)(v^u_{\alpha(x)} + v^s_{\alpha(x)}) \in \tilde{E}^s_x \oplus A^{-1}(x)\tilde{E}^u_{\beta(x)}.
$$
Note that for $\varepsilon < 1/R$ the following holds
$$
A^{-1}(x) H^u_x A(x) H^s_x (E^s_x) \subset E^s_x, \quad \|A^{-1}(x) H^u_x A(x) H^s_x\| < 1,
$$ 
hence $Id - A^{-1}(x) H^u_x A(x) H^s_x$ is bijective on $E^s_x$. In particular
$
E^s_x \subset \tilde{E}^s_x \oplus A^{-1}(x)\tilde{E}^u_{\beta(x)}
$.
Similarly, we show
$
F_x = A^{-1}(x)E^u_{\alpha(x)} \subset \tilde{E}^s_x \oplus A^{-1}(x)\tilde{E}^u_{\beta(x)}
$,
and hence 
$
E^s_x \oplus F_x \subset \tilde{E}^s_x \oplus A^{-1}(x)\tilde{E}^u_{\beta(x)}
$,
which together with inclusion \eqref{eq:tildeE-EF-1} implies \eqref{eq:tildeE-EF}.
\end{proof}

\begin{proof}[Proof of Theorem \ref{thm:SS}]
Assume that $f$ satisfies uniformly continuous generalized hyperbolicity with constants $C > 0$, $\lambda \in (0, 1)$. Then the pair $(f, Df)$ satisfies uniformly continuous cocycle generalized hyperbolicity.
Let $\varepsilon_0$ be as in Lemma~\ref{lem:SS2}. Choose $\varepsilon_1 \in (0,\varepsilon_0)$ and $\lambda_1\in(\lambda,1)$, and let $C_1>1$ and $\delta>0$ be the corresponding constants.
Let $R > 1$ be the constant in \eqref{eq:C1-2}, $L = L(C_1, \lambda_1)>0$ be the constant provided by Lemma \ref{lem:SS1}. Set $\varepsilon=\frac{1}{3L}$. Choose $d_0 < \min(\frac{\delta}{R}, \frac{1}{3L})$, such that \eqref{eq:s-Lip} holds with this $\varepsilon$ for $|v_1|, |v_2| < 2Ld_0$. 
Fix $d\in(0,d_0)$ and $g\in\DUC(\Bb)$ satisfying $\|g-f\|_{C^1}<d$.
Denote $\Delta = g - f$. 
Then
    \begin{equation}\label{eq:Delta-c1}
        |\Delta(x)| < d, \quad \|D\Delta(x)\| < d, \quad x \in \Bb. 
    \end{equation}
Moreover, for every $x\in\Bb$, setting $y=g^{-1}(x)$, we have
\[
|f^{-1}(x)-g^{-1}(x)|
=
|f^{-1}(g(y))-f^{-1}(f(y))|
\leq R|g(y)-f(y)|
<Rd<\delta.
\]
Together with $\sup_x|f(x)-g(x)|<d<\delta$, this verifies
\eqref{eq:al-be-ep}.
    
By Lemma \ref{lem:SS2}, the pair $(g, Df)$ satisfies uniformly continuous cocycle generalized hyperbolicity with constants~$C_1, \lambda_1$. 

Denote the splitting corresponding to the pair $(f, Df)$ by 
$$
\Bb = E^s_x \oplus E^u_x, \quad F_x = (Df(x))^{-1}E^u_{f(x)} \subset E^u_x
$$
and splitting provided by Lemma \ref{lem:SS2} corresponding to pair $(g, Df)$ by 
\begin{equation}\notag%\label{eq:def-tFx}
\Bb = \tilde{E}^s_x \oplus \tilde{E}^u_x, \quad \tilde{F}_x = (Df(x))^{-1}\tilde{E}^u_{g(x)} \subset \tilde{E}^u_x.
\end{equation}
\begin{remark}
We do not claim that the pair $(g, Dg)$ satisfies uniformly continuous cocycle generalized hyperbolicity, although we believe it to be true. 
\end{remark}
According to \eqref{eq:tildeE-EF} the following equality holds
$
\tilde{E}^s_x \oplus \tilde{F}_x = {E}^s_x \oplus {F}_x.
$
Hence
$$
\{v \in X: v(x) \in \tilde{E}_x^s \oplus \tilde{F}_x\} = Y.
$$

By Lemma \ref{lem:SS1} there exist operators $O_f, O_g: X \to Y$
satisfying 
$$
T_{f, Df}O_f = Id_X, \; \|O_f\| \leq L,
\quad
T_{g, Df}O_g = Id_X, \; \|O_g\| \leq L.
$$
Denote $M = 2L$.

\textbf{Step 1.} First, we find a map $h_1: \Bb \to \Bb$, $h_1 \in Y$, satisfying \eqref{eq:conj-h1}.
Substituting $g = f + \Delta$ in equation \eqref{eq:conj-h1} and using the definition of $s_f$ we get
    $$
    f(x) + Df(x)h_1(x) + s_f(x, h_1(x)) + \Delta(x+h_1(x)) = f(x) + h_1(f(x)), \quad x \in \Bb.
    $$
    Canceling $f(x)$ and regrouping terms we get
\begin{equation}\label{eq:conj-final}
        h_1(f(x)) = Df(x)h_1(x) + s_f(x, h_1(x)) + \Delta(x+h_1(x)), \quad x \in \Bb.    
\end{equation}
In order for \eqref{eq:conj-final} to hold, it suffices that
\begin{equation}
    h_1(x) = O_f(s_f(x, h_1(x)) + \Delta(x+h_1(x))).
\end{equation}    
Consider the (nonlinear) operator $G_f:BUC(\Bb, \Bb) \to BUC(\Bb, \Bb)$ defined by
$$
G_f(h_1)(x) = s_f(x, h_1(x)) + \Delta(x+h_1(x)), \quad x \in \Bb.
$$
Using \eqref{eq:s-Lip} and \eqref{eq:Delta-c1} we conclude that for $|h_1|, |h'_1| \leq 2L d$ the following inequalities hold 
\begin{multline*}
|(G_f(h_1)-G_f(h'_1))(x)| = |s_f(x, h_1(x)) - s_f(x, h'_1(x)) + \Delta(x+h_1(x)) - \Delta(x+h'_1(x))| \leq \\
(\varepsilon+d)|h_1(x) - h'_1(x)| \leq \frac{2}{3L}\|h_1 - h'_1\|.
\end{multline*}
Using \eqref{eq:Sf} we conclude that 
$
\|O_f \circ G_f(h_1) - O_f \circ G_f(h'_1)\| \leq \frac{2}{3}\|h_1-h'_1\|,
$
hence $O_f \circ G_f$ is a contraction. For $\|h_1\| \leq 2Ld$, equations \eqref{eq:s-Lip}, \eqref{eq:Delta-c1} imply $\|G_f(h_1)\| \leq \varepsilon 2Ld + d$, and
$$
\|O_f \circ G_f(h_1)\| \leq L(\varepsilon \cdot 2Ld + d) =  \left(\frac{2L}{3}+L\right)d \leq 2Ld.
$$
The latter implies existence and uniqueness of a map $h_1$ with $\|h_1\| \leq 2Ld$ satisfying $O_f \circ G_f (h_1) = h_1$,
and hence \eqref{eq:conj-h1}.  
Therefore $h_1$ satisfies \eqref{eq:h1h2leqMdelta}. Note that $\mbox{Im }O_f = Y$
and hence $h_1 \in Y$. 

\textbf{Step 2.} To construct a map $h_2 \in Y$, satisfying \eqref{eq:conj-h2} we follow an approach similar to that of Step 1. For completeness, we provide the full proof below. 
The conjugacy \eqref{eq:conj-h2} 
    is equivalent to
    $$
    f(x+h_2(x)) = f(x) + \Delta(x) + h_2(g(x)), \quad x \in \Bb.
    $$
    Using definition of $s_f$ and canceling $f(x)$ we get
\begin{equation}\label{eq:conj-final-g}
        h_2(g(x)) = Df(x)h_2(x) + s_f(x, h_2(x)) + \Delta(x), \quad x \in \Bb.    
\end{equation}
In order for \eqref{eq:conj-final-g} to hold, it suffices that
\begin{equation}
    h_2(x) = O_g(s_f(x, h_2(x)) + \Delta(x)).
\end{equation}    
Consider the (nonlinear) operator $G_g:BUC(\Bb, \Bb) \to BUC(\Bb, \Bb)$ defined by
$$
G_g(h_2)(x) = s_f(x, h_2(x)) + \Delta(x), \quad x \in \Bb.
$$
Using \eqref{eq:s-Lip} and \eqref{eq:Delta-c1} we conclude that for $|h_2|, |h'_2| \leq 2L d$ the following inequalities hold 
\begin{multline*}
|(G_g(h_2)-G_g(h'_2))(x)| = |s_f(x, h_2(x)) - s_f(x, h'_2(x))| \leq 
\varepsilon|h_2(x) - h'_2(x)| \leq \frac{1}{3L}\|h_2 - h'_2\|.   
\end{multline*}
Using \eqref{eq:Sf} we conclude that 
$
\|O_g \circ G_g(h_2) - O_g \circ G_g(h'_2)\| \leq \frac{1}{3}\|h_2-h'_2\|
$
and hence $O_g \circ G_g$ is a contraction. If $\|h_2\| \leq 2Ld$ then  \eqref{eq:s-Lip} and \eqref{eq:Delta-c1} imply $\|G_g(h_2)\| \leq \varepsilon 2Ld + d$, and 
$$
\|O_g \circ G_g(h_2)\| \leq L(\varepsilon \cdot 2Ld + d)  = \left(\frac{2L}{3}+L\right)d \leq 2Ld.
$$
The latter implies existence and uniqueness of a map $h_2$ with $\|h_2\| \leq 2Ld$ satisfying
$
O_g \circ G_g (h_2) = h_2,
$
and hence \eqref{eq:conj-final-g}. Therefore, $h_2$ satisfies \eqref{eq:conj-h2} and \eqref{eq:h1h2leqMdelta}.  
Note that $\mbox{Im }O_g = Y$, hence $h_2 \in Y$. This proves Theorem~\ref{thm:SS}\,(\ref{item:SS-semi}).

\textbf{Step 3.} Under the additional assumption in Theorem~\ref{thm:SS}\,(\ref{item:SS-full}), we prove \eqref{eq:Idh-homeo}.
From \eqref{eq:conj-h1} and \eqref{eq:conj-h2} we obtain the following.
\begin{equation}\label{eq:selff}
f \circ (Id+h_2) \circ (Id+h_1) = (Id+h_2) \circ (Id+h_1) \circ f,   
\end{equation}
\begin{equation}\label{eq:selfg}
g \circ (Id+h_1) \circ (Id+h_2) = (Id+h_1) \circ (Id+h_2) \circ g.
\end{equation}
Since $h_1, h_2 \in Y$ and the spaces $E^s_x \oplus (Df(x))^{-1}E^u_{f(x)} = E^s_x \oplus F_x$ do not depend on $x$, for some functions $h_3, h_4 \in Y$ the following hold
$$
(Id+h_2) \circ (Id+h_1) = Id + h_3, \quad  (Id+h_1) \circ (Id+h_2) = Id + h_4. 
$$
\begin{equation}\label{eq:h3h4-cond}
h_3(x), h_4(x) \in E^s_x \oplus F_x, \quad |h_3(x)|, |h_4(x)| \leq 2Md, \quad x \in \Bb.
\end{equation}
By \eqref{eq:selff} the identity   
$
f\circ (Id+h_3) = (Id+h_3) \circ f
$ holds.
Arguing as in Step~1, we conclude that
$$
h_3(f(x)) - Df(x)h_3(x) = s_f(x, h_3(x)).
$$
Since $h_3 \in Y$, Lemma \ref{lem:SS1} implies that 
\begin{equation}\label{eq:h3-fixed}
h_3 = O_f(s_f(x, h_3(x))).
\end{equation}
Similarly to the operator $O_f \circ G_f$, the map $h_3 \to O_f(s_f(x, h_3(x)))$ is a contraction (decreasing $d$ if necessary) on the set of $h_3$ with $\|h_3\| \leq 2Md$. Therefore, equation \eqref{eq:h3-fixed} has a unique solution for $h_3$ satisfying \eqref{eq:h3h4-cond}. On the other hand, $h_3 \equiv 0$ satisfies both \eqref{eq:h3h4-cond} and \eqref{eq:h3-fixed}. By uniqueness $h_3 \equiv 0$.

The argument for $h_4$ is analogous. Equation \eqref{eq:selfg} implies that $g \circ (Id+h_4) = (Id+h_4) \circ g$.
Arguing as in Step~2, we obtain
$$
h_4(g(x)) - Df(x)h_4(x) = s_f(x, h_4(x)) + \left( \Delta(x+h_4(x)) - \Delta(x) \right).
$$

As in the proof for $h_3$, this equation has a unique solution $h_4$, satisfying \eqref{eq:h3h4-cond}. Hence, $h_4 \equiv 0$. This completes the proof of \eqref{eq:Idh-homeo} and Theorem \ref{thm:SS}\,(\ref{item:SS-full}).
\end{proof}
\begin{remark}
Note that in Step 3 we strongly used that $E^s_x\oplus F_x = \tilde{E}^s_x\oplus \tilde{F}_x$ and hence $h_1$ and $h_2$ are unique in the same space $Y$.  This required special attention in Lemma \ref{lem:SS2} to establish equality \eqref{eq:tildeE-EF}. The additional condition in Theorem~\ref{thm:SS}\,(\ref{item:SS-full}) was used only to establish the inclusions~\eqref{eq:h3h4-cond}.
\end{remark}

\section{Discussion and open questions}\label{sec:discussion}

\textbf{Structural stability.} The main open problem is whether
generalized hyperbolicity alone implies structural stability, as proposed in Conjecture~\ref{conj:SS-general}. 
The proof of Theorem~\ref{thm:SS} relies on a bounded right inverse for the operator $T_{\alpha,A}$ defined in \eqref{eq:Twv}. To address Conjecture~\ref{conj:SS-general}, it is natural to seek such an inverse on the space of bounded continuous functions; linearity is not essential, and a possibly nonlinear Lipschitz right inverse would suffice.
\begin{conj}\label{conj:inv-cocycle}
Let the pair $(\alpha,A)$ satisfy cocycle generalized hyperbolicity
with constants $C,\lambda$, with no continuity assumption on the
projections $P_x,Q_x$. Let
$
X=C_b(\Bb,\Bb)
$
be the Banach space of bounded continuous functions with the supremum
norm, and let $T_{\alpha,A}:X\to X$ be defined by \eqref{eq:Twv}.
Then there exist a constant $L=L(\lambda,C)$ and a not necessarily
linear map $S_{\alpha,A}:X\to X$ such that
$
T_{\alpha,A}\circ S_{\alpha,A}=Id_X,
$
and $S_{\alpha,A}$ is Lipschitz with constant $L$.
\end{conj}
When the projections $P_x,Q_x$ are discontinuous, the expression \eqref{eq:Svdef} need not define a continuous function. However, we believe that the argument can be modified by using the continuity of $A(x)$. In this paper, we overcame a similar difficulty in the proof of Lipschitz shadowing.

\smallskip

\textbf{Exponential dichotomy, transversality, and generalized hyperbolicity.} For finite-dimensional sequences of linear isomorphisms,
generalized hyperbolicity coincides with the existence of exponential
dichotomies on $\Zz^+$ and $\Zz^-$ together with the transversality
condition \cite{PilCL,Pli77}. In infinite dimensions, the implication from exponential dichotomies on $\Zz^\pm$ together with transversality to generalized hyperbolicity is considerably more subtle, and we single it out.

\begin{quest}\label{quest:ED}
    When do exponential dichotomies on $\Zz^{\pm}$ together with the transversality condition imply generalized hyperbolicity?
\end{quest}

We expect that in a Hilbert space Pilyugin's construction \cite{PilCL} can be adapted by gluing the two half-line splittings at $k=0$ and replacing the finite-dimensional angle estimate by an open-mapping argument. In a general Banach space, such a construction would require a bounded projection onto
$
D=E^{s,+}_0\cap E^{u,-}_0,
$
which may fail to exist even when the two intersecting subspaces are complemented. This is a continuous-selection obstruction related to the Bartle--Graves and Michael selection theorems
\cite{BartleGraves1952,Michael1956I,Michael1956II}.

\smallskip

\textbf{Converse implication.} 
On compact finite-dimensional manifolds, Lipschitz shadowing implies generalized hyperbolicity by \cite{TikhLipSh} and Theorem~\ref{stat:finite-dimensional-characterization}. We ask whether the same implication holds in Banach spaces. Following \cite{TikhLipSh,TikhHolSh}, we expect Lipschitz shadowing to imply the bounded solution property, with a uniform constant, for the derivatives along trajectories. 
The main difficulty is the converse linear implication: whether the bounded solution property implies generalized hyperbolicity.

A natural first step is to assume that $\Bb$ is a Hilbert space. 
Let $\{A_k:\Bb\to\Bb\}$ be a sequence of bounded linear isomorphisms satisfying the bounded solution property with constant $L$. 
Following a Gram--Schmidt-type construction, consider subspaces $E_k\subset\Bb$ and operators $B_k:E_k\to E_{k+1}$ defined by
$$
E_{k+1} = \left(A_k(E_k^{\perp})\right)^{\perp}, \quad B_k = P_{E_{k+1}}A_k.
$$
The finite-dimensional argument of \cite[Lemma~4.2, Theorem~2.3]{TikhHolSh} yields generalized hyperbolicity, but the resulting constant $C$ tends to infinity and $\lambda$ tends to $1$ as the dimensions of the spaces $E_k$ increase. 
The main challenge is to obtain dimension-free estimates that extend to an infinite-dimensional Hilbert space.

In general Banach spaces there is an additional difficulty: it is not known whether the bounded and strong bounded solution properties are equivalent. 
As in Question~\ref{quest:ED}, this is a continuous-selection problem related to the Michael and Bartle--Graves theorems \cite{BartleGraves1952,Michael1956I,Michael1956II}. 
For Lipschitz selection results, see \cite{BorweinDontchev2003,Messerschmidt2019}.

\smallskip

\textbf{Flows and semiflows in Banach spaces.}
An analog of generalized hyperbolicity can be defined for flows on Banach spaces, although we do not develop it here. One difference from diffeomorphisms is that the relevant dimensions may change at fixed points and closed trajectories; the absence of continuity assumptions on $P_x,Q_x$ allows such changes to be incorporated. A closely related structure appears in~\cite{ACK2025}: for Morse--Smale semigroups on a Hilbert space, compatible subbundles over the global attractor yield Lipschitz shadowing on the attractor, H\"older shadowing nearby, and structural stability of the attractor. See also~\cite{BackesDragicevic2021}.

In finite dimensions, the relations among shadowing, structural stability, and hyperbolicity for flows are well developed
\cite{Pil97,Robinson1974,Hayashi1997,GanWen2006,Tikh2008,PalmPilTikh2012,GanLiTikh2016,Li2020,WenBook2016}.
Flows also exhibit phenomena absent for diffeomorphisms
\cite{GuckenheimerWilliams1979,MoralesPacificoPujals2004,PilTikh2008,PilTikh2010,Tikh2015,Murakami2023};
see \cite{FisherHasselblatt} for an overview.
For infinite-dimensional flows, shadowing is more subtle
\cite{Henry94,CLP1989}.
Stronger results are available in the presence of inertial manifolds
\cite{SellYou,FST1988,Ragazzo1994,Robinson2001,Llave2009},
while additional difficulties arise without them
\cite{ACH1992,CLR2012,ACK2025};
see also the shadowing-based approach to a posteriori estimates in
\cite{ApushRepinTikh}.

Semilinear parabolic PDEs and delay differential equations are commonly described by nonlinear semiflows $\varphi_t$ on $\Bb$
\cite{Pazy1983,Lunardi1995,Henry1981,Diekmann-vGils-VerduynLunel-Walther1995,Hale-VerduynLunel1993}.
Their images often belong to nested Banach spaces $\mathcal F_t\subset\Bb$; for example,
\cite{HochsRoberts2019},
\[
\varphi_{t_1}(\Bb)\subset\mathcal F_{t_1},
\qquad
\mathcal F_{t_2}\subset\mathcal F_{t_1},
\qquad t_2>t_1>0.
\]
These inclusions are compatible in spirit with {\rm(GH2)}, suggesting a possible extension of generalized hyperbolicity to non-reversible semiflows. A major additional difficulty is that the generator is typically densely defined and unbounded, particularly for arguments using~\eqref{eq:C1-r1}.

\section*{Appendix A}\label{app:BP-constant}

In this appendix, we give a self-contained proof of Theorem~\ref{thm:Bernandes-shadowing}, keeping track of the shadowing constant used in Lemma~\ref{lem:Bk-Inf}. The argument adapts the proof of \cite[Theorem 1]{NilsonPeris}: the key idea of gluing finitely-shadowing points by linear interpolation is theirs; the only additions are the bookkeeping of constants and the organization of the argument around a single error-halving step, which is then iterated. 

\begin{lem}\label{lem:halving}
Assume that $T$ has the finite Lipschitz shadowing property with constant $L$. Then for every $d>0$ and every $d$-pseudotrajectory $\{x_j\}_{j \in \Zz}$ of $T$ there exists a $\tfrac{d}{2}$-pseudotrajectory $\{x'_j\}_{j \in \Zz}$ of $T$ such that
$$
\|x_j - x'_j\| \le Ld, \qquad j \in \Zz.
$$
\end{lem}

\begin{proof}
Fix an integer $m \ge 4L$ and set
$
n_k := \frac{k(k-1)}{2}m
$,
so that $n_0 = n_1 = 0$ and $n_{k+1} - n_k = km$ for $k \ge 1$; the intervals $[n_k, n_{k+1})$, $k \ge 1$, partition $[0, +\infty) \cap \Zz$. For each $k \ge 1$ the restriction $\{x_j\}_{|j| \le n_{k+1}}$ is a finite $d$-pseudotrajectory, so by the finite Lipschitz shadowing property there is $u_k \in X$ with
\begin{equation}\label{eq:BP-uk}
\|x_j - T^j u_k\| \le Ld \qquad (|j| \le n_{k+1}).
\end{equation}
We interpolate linearly between consecutive shadowing points. For $j \ge 0$ write $j = n_k + i$ with $k \ge 1$ and $0 \le i < km$, and put
$$
x'_j := \Big(1 - \tfrac{i}{km}\Big) T^j u_k + \tfrac{i}{km}\, T^j u_{k+1};
$$
for $j < 0$ write $-j = n_k + i$ in the same way and put $x'_j := \big(1 - \tfrac{i}{km}\big) T^j u_k + \tfrac{i}{km}\, T^j u_{k+1}$.

\textbf{Displacement.} Let $j \ge 0$, $j = n_k + i$. Then $j \le n_{k+1} \le n_{k+2}$, so \eqref{eq:BP-uk} holds at $j$ both for $u_k$ and for $u_{k+1}$; as $x'_j$ is a convex combination of $T^j u_k$ and $T^j u_{k+1}$,
$$
\|x_j - x'_j\| \le \Big(1-\tfrac{i}{km}\Big)\|x_j - T^j u_k\| + \tfrac{i}{km}\,\|x_j - T^j u_{k+1}\| \le Ld.
$$
The case $j<0$ is identical, so $\|x_j - x'_j\| \le Ld$ for all $j \in \Zz$.

\textbf{Halved error.} Let $j \ge 0$, $j = n_k + i$ with $0 \le i < km$. If $i+1 \le km-1$, then $j+1$ lies in the same block and a direct computation gives
$$
Tx'_j - x'_{j+1} = \tfrac{1}{km}\big(T^{j+1}u_k - T^{j+1}u_{k+1}\big).
$$
If $i+1 = km$, then $j+1 = n_{k+1}$ is the first index of the next block, where $x'_{j+1} = T^{j+1}u_{k+1}$, and the same identity follows since $x'_j = \tfrac{1}{km}T^{j}u_k + \big(1-\tfrac{1}{km}\big)T^{j}u_{k+1}$ in this case. In either case $j+1 \le n_{k+1} \le n_{k+2}$, so \eqref{eq:BP-uk} applies at $j+1$ for both $u_k$ and $u_{k+1}$, and therefore
$$
\|Tx'_j - x'_{j+1}\| \le \tfrac{1}{km}\Big(\|T^{j+1}u_k - x_{j+1}\| + \|x_{j+1} - T^{j+1}u_{k+1}\|\Big) \le \frac{2Ld}{km} \le \frac{2Ld}{m} \le \frac{d}{2},
$$
where we used $k \ge 1$ and $m \ge 4L$. The estimate for $j < 0$, including the junction between $j=-1$ and $j=0$, is analogous. Hence, $\{x'_j\}$ is a $\tfrac{d}{2}$-pseudotrajectory.
\end{proof}

\begin{proof}[Proof of Theorem \ref{thm:Bernandes-shadowing}]
Let $\{x_j\}_{j \in \Zz}$ be a $d$-pseudotrajectory of $T$. Applying Lemma~\ref{lem:halving} repeatedly, we construct pseudotrajectories $\{x^{(l)}_j\}_{j \in \Zz}$, $l \ge 0$, with $x^{(0)}_j = x_j$ and, for every $l \ge 0$,
$$
\{x^{(l+1)}_j\}\ \text{is a}\ \tfrac{d}{2^{l+1}}\text{-pseudotrajectory}, \qquad \|x^{(l+1)}_j - x^{(l)}_j\| \le L\,\frac{d}{2^{l}} \quad (j \in \Zz),
$$
where the displacement bound uses that $\{x^{(l)}_j\}$ is a $\tfrac{d}{2^{l}}$-pseudotrajectory.

Fix $j \in \Zz$. The bound $\|x^{(l+1)}_j - x^{(l)}_j\| \le Ld\,2^{-l}$ shows that $\big(x^{(l)}_j\big)_{l \ge 0}$ is a Cauchy sequence in $X$; let $v_j := \lim_{l \to \infty} x^{(l)}_j$. Letting $l \to \infty$ in $\|Tx^{(l)}_j - x^{(l)}_{j+1}\| \le d\,2^{-l}$ yields $Tv_j = v_{j+1}$, so $\{v_j\}$ is an exact trajectory, that is, $v_j = T^j v$ with $v := v_0$. Finally,
$$
\|x_j - T^j v\| = \|x^{(0)}_j - v_j\| \le \sum_{l \ge 0} \|x^{(l)}_j - x^{(l+1)}_j\| \le \sum_{l \ge 0} L\,\frac{d}{2^{l}} = 2Ld.
$$
As $j \in \Zz$ was arbitrary, $T$ has the Lipschitz shadowing property with constant $2L$.
\end{proof}

\section*{Appendix B}\label{app:finite-dimensional}

In this appendix, we explain the meaning of Definition~\ref{def:CLf} on manifolds. Let $M$ be a finite-dimensional smooth compact manifold equipped with a Riemannian metric, and let $f:M\to M$ be a $C^1$-diffeomorphism.
We say that $f$ satisfies generalized hyperbolicity if the tangent-bundle analogue of Definition~\ref{def:CLf} holds: there exist $C>0$, $\lambda\in(0,1)$ and projections
\[
P_x,Q_x:T_xM\to T_xM,
\qquad
E_x^s=P_x(T_xM),\quad E_x^u=Q_x(T_xM),
\]
satisfying conditions {\rm(GH1)}--{\rm(GH3)}, with $\Bb$ replaced by
$T_xM$. As in Definition~\ref{def:CLf}, no continuity of the projections is assumed.

We now prepare the proof of Theorem~\ref{stat:finite-dimensional-characterization}.
For $p\in M$, define
\[
B^+(p)=
\left\{
v\in T_pM:
|Df^n(p)v|\to0\ \text{as }n\to+\infty
\right\},
\]
\[
B^-(p)=
\left\{
v\in T_pM:
|Df^n(p)v|\to0\ \text{as }n\to-\infty
\right\}.
\]

We use the following infinitesimal criterion of Mañé \cite{Mane1977}.
%, in the formulation quoted in \cite[Proposition~4]{TikhLipSh}.

\begin{thm}[Mañé's infinitesimal criterion]
\label{thm:app-Mane}
A $C^1$-diffeomorphism $f:M\to M$ is structurally stable if and only if
\begin{equation}
\label{eq:app-Mane}
B^+(p)+B^-(p)=T_pM
\qquad\text{for every }p\in M.
\end{equation}
\end{thm}

We also use the $C^1$ stability theorem: every structurally stable $C^1$-diffeomorphism satisfies Axiom~A and the strong transversality condition \cite{Mane1987}.

For the implication from Axiom~A and strong transversality, we need the following uniform form of the Robbin--Robinson construction, recorded immediately after Definition~3 in \cite[p.~277]{PilCL}.

\begin{statement}[Uniform orbitwise $(C,\lambda)$-structure] \label{stat:uniform-CL-finite}
Assume that $f$ satisfies Axiom~A and the strong transversality condition.
Then there exist constants $C_0>0$ and $\lambda_0\in(0,1)$ such that, for every $p\in M$, the derivative sequence
$
\left\{
Df(f^k(p)):T_{f^k(p)}M\to T_{f^{k+1}(p)}M
\right\}_{k\in\Zz}
$
has a $(C_0,\lambda_0)$-structure in the sense of \cite{PilCL}.
\end{statement}

The constants in this statement are independent of the trajectory.
We recall the ingredients of the Robbin--Robinson construction responsible for this uniformity.

\noindent\textbf{Adapted metric.}
Let $|\cdot|_0$ denote the norm of the Riemannian metric fixed above.
Since $\Omega(f)$ is hyperbolic, the standard adapted-metric construction
gives a continuous Riemannian norm $|\cdot|_*$ and a number
$\theta_0\in(0,1)$ such that
\[
|Df(x)v^s|_*\leq\theta_0|v^s|_*,
\qquad
|D(f^{-1})(x)v^u|_*\leq\theta_0|v^u|_*
\]
for $x\in\Omega(f)$, $v^s\in E_{\Omega,x}^s$, and
$v^u\in E_{\Omega,x}^u$. Extend this metric to $M$. Robinson's
compatible-bundle construction then provides neighborhoods of the basic sets and a number $\theta\in(\theta_0,1)$ for which the one-step estimates
in Proposition~\ref{stat:Robinson-bundles} hold. This is the adapted
metric used in Robinson's construction; see the proof of
Theorem~5.1 in \cite{Robinson1976}.
Since $M$ is compact, the two norms are uniformly equivalent: there is
$\kappa\geq1$ such that
\begin{equation}
\label{eq:app-equivalent-metrics}
\kappa^{-1}|v|_0\leq|v|_*\leq\kappa|v|_0,
\qquad v\in TM.
\end{equation}

\noindent\textbf{Compatible bundles.}
Let
$
\Omega(f)=\Omega_1\cup\cdots\cup\Omega_m
$
be the spectral decomposition. In the notation of this paper,
Proposition~7.1 of \cite[p.~66]{Robinson1976} gives the following
statement with respect to the adapted metric.

\begin{statement}[Robinson, Proposition~7.1]\label{stat:Robinson-bundles}
There exist neighborhoods $Z_i$ of $\Omega_i$, continuous subbundles $E_i^u,E_i^s$ of $TM|_{\mathcal O(Z_i)}$, $i=1,\ldots,m$, and $\theta\in(0,1)$ such that:
\begin{itemize}
\item[(1)] $E_i^u$ and $E_i^s$ are invariant under $Df$;
\item[(2)] if $x\in\mathcal O^+(Z_j)\cap\mathcal O^-(Z_i)$, then
$
E_{i,x}^u\subseteq E_{j,x}^u$, $E_{i,x}^s\supseteq E_{j,x}^s;
$
\item[(3)]
$
E_{i,x}^u \oplus E_{i,x}^s=T_xM$, for any $x\in\mathcal O(Z_i)
$;
\item[(4)] if $x\in Z_i$, $v^s\in E_{i,x}^s$, and $v^u\in E_{i,x}^u$, then
$$
|Df(x)v^s|_*\leq\theta|v^s|_*, \qquad
|D(f^{-1})(x)v^u|_*\leq\theta|v^u|_*.
$$
\end{itemize}
Here
$
\mathcal O^+(Z)=\bigcup_{n\geq0}f^n(Z)$, 
$\mathcal O^-(Z)=\bigcup_{n\leq0}f^n(Z)$, 
$\mathcal O(Z)=\mathcal O^+(Z)\cup\mathcal O^-(Z)
$.
The one-step estimates in item~(4) are asserted for $x\in Z_i$, not for all $x\in\mathcal O(Z_i)$.
\end{statement}

The following uniform consequences are stated immediately after the proof of Proposition~7.1 in \cite[p.~66]{Robinson1976}. We label them (5) and (6) for use below; they are unnumbered in Robinson's paper.

\begin{statement}[Uniform estimates following Proposition~7.1]\label{stat:Robinson-uniform}
Let $\{\gamma_i\}_{i=1}^m$ be a $C^{\infty}$ partition of unity subordinate to the finite open cover $\{\mathcal O(Z_i)\}_{i=1}^m$. There exists an integer $q\geq0$ such that, for every $i$ and every $x\in\supp\gamma_i$,
\begin{itemize}
\item the forward orbit of $x$ lies in $\bigcup_{j\geq i}Z_j$ except at most $q$ iterates;
\item the backward orbit of $x$ lies in $\bigcup_{j\leq i}Z_j$ except at most $q$ iterates.
\end{itemize}
Set
$
b_*:=\sup_{x\in M}\max\left\{\|Df(x)\|_*,\|D(f^{-1})(x)\|_*\right\}$, 
$C_*:=\frac{b_*^q}{\theta^q}.
$
Then:
\begin{itemize}
\item[(5)] for every $x\in\supp\gamma_i$, $n\geq0$, $v^s\in E_{i,x}^s$, and $v^u\in E_{i,x}^u$,
\begin{equation}\label{eq:app-Robinson-stable}
|Df^n(x)v^s|_*\leq C_*\theta^n|v^s|_*, 
\qquad 
|D(f^{-n})(x)v^u|_*\leq C_*\theta^n|v^u|_*;
\end{equation}
\item[(6)] the stable and unstable bundles are uniformly transverse on
$\operatorname{supp}\gamma_i$. Consequently, there exists
$C_{\mathrm{proj},*}>0$, independent of $i$ and $x$, such that the
projections
\[
\pi_{i,x}^s:T_xM\to E_{i,x}^s
\quad\text{along }E_{i,x}^u,
\qquad
\pi_{i,x}^u:T_xM\to E_{i,x}^u
\quad\text{along }E_{i,x}^s
\]
satisfy
$
\|\pi_{i,x}^s\|_*,
\|\pi_{i,x}^u\|_*
\leq C_{\mathrm{proj},*}$, for
$x\in\operatorname{supp}\gamma_i.
$
\end{itemize}
\end{statement}

Indeed, compatibility allows item~(4) of Proposition~\ref{stat:Robinson-bundles} to be applied at every nonexceptional iterate. At each of the at most $q$ exceptional iterates, the norm increases by at most $b_*$. Hence
$
b_*^q\theta^{n-q}=C_*\theta^n,
$
which gives item~(5). Item~(6) follows from continuity of the bundles, compactness of $\supp\gamma_i$, and finiteness of the family.

\noindent\textbf{Return to the original metric.}
By \eqref{eq:app-equivalent-metrics}, items~(5)--(6) remain valid in the
original metric after replacing their constants by
$
C_1:=\kappa^2\max\{C_*,C_{\mathrm{proj},*}\}.
$
Thus, for $x\in\supp\gamma_i$, $n\geq0$,
$v^s\in E_{i,x}^s$, and $v^u\in E_{i,x}^u$,
\begin{equation}
\label{eq:app-original-estimates}
|Df^n(x)v^s|_0\leq C_1\theta^n|v^s|_0,
\quad
|D(f^{-n})(x)v^u|_0\leq C_1\theta^n|v^u|_0, \quad \|\pi_{i,x}^s\|_0,\|\pi_{i,x}^u\|_0\leq C_1.
\end{equation}

Consequently, all constants in the compatible-bundle construction are
uniform in the original metric. Together with the orbitwise gluing in the
Robbin--Robinson construction, this yields
Proposition~\ref{stat:uniform-CL-finite}; see also
\cite[p.~277, after Definition~3]{PilCL}.

\begin{proof}[Proof of Theorem~\ref{stat:finite-dimensional-characterization}]
Assume first that $f$ satisfies Axiom~A and the strong transversality
condition. By Proposition~\ref{stat:uniform-CL-finite}, every derivative
sequence admits a $(C_0,\lambda_0)$-structure, with constants independent
of the trajectory.

On $\Omega(f)$, use the ordinary hyperbolic splitting
$
TM|_{\Omega(f)}=E_\Omega^s\oplus E_\Omega^u,
$
whose exponential estimates and projection bounds are uniform; denote
corresponding constants by $C_\Omega>0$ and
$\lambda_\Omega\in(0,1)$.

The invariant set $M\setminus\Omega(f)$ contains no periodic points.
Choose one representative $p$ from each orbit in this set and write a
chosen $(C_0,\lambda_0)$-structure along its derivative sequence as
\[
T_{f^k(p)}M=E_k^s\oplus E_k^u,\qquad k\in\Zz.
\]
Define
$
E_{f^k(p)}^s:=E_k^s$, $E_{f^k(p)}^u:=E_k^u.
$
These definitions are unambiguous because the orbit is nonperiodic. On $\Omega(f)$, set
$
E_x^s=E_{\Omega,x}^s$,
$
E_x^u=E_{\Omega,x}^u.
$
With
$
C=\max\{C_0,C_\Omega\},
$
$
\lambda=\max\{\lambda_0,\lambda_\Omega\},
$
the associated projections satisfy {\rm(GH1)}--{\rm(GH3)} on $M$.
No continuity of the projections is required.

Conversely, assume that $f$ satisfies generalized hyperbolicity.
Condition {\rm(GH3)} gives
$
E_p^s\subseteq B^+(p)$,
$E_p^u\subseteq B^-(p).
$
Therefore,
\[
T_pM=E_p^s\oplus E_p^u
\subseteq B^+(p)+B^-(p)
\subseteq T_pM,
\]
and hence \eqref{eq:app-Mane} holds. By
Theorem~\ref{thm:app-Mane}, $f$ is structurally stable. The $C^1$
stability theorem \cite{Mane1987} then implies Axiom~A and the strong
transversality condition.
\end{proof}

\section*{Appnedix C}
\begin{proof}[Proof of Proposition \ref{stat:difconj}]
{\rm(P1)} For $f_1, f_2 \in \DUC(\Bb)$ the chain rule gives $\|D(f_1 \circ f_2)(x)\| \leq R_1R_2$ and
$
\|D(f_1 \circ f_2)(x+v) - D(f_1 \circ f_2)(x)\| \leq R_1 r_2(|v|) + R_2 r_1(R_2|v|),
$
so $f_1 \circ f_2 \in \DUC(\Bb)$; in particular $g \in \DUC(\Bb)$. The spaces
$\tilde E^{s,u}_x = Dh(h^{-1}(x))E^{s,u}_{h^{-1}(x)}$ satisfy {\rm(GH1)}--{\rm(GH3)} for $g$
with constants $R_h^2C$ and $\lambda$.

{\rm(P2)} The two inclusions in {\rm(GH2)} and the two estimates in {\rm(GH3)} interchange
under $f \to f^{-1}$, $E^s_x \leftrightarrow E^u_x$.

{\rm(P3)} Conditions {\rm(GH1)}, {\rm(GH2)} are preserved by iteration, and
$|Df^{nm}(x)v^s| \leq C\lambda^{nm}|v^s| = C(\lambda^n)^m|v^s|$.

{\rm(P4)} Since $f_i(0) = 0$ and $\|Df_i\| \leq R$, we have $|f_i(x_i)| \leq R|x_i|$, so
$f$ maps $\Bb$ into $\Bb$; the same applies to $f^{-1}$. Since $r$ is nondecreasing and
$|v_i| \leq \|v\|$,
$$
\bigl\|\{f_i(x_i+v_i) - f_i(x_i) - Df_i(x_i)v_i\}\bigr\| \leq r(\|v\|)\|v\|,
$$
$$
\|Df(x+v) - Df(x)\| = \sup_i \|Df_i(x_i+v_i) - Df_i(x_i)\| \leq r(\|v\|),
$$
hence $f \in \DUC(\Bb)$ and $Df(x) = \bigoplus_i Df_i(x_i)$. The projections
$P_x = \bigoplus_i P^{(i)}_{x_i}$ and $Q_x = \bigoplus_i Q^{(i)}_{x_i}$ are diagonal, so
$\|P_x\| = \sup_i \|P^{(i)}_{x_i}\| \leq C$, and {\rm(GH2)}, {\rm(GH3)} hold coordinatewise.
\end{proof}

\section*{Acknowledgement}
The author is grateful to E. Pujals and N. Bernardes for an introduction to the topic of generalized hyperbolicity; to the seminars of IMPA and EDAI for inspiration that led us to extend shadowing results to robustness and structural stability; to A.~Baranov, T.~Pereira, F. Petrov, S. Pilyugin, D. Smania for stimulating discussions; and to A. Gorodetski for remarks on a preliminary version of the manuscript.

The work was supported by Projeto Paz,  Coordenação de Aperfeiçoamento de Pessoal de Nível Superior - Brasil (CAPES) 23038.015548/2016-06,  FAPERJ APQ1 E-26/210.702/2024 (295291), National Council for Scientific and Technological Development – CNPq, Grants 404123/2023-6, 403673/2025-9 and Bolsa de Produtividade em Pesquisa do CNPq - Nível C.

\section*{Declaration of competing interest}
The author declares that he has no known competing financial interests or personal relationships that could have appeared to influence the work reported in this paper.

\section*{Declaration of generative AI and AI-assisted technologies
in the manuscript preparation process}
During the preparation of this work, the author used ChatGPT (OpenAI) and Claude (Antropic) to assist with language editing and readability. The author reviewed and edited all suggestions and takes full responsibility for the content of the article.

\end{document}